\documentclass[11pt,reqno,oneside]{amsart}
\usepackage{amsmath,amssymb,amsthm}
\usepackage{hyperref}
\hypersetup{bookmarksdepth=3}
\usepackage{xcolor}
\usepackage[a4paper]{geometry}
\usepackage{enumitem}
\usepackage{graphicx}
\usepackage{subcaption}

\newcommand{\N}{\mathbb{N}}
\newcommand{\R}{\mathbb{R}}
\newcommand{\C}{\mathbb{C}}
\newcommand{\T}{\mathbb{T}}
\newcommand{\interface}{\mathcal{T}}
\newcommand{\eps}{\varepsilon}
\newcommand{\np}{\mathrm{np}}
\newcommand{\loc}{\mathrm{loc}}

\renewcommand{\Re}{\operatorname{Re}}
\renewcommand{\Im}{\operatorname{Im}}

\DeclareMathOperator{\Range}{\mathrm{ran}}
\DeclareMathOperator{\Ind}{\mathrm{Ind}}

\DeclareMathOperator{\Span}{\mathrm{span}}
\DeclareMathOperator{\diag}{\mathrm{diag}}

\newtheorem{thm}{Theorem}[section]
\newtheorem{lem}[thm]{Lemma}
\newtheorem{cor}[thm]{Corollary}
\newtheorem{prop}[thm]{Proposition}

\theoremstyle{remark}
\newtheorem{rk}[thm]{Remark}
\theoremstyle{definition}

\numberwithin{equation}{section}

\usepackage[normalem]{ulem}

\title[Waves with two layers of constant vorticity]{Large-amplitude periodic solutions to the steady Euler equations with piecewise constant vorticity}
\author{Alex Doak} 
\author{Karsten Matthies}
\author{Jonathan Sewell}
\author{Miles H. Wheeler}

\begin{document}
\begin{abstract}
    We consider steady solutions to the incompressible Euler equations in a two-dimensional channel with rigid walls. 
    The flow consists of two periodic layers of constant vorticity separated by an unknown interface. 
    Using global bifurcation theory, we rigorously construct curves of solutions that terminate either with stagnation on the interface or when the conformal equivalence between one of the layers and a strip breaks down in a $C^1$ sense.
    We give numerical evidence that, depending on parameters, these occur either as a corner forming on the interface or as one of the layers developing regions of arbitrarily thin width. 
    Our proof relies on a novel formulation of the problem as an elliptic system for the velocity components in each layer, conformal mappings for each layer, and a horizontal distortion which makes these mappings agree on the interface. 
    This appears to be the first local formulation for a multi-layer problem which allows for both overhanging wave profiles and stagnation points.
\end{abstract}

\maketitle
\tableofcontents

\section{Introduction}

The steady two-dimensional incompressible Euler equations are a classical problem with a long history.  Recent developments include `rigidity' results for flows without stagnation points \cite{HamelNadirashvili:NoStagImpliesShear} or with analytic regularity \cite{elgindi:vortFnExistsIfAnalytic}, and existence results for $C^k$ solutions with compactly supported velocity \cite{enciso2024schiffertypeproblemannuliapplications, enciso2024smoothnonradialstationaryeuler}.
There is a substantial literature on weak solutions where fluid properties change discontinuously across  unknown interfaces, such as vortex patches with piecewise-constant vorticity and internal gravity waves with piecewise-constant density. On one hand, this allows for the use of comparatively simple models (e.g.~irrotational flow) in the layers separated by the interfaces. On the other hand, the fact that these interfaces must be determined as part of the solution introduces many new complications. 
In particular, the interface can be singular, as for the celebrated irrotational water waves `of greatest height'. These waves feature corners at their crests, with interior angle $2\pi/3$, which are also \emph{stagnation points} where the fluid velocity vanishes \cite{kn:existence,amick:stokesWavesRigorous,plotnikov:conjecture}. The existence of such waves had originally been conjectured by Stokes~\cite{stokes1880considerations}; an analogous conjecture for rotating vortex patches in the plane -- this time with an interior angle of $\pi/2$ -- remains open despite some partial progress \cite{overman:limiting,HassainiaMasmoudiWheeler:catseyevortpatch,cg:nonconvex,wang2024:boundaryregularityuniformlyrotating}.
Numerics for internal gravity waves show still other possibilities, for instance \emph{overhanging} interfaces that are not the graph of a function \cite{meiron1983overhanging} or interfaces which touch a rigid wall \cite{dvb:fronts}.

One of the main challenges when studying such free boundary problems is to reformulate the problem in a fixed domain, either by introducing appropriate coordinates to flatten the fluid layers or, in some circumstances, by converting to a non-local equation posed on the interface itself. Conformal mappings are the canonical choice for problems involving overhanging interfaces, and have many desirable properties, especially for one-layer fluids. However, the case of two or more fluid layers presents a challenge: given two fluid layers sharing a common boundary, the conformal maps that take each layer to a known domain will in general disagree on this common boundary, and so the interfacial boundary conditions become non-local in the new variables. In this paper we resolve this issue by carefully distorting of one of the conformal mappings in such a way that the transformed system is not only local, but an elliptic system in the Agmon--Douglis--Nirenberg sense.
This novel formulation has many advantages, not least of which are access to Schauder-type estimates for linearized operators and the efficient calculation of Fredholm indices via homotopy arguments.

Using these new coordinates, we study waves propagating in a fluid with two layers of constant vorticity, bounded above and below by rigid walls. The boundary conditions along the interface are the same as in the vortex patch problem mentioned above, while the geometry of the two layers is closer to the internal gravity wave problem. The Fredholm properties afforded by the elliptic system setting allow us to apply analytic global bifurcation theory \cite{DancerE.N.1973BTfA,Buffoni--Toland:Book} to construct continuous curves of solutions which are in some sense maximal. Through additional elliptic estimates and maximum principle arguments we then give a more precise description of the limiting behaviour at the ends of these curves: either the derivative of one of the conformal mappings (or its inverse) becomes arbitrary large at the interface or the minimum speed of the fluid on the interface becomes arbitrarily small.

We complement these rigorous results with numerical calculations, this time using a non-local formulation, which give a much more detailed picture of the solutions. Near the ends of our numerical bifurcation branches, depending on the values of the parameters, interfaces sometimes develop corners and sometimes touch one of the rigid walls. 

\subsection{Statement of the main result}\label{sec:intro:setup}
We consider a steady incompressible inviscid fluid with constant density occupying an infinite horizontal channel $\R \times [0,1]$. Here we have used the height of the channel as our length scale when non-dimensionalising. We denote the horizontal and vertical coordinates by $X$ and $Y$ respectively, and the horizontal and vertical components of the fluid velocity field by $U=U(X,Y)$ and $V=V(X,Y)$ respectively. 
The fluid is divided into a lower layer $\Omega_0$ with constant vorticity $\omega_0$ and an upper layer $\Omega_1$ with constant vorticity $\omega_1$, so that
\begin{subequations}\label{eqn:largeampstream}
\begin{alignat}{2}
    \label{eqn:largeampstream:divfree}
    \partial_X U + \partial_Y V &= 0 &\qquad& \text{ in } \Omega_0 \cup \Omega_1 \\
    \label{eqn:largeampstream:lap}
    \partial_Y U - \partial_X V &= \omega_i &\qquad&\text{ in } \Omega_i \text{ for } i=0,1.
\end{alignat}
Assuming that the vorticities $\omega_1 \ne \omega_2$ are distinct, we now non-dimensionalise time so that $\omega_0-\omega_1=1$. 

We call the common boundary $\partial\Omega_0 \cap\partial\Omega_1$ the \emph{interface}, and denote it by $\Gamma$. Imposing the boundary conditions
\begin{align}
    \label{eqn:largeampstream:cont}
    (U,V) &\text{ continuous across } \Gamma \\
    \label{eqn:largeampstream:kinint}
    (U,V) &\text{ parallel to }\Gamma
\end{align}
guarantees that $(U,V)$ is a weak solution to the steady Euler equations in $\R \times (0,1)$ for some appropriate pressure $P$. 
The boundary condition \eqref{eqn:largeampstream:kinint} is a kinematic condition guaranteeing that fluid particles on the interface remain there for all time. Here we assume that $\Omega_0 \cup \Omega_1 \cup \Gamma =\R \times (0,1)$ and that the layers are unbounded with $\{Y=0\} \subseteq \partial \Omega_0$ and $\{Y=1\} \subseteq \partial\Omega_1$. On these upper and lower boundaries we also impose kinematic boundary conditions
\begin{alignat}{2}
    \label{eqn:largeampstream:kintop}
    V &= 0 &\qquad& \text{ on } Y=1 \\
    \label{eqn:largeampstream:kinbot}
    V &= 0 &\qquad& \text{ on } Y=0.
\end{alignat}

We emphasise that the geometry of the two layers is a priori unknown, and determined as part of the solution. 
Moreover, overturning waves are permitted, that is, $\Gamma$ need not be the graph of a function of $X$; see Figure~\ref{fig:setup}. 
We further restrict to solutions which are periodic in $X$ with prescribed wavenumber $k \in \R$,
\begin{equation}\label{eqn:largeampstream:periodic}
    (U,V) \text{ are $2\pi/k$-periodic in $X$, }
\end{equation}
and where $\Omega_0$ has prescribed average depth $H \in (0,1)$ in the sense that
\begin{align}\label{eqn:largeampstream:avg}
    \operatorname{Area}\bigl({\Omega_0 \cap [0,2\pi/k]\times[0,1]}\big)&=\frac{2\pi H}{k}.
\end{align}
See Figure~\ref{fig:setup} for a sketch of the problem.
\end{subequations}
\begin{figure}
     \centering
     \begin{subfigure}[b]{0.85\textwidth}
         \centering
         \includegraphics[width=\textwidth]{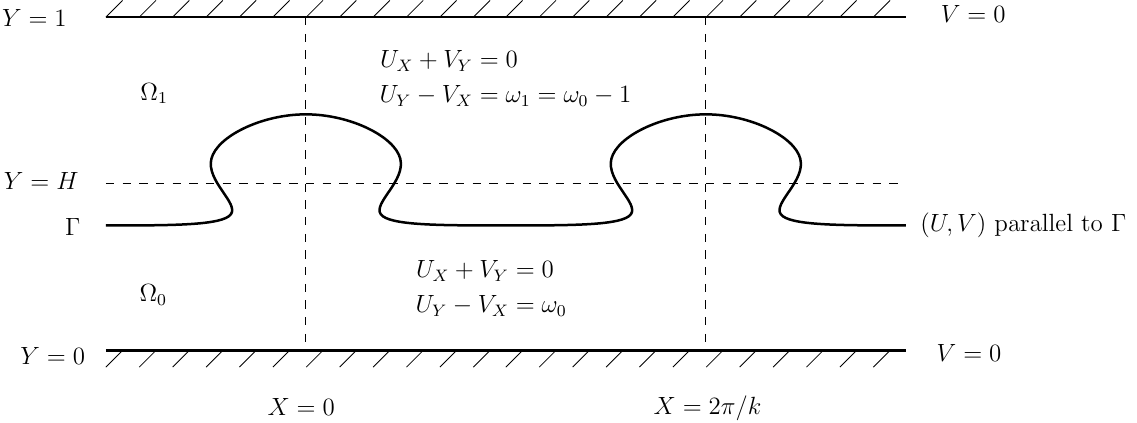}
         \label{fig:setup:setup}
     \end{subfigure}
        \caption{ Sketch of the problem \eqref{eqn:largeampstream}. The interface $\Gamma$ overturns, but is still monotone in the sense that the $Y$-coordinate strictly decreases as one moves form the crest at $X=0$ to the trough at $X=\pi/k$. }
        \label{fig:setup}
\end{figure}

The main theorem of this paper is the following.
\begin{thm}\label{thm:largeampmain}
   Fix real constants $H$, $k$, $\omega_0$, and $\omega_1$ such that $H \in (0,1)$ and $\omega_0-\omega_1=1$. 
    There exists a continuous curve of solutions to \eqref{eqn:largeampstream} with the following properties.
    \begin{enumerate}[label=\textup{(\roman*)}]
    \item 
    The geometry of each solution along the curve is described by constants $h_0,h_1>0$ and conformal maps
    \begin{align}\label{eqn:conformal}
        \begin{aligned}
        (\hat X_0, \hat Y_0) & \colon \R \times (0,h_0) \longrightarrow \Omega_0\\
        (\hat X_1, \hat Y_1) & \colon \R \times (h_0,h_0+h_1) \longrightarrow \Omega_1
        \end{aligned}
    \end{align}
    such that $\hat X_i(x,y)-2\pi x/k$ and $\hat Y_i(x,y)$ are periodic in $x$ with period $2\pi$. 
    \item
    The curve begins at a solution with no dependence on $X$, and as the parameter $\tau \to \infty$ along the curve, the positive quantity
    \begin{equation}\label{eqn:finalConclusion}
    \min_{i=0,1} \left\{  \frac{1}{\sup \limits_{y=h_0}|\nabla \hat Y_i|}, \ \inf_{y=h_0}|\nabla \hat Y_i|, \ \inf_{\Gamma} \big( U^2+V^2 \big) \right\} \longrightarrow 0.
    \end{equation}
    \item 
    For all solutions on the curve, the sets $\Omega_0,\Omega_1,\Gamma$ are symmetric under reflections in $X$ while $U$ is even in $X$ and $V$ is odd in $X$.
    \item 
    With the exception of the initial solution with no $X$ dependence, all solutions on the curve are strictly monotone in the sense that $V < 0$ in $(0,\pi/k) \times (0,1)$ and $\hat Y_i(h_0,x) < 0$ for $x \in (0,\pi)$.
    \end{enumerate}
\end{thm}
Let us briefly discuss the conclusion \eqref{eqn:finalConclusion}. Either of the first two terms vanishing corresponds to a breakdown in the conformal equivalence between $\Omega_i$ and a periodic strip: if the first vanishes then the gradient of the conformal map from the strip to $\Omega_i$ is becoming unboundedly large, and similarly for the second term and the inverse map.
The third term vanishing corresponds to the formation of a stagnation point on the interface, that is a point with fluid velocity $(U,V)=0$. Such a stagnation point would necessarily occur at the corner of a hypothetical limiting solution. Indeed, at any point on $\Gamma$ where $(U,V) \ne 0$, basic elliptic estimates and \cite[Theorem 3.1$'$]{spruck:interfaceisanalytic} imply that $\Gamma$ is a locally real-analytic curve. 

To get more precise predictions about the qualitative behaviour of the bifurcation branches in Theorem~\ref{thm:largeampmain}, we numerically calculate solutions of \eqref{eqn:largeampstream} for a wide range of parameter values $H,k,\omega_0$. 
Rather than using conformal mappings, we recast the problem as a non-local integral equation involving the values of the unknowns at the interface, which we parametrise in terms of a rescaled arclength.
These unknowns (the interface, and the velocity field at the interface) are sought as functions of arclength, and are expressed as Fourier series, which in turn are truncated. 
We discretise the boundary integrals into a system of non-linear equations,  and the finite number of Fourier coefficients are solved for via Newton--Raphson iterations. We observe that depending on the values of the parameters, either $\Gamma$ approaches one or both of the walls, or a corner develops on the peak or trough of $\Gamma$; see Figure~\ref{fig:main}.
\begin{figure}
     \centering
     \begin{subfigure}[b]{0.475\textwidth}
         \centering
         \includegraphics[width=\textwidth,trim={0 30 0 0},clip]{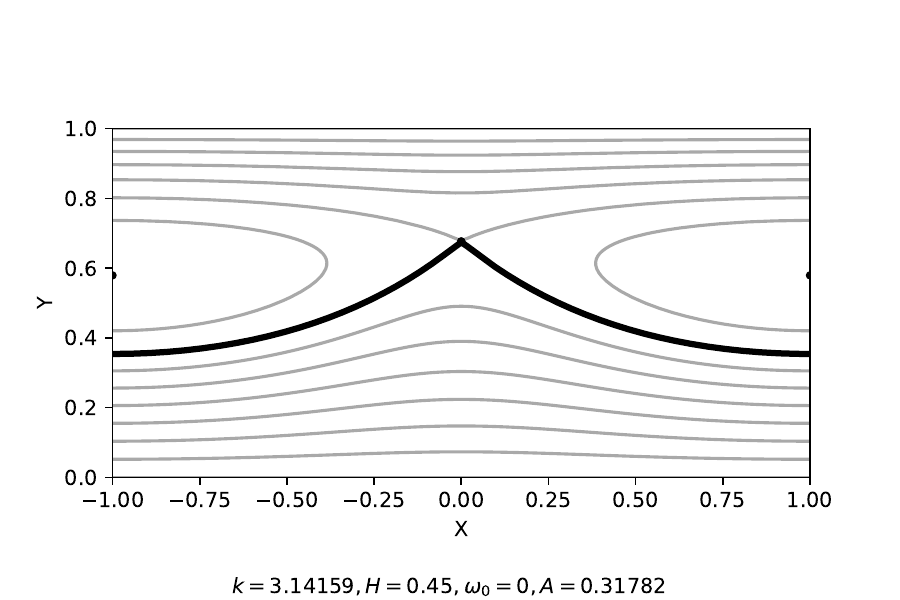}
	\caption{}
         \label{fig:main:corner}
     \end{subfigure}
     \begin{subfigure}[b]{0.475\textwidth}
         \centering
         \includegraphics[width=\textwidth,trim={0 30 0 0},clip]{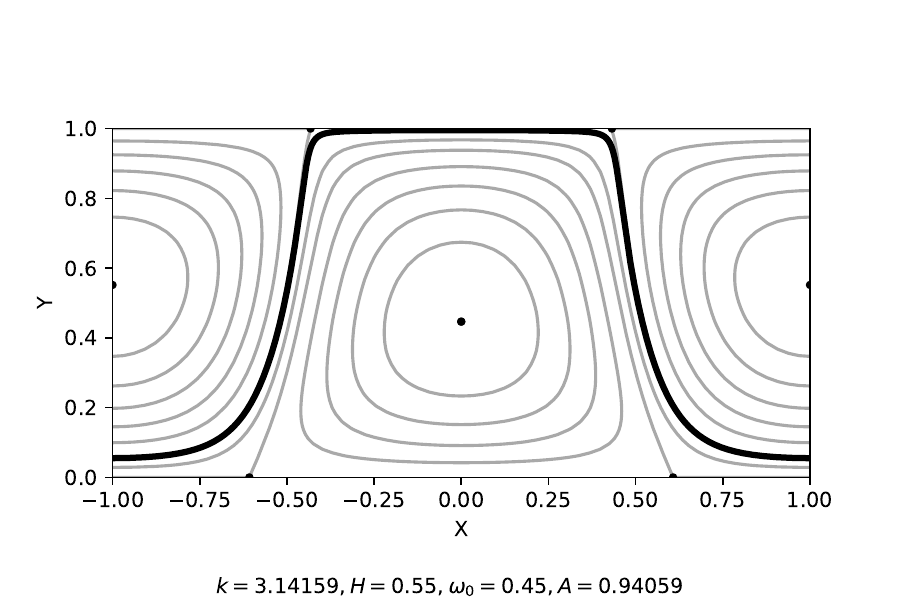}
	\caption{}
         \label{fig:main:touchwall}
     \end{subfigure}
        \caption{The streamlines of some numerical solutions near the end of their respective solution branches. 
	The interface is given by the black curve, the grey curves are interior streamlines, and the black points are stagnation points. 
	We have a solution in which the interface is almost singular in panel~(A), and a solution where the interface almost touches the upper wall in panel~(B). The solution in (A) has parameters $k=\pi$, $H=0.45$, $\omega_0=0$ and amplitude $0.31782$, and the solution in (B) has parameters $k=\pi$, $H=0.55$, $\omega_0=0.45$ and amplitude $0.94059$.}
        \label{fig:main}
\end{figure}
For certain critical parameter values, more than one of these may occur simultaneously.
When $\Gamma$ approaches the upper wall, this necessarily causes a breakdown of the conformal mapping in the upper layer, but we do not necessarily see stagnation points developing on $\Gamma$, and similarly for the lower wall and lower layer.
As expected, the formation of a corner results in stagnation on the interface.

As mentioned at the start of the introduction, the rigorous existence of steadily rotating vortex patches with singular boundaries remains an important open problem. Very recently, Wang, Zhang, and Zhou \cite{wang2024:boundaryregularityuniformlyrotating} have proven that the only possible boundary singularities for Lipschitz patches are stagnation points with corners with an angle of $\pi/2$; see \cite{overman:limiting} for an earlier result under much more restrictive hypotheses. Since their argument is local, and since our problem coincides with theirs in a small neighbourhood of any point on $\Gamma$, we expect that the same result should hold for a hypothetical limiting solution to our problem, provided that $\Gamma$ stays bounded away from the rigid walls at $Y=0$ and $Y=1$. 
Unfortunately, for our numerical method, the decay of the Fourier coefficients begins to slow as one approaches the limiting solutions. This is most likely caused by the discontinuity in the derivative of the vertical component of the interface with respect to arclength that occurs for solutions with a stagnation point on the interface. Hence, while solutions close to the limiting one can be recovered (such as those in Figure~\ref{fig:main}), in turn providing numerical evidence of the limiting behaviour, limiting solutions cannot be recovered. We can therefore not say with confidence what the angle at the crest of a limiting wave is. A more specialised solver is needed for this, and is the subject of future investigation.  

We also note that, while the formulations for both our theoretical and numerical work allow for overhanging waves, we did not observe this behaviour when sweeping parameter space numerically. Nevertheless, there are many other two-layer problems, for example internal gravity waves, where numerical results reveal overhanging waves (for example, see \cite{meiron1983overhanging}), and which could be treated by generalised versions of our approach.

Before continuing with the introduction, we pause to introduce and clarify some terminology.
By a steady solution to the two-dimensional incompressible Euler equations we mean a solution for which the fluid velocity $(U,V)$ and pressure $P$ are independent of time $t$. 
Thanks to Galilean invariance, this is equivalent to looking at travelling-wave solutions which depend on $X$ and $t$ only through the linear combination $X-ct$ for some wave speed $c$.
By a \emph{critical layer} we mean a curve along which the horizontal velocity $U$ vanishes.
We call a solution a \emph{shear} or a \emph{shear flow} if it has no dependence on the horizontal variable $X$.
Solutions with no vorticity are \emph{irrotational}, while those with vorticity are \emph{rotational}.
Irrotational flows have a velocity potential, that is, a function whose gradient is the velocity field.
Incompressibility \eqref{eqn:largeampstream:divfree} and two-dimensionality yields the existence of a \emph{stream function} $\Psi$, defined up to an additive constant by $\Psi_Y = U$ and $\Psi_X=-V$ so that it is constant along fluid particle trajectories, called \emph{streamlines}.
With our sign conventions, the vorticity can then be expressed as $U_Y-V_X = \Delta \Psi$.
Taking the curl of the steady incompressible Euler equations, one discovers that they are equivalent to the gradients of $\Psi$ and $\Delta\Psi$ being everywhere parallel.
In particular, near any non-stagnation point of a sufficiently smooth solution, one can find a \emph{vorticity function} 
$\gamma \colon \R \to \R$ such that $\Delta \Psi = \gamma(\Psi)$.
Assuming the \emph{global} existence of such a function reduces the steady Euler equations to a semilinear elliptic problem for $\Psi$, but this reduction is not possible for general steady flows.
When $\Psi$ is analytic, however, such a global vorticity function has very recently been shown to exist under surprisingly mild hypotheses \cite{elgindi:vortFnExistsIfAnalytic}.

\subsection{On the reformulation as a local elliptic system}\label{sec:intro:coords}

The key novelty in our proof of Theorem~\ref{thm:largeampmain} is a choice of coordinates which transforms \eqref{eqn:largeampstream} to a local elliptic system in a fixed domain. In the absence of critical layers, one could use a semi-hodograph transformation due to Dubreil-Jacotin \cite{DubreilJacotin}, treating the horizontal spatial coordinate $X$ and stream function $\Psi$ as independent variables and the vertical coordinate $Y$ as the sole dependent variable. However, \emph{all} small perturbations of shear solutions to \eqref{eqn:largeampstream} have critical layers and stagnation points, and so this approach is unavailable. We note that, for analogues of our problem where the velocity field is globally $C^2$, the existence of these stagnation points would follow from \cite{HamelNadirashvili:NoStagImpliesShear}.

Allowing for critical layers but requiring $\Gamma$ to be a graph $Y=\eta(X)$ enables the use of a variety of explicit  ``flattening transformations''. Indeed, the last three authors recently employed such a transformation to study small-amplitude solutions of a solitary-wave version of \eqref{eqn:largeampstream} using dynamical systems techniques \cite{ourpaper}. However, such an approach does not seem well-suited to global bifurcation. The transformed equations are of a non-standard type, coupling fluid unknowns defined on two-dimensional domains with the surface profile $\eta(X)$, an unknown function of a single variable. This significantly complicates, among other things, the calculation of Fredholm indices for general linearised operators. On the other hand, it may be possible to overcome these difficulties through the introduction of carefully chosen ``$\mathcal T$-isomorphisms'' \cite{EhrnströmWahlen:multipleCritlayers,Varholm:globalMultiCritLayer}.

Using conformal mappings as in \eqref{eqn:conformal}, we can allow for both internal stagnation points and overhanging geometry. With only a single fluid layer as in \cite{haziotWheeler}, the transformed equations could be written as an elliptic system for two or three scalar unknowns. For our two-layer problem, however, the parameterisations of the interface $\Gamma$ given by the conformal mappings $(\hat X_0,\hat Y_0),(\hat X_1,\hat Y_1)$ in \eqref{eqn:conformal} will in general disagree by some one-dimensional diffeomorphism $\R \times \{h_0\} \to \R \times \{h_0\}$, so that stating the interfacial boundary conditions \eqref{eqn:largeampstream:cont},\eqref{eqn:largeampstream:kinint} in the new variables requires introducing this diffeomorphism as an additional unknown. 
See \cite{jeanmarc:NumCalcGravCap} for a related numerical method. Unfortunately, the transformed boundary conditions are now non-local, involving compositions with the unknown diffeomorphism.
To obtain local boundary conditions, we instead work with a harmonic extension of this one-dimensional diffeomorphism to a two-dimensional diffeomorphism $\R \times (h_0,h_1) \to \R \times (h_0,h_1)$, and replace the conformal mapping $(\hat X_1,\hat Y_1)$ by its composition with this diffeomorphism. 

This choice leads to a local system of equations for nine dependent variables: four for the components of the velocity in the two layers, another four representing the mappings, and a final unknown for the diffeomorphism. Lengthy but ultimately straightforward algebraic calculations reveal that this system is elliptic in the Agmon--Douglis--Nirenberg~\cite{ADN} sense, which gives us access to powerful results from the theory of elliptic systems. 
In particular, linearised operators are Fredholm and enjoy Schauder-type estimates, and their Fredholm indices can be calculated by performing homotopies which stay in the same elliptic class. This in turn allows us to apply the analytic global bifurcation theory of Buffoni and Toland \cite{Buffoni--Toland:Book} based on work by Dancer \cite{DancerE.N.1973BTfA}.
To the best of our knowledge, the only previous work using elliptic systems tools to construct solutions to steady fluids problems is~\cite{haziotWheeler}, which considered a single-layer water wave problem. 

While we have developed this approach for \eqref{eqn:largeampstream}, it should generalise to a wide range of multi-layer problems. This includes internal gravity wave problems with a similar channel geometry but different boundary conditions, and also patch-type problems with similar boundary conditions but different geometry. For more general vorticity functions one would have to introduce the stream function $\Psi$ rather than work directly with the velocity components $U,V$; the advantage of working with $U,V$ in our problem is the particularly simple form of the boundary condition \eqref{eqn:largeampstream:kinint}.

\subsection{Related work}\label{sec:intro:lit}
While the linearisation of \eqref{eqn:largeampstream} about shear solutions dates back to the work of Rayleigh~\cite{Rayleigh:instability}, we are unaware of any rigorous work on the nonlinear steady problem beyond \cite{ourpaper}, where three of the authors perturbatively constructed solitary-wave solutions of small amplitude. We note that Hunter, Moreno-Vasquez, and Shu \cite{hmvsz:Burgers} and Biello and Hunter \cite{hunter:burgers4} have studied small-amplitude solutions to a related time-dependent problem with two semi-infinite layers.

On the other hand, there is an analogy between \eqref{eqn:largeampstream} and steady internal wave problems, where the two layers have differing constant densities $\rho_0 > \rho_1$ rather than differing vorticities and \eqref{eqn:largeampstream:kinint} is replaced by a more complicated condition expressing pressure continuity. 
Small-amplitude periodic and solitary waves of this type were constructed by Turner~\cite{turner:rapidly} as limits of solutions with smoothly-varying density using a Dubreil-Jacotin-type semi-hodograph transformation. Amick and Turner \cite{at:twofluidglobal} subsequently found large-amplitude solitary waves using a related approach, including sequences of solutions which either develop streamlines with vertical tangents -- in which case the semi-hodograph transformation breaks down -- or else broaden into fronts. 
This limitation of the semi-hodograph transformation prevents it from capturing overhanging waves.
While these internal waves are irrotational in each layer, Matioc \cite{MatiocAncaVoichita:criticalLayer} and Chen, Walsh, and Wheeler \cite{ChenWalshWheeler:WithoutPhaseSpace} have constructed small-amplitude internal waves with non-zero constant vorticity in one or both layers. In this case, as for single-layer water waves with constant \cite{Wahlen:critLayer} or affine  \cite{EhrnströmWahlen:multipleCritlayers,EhrnströmEscherVillari:multipleCritLayers}
vorticity, critical layers are possible even for perturbative solutions, and are captured using a piecewise-linear pointwise-in-$X$ flattening transformation.
We note that \eqref{eqn:largeampstream} can be recovered from this internal wave problem in the homogeneous-density limit $\rho_1 \to \rho_0$.
There are also results including surface tension effects on the interface. In particular, we mention the work of Akers, Ambrose, and Wright \cite{ambrose:vortexsheetsSmall}, who use a non-local formulation in terms of Cauchy integrals, much as we will do in the numerics section of this paper. These waves were continued to large amplitude by Ambrose, Strauss and Wright \cite{AMBROSE:vortexsheets}.
For a more thorough discussion of the literature on internal and stratified waves, we refer the reader to \cite[Section~7]{WaterWavesSurvey} and the references therein.

There is also a strong connection between \eqref{eqn:largeampstream} and rotating vortex patches. Here, working in a non-inertial rotating reference frame, the two layers
$\Omega_0$ and $\Omega_1$ are replaced by a connected set $D$ and its complement $\R^2 \setminus \overline D$, while the kinematic conditions on the walls $Y=0,1$ are replaced by asymptotic conditions for the velocity field at infinity. There are exact shear-type solutions when $D$ is a disk or annulus, and Kirchhoff~\cite{kirchhoff1897} discovered that ellipses give another class of explicit solutions. Other solutions near the disk, sometimes called ``Kelvin waves'' or ``V-states'', were rigorously constructed by Burbea \cite{Burbea:vortexPatches} (also see
\cite{HmidiMateuVerdera:vortexPatches}) and continued to global curves of solutions by Hassainia, Masmoudi, and Wheeler \cite{HassainiaMasmoudiWheeler:catseyevortpatch}. 
These works use a non-local formulation involving Cauchy integrals, but choose to parametrise the interface using the trace of a conformal mapping for $\R^2 \setminus \overline D$ rather than arc length. 
The bifurcation curves in \cite{HassainiaMasmoudiWheeler:catseyevortpatch} terminate as the interface develops either a stagnation point or a radial tangent. 
As for our problem, the perturbative solutions all exhibit stagnation points and (polar coordinate equivalents of) critical layers \cite{HassainiaMasmoudiWheeler:catseyevortpatch}. 
While there are a many related perturbative results, for instance near ellipses \cite{ccg:reg,hm:ellipse} or annuli \cite{hhmv:doubly}, the only other global bifurcation result is due to García and Haziot \cite{gh:globalpairs}, who treat symmetric pairs of patches. 
On the other hand, there are interesting `rigidity' results classifying the possible 
singularities of solutions \cite{overman:limiting,wang2024:boundaryregularityuniformlyrotating} as well as their angular velocities \cite{fraenkel:book,hmidi:trivial,gpsy:symmetry}.

Let us also briefly mention several relevant global bifurcation results for the single-layer water wave problem with vorticity. The first is due to Constantin and Strauss~\cite{ConstantinStrauss:TopologicalGlobBif}, who constructed periodic waves with general vorticity functions using the Dubreil-Jacotin partial hodograph transformation.
Later, Constantin, Strauss and V\u{a}rv\u{a}ruc\u{a} \cite{csv:global} constructed the first 
global curve allowing for overhanging waves (as well as critical layers and stagnation points). They used a subtle non-local formulation of the problem based on conformal mappings and restricted to constant vorticity. 
Whether such global bifurcation curves actually contain overhanging waves largely remains an open problem, but there are positive results in certain extreme parameter regimes \cite{HurWheeler:exacltyCrapper,HurWheeler:NearCrapper,goncalves:touching}; also see \cite{jmm:overhanging}.
Moving beyond constant vorticity, Wahlén and Weber constructed curves of potentially overhanging solutions with general vorticity functions \cite{WahlénErik2024Lsgw} using conformal mappings and a hybrid formulation which is partially locally and partially non-local. 
We also mention \cite{haziotWheeler}, which constructs possibly overturning solitary waves with constant vorticity, again using conformal mappings, but this time using a formulation as a local elliptic system. 

To obtain the numerical results described after the statement of Theorem~\ref{thm:largeampmain}, we express the fluid unknowns using Cauchy's integral formula, arriving at a non-local formulation of the problem posed on the one-dimensional interface. Similar numerical methods have been used for a wide variety of irrotational and constant vorticity free-surface flows; see for instance \cite{jeanmarc:book} and the references therein. 
Their first application to a steady multi-layer flow was by Vanden-Broeck \cite{jeanmarc:NumCalcGravCap}. 
Turning to rotational waves, large waves with constant vorticity in a single layer have been computed for infinite \cite{SimmenSaffman:1layervort} and finite depth
\cite{peregrine:1layervorticity}, as well as in the solitary wave limit
\cite{jeanmarc:solitaryConstVort}. Also see \cite{peregrine:1layervorticity} and \cite{jeanmarc:solitaryConstVort} for more complete pictures of the flows found in \cite{riberoMilewskiNachbin:1layerVort} and \cite{GUAN:solitary}.
Formulations based on Cauchy’s integral formula have also been used for time-dependent irrotational simulations, for example by Dold \cite{Dold:timeDepNumerics} for surface waves and by Guan and Vanden-Broeck \cite{Guan:novelBoundary} for interfacial waves.
Similar methodologies have been used for numerical computations of vortex patches. 
The seminal paper by Zabusky, Hughes and Roberts \cite{zabuskyHugesRoberts} developed a time-dependent algorithm to solve the dynamics of vortex patches known as \emph{contour dynamics}. The method makes use of the Green's function to express the velocity field in terms of integrals along the boundaries of the vortex patches, which in turn are used to evolve the boundaries in time. Deem and Zabuksy \cite{dz:vstates} sought time-independent ``V-states" by assuming constant translation or rotation of the flow configuration. This non-local formulation was utilised by a variety of authors to compute steady state solutions to the vortex patch problem (see, for example \cite{Pierrehumbert_1980,saffman:shapes,Dritschel_1985,woz:numerical},  and for a review, \cite{pullinReview}). 

\subsection{Outline of the paper}

Firstly, in Section~\ref{sec:largeampprelims} we introduce the coordinate change discussed in Section~\ref{sec:intro:coords} to eliminate the free boundary.
This recasts the problem, and we discuss under what conditions solutions to the recast problem correspond to solutions of the original problem. 
We then state an abstract global bifurcation theorem, Theorem~\ref{thm:buffoniToland}, and introduce the family of shear solutions which we will bifurcate from.
Each of the three sections that follow verifies a hypothesis of Theorem~\ref{thm:buffoniToland}.
In Section~\ref{sec:fredholmProperties}, we prove that all relevant linearised operators are Fredholm between appropriate Hölder spaces. 
This is achieved by showing the linearised operator constitutes a so-called $L$-elliptic system \cite{wrl:elliptic}.
Section~\ref{sec:localBifurcation} verifies the ``transversality condition'' of Crandall--Rabinowitz~\cite{cr:simple}, yielding a local bifurcation result. 
Thus our recast problem has solutions a small but non-zero distance from the shear solutions in Section~\ref{sec:largeampprelims}.
To extend to solutions far from a shear flow, we establish a compactness result in Section~\ref{sec:UniformReg} as a corollary of a much stronger result: certain $C^1$ bounds on components of a solution imply $C^{\ell,\alpha}$ bounds on the whole solution, for arbitrary $\ell \in \N$.
We use this stronger result in Section~\ref{sec:refining} to deduce that if blow-up occurs, this must be in the sense of \eqref{eqn:finalConclusion}, rather than the weaker sense of Theorem~\ref{thm:buffoniToland}(a).
We also use maximum principle arguments to rule out the possibility that the solution curve is a closed loop.
Finally, in Section~\ref{sec:numerics}, we present numerical solutions illuminating the qualitative behaviour of these large amplitude waves, consistent with Theorem~\ref{thm:largeampmain}.

\section{Preliminaries}\label{sec:largeampprelims}
Here we recast our problem \eqref{eqn:largeampstream} in an entirely novel manner. The free boundary is eliminated at the cost of making the problem highly non-linear.
However, it remains locally posed, which crucially allows us later in Section~\ref{sec:fredholmProperties} to show it gives an elliptic system when linearised, immediately yielding Schauder estimates.
We then introduce notation to discuss various function spaces.
As in many global bifurcation problems, some simple solutions are found, which serve as a starting point from which to bifurcate.
Given a solution to the reformulated  problem, we show we can recover a solution to the original problem. 
Finally, a global bifurcation theorem of Buffoni and Toland \cite{Buffoni--Toland:Book} is stated.
This result is fundamental for our approach as it allows us to find solutions with large amplitude.

\subsection{Reformulation}\label{sec:largeampprelims:reformulation}

Let $\mathcal{D}=\T \times (0,1)$, where $\T=\R/2\pi\mathbb{Z}$.
We seek two diffeomorphisms, $(X_0,Y_0) \colon \mathcal{D} \to \Omega_0$ and $(X_1,Y_1) \colon \mathcal{D} \to \Omega_1$.
The map $(X_0,Y_0)$ should be sense-preserving and $(X_1,Y_1)$ sense-reversing, that is, the determinants of their Jacobians should be everywhere positive and negative respectively. 
This is consistent with $(X_0,Y_0)$ mapping $\T \times \{0\}$ to $\T \times \{0\}$, $(X_1,Y_1)$ mapping $\T \times \{0\}$ to $\T \times \{1\}$, and both the $(X_i,Y_i)$ mapping  $\T \times \{1\}$ to the interface $\Gamma$. 

The naive idea of letting $(X_0,Y_0)$ be a conformal mapping and $(X_1,-Y_1)$ be a conformal mapping has a major drawback, namely that, generically, $X_0(x,1) \neq X_1(x,1)$ and $Y_0(x,1) \neq Y_1(x,1)$, so the boundary conditions at the interface become non-local.
We instead do some pre-processing by introducing functions $S_0$ and $S_1$:
\begin{subequations}
\label{eqn:introduceXhatYhat}
\begin{align}
    \label{eqn:introduceXhatYhat:BigS}
    S_0(x,y)=(x,h_0 y), \qquad S_1(x,y) = (x+s(x,y),h_0 + h_1 - h_1 y),
\end{align}
where the $h_i$ are positive real parameters which give us enough freedom to control the average depth of the interface and the period (in physical space) of our solution.
The function $s$ is unknown, and determined as part of our solution. We use many arguments involving ellipticity, so it is natural to let $s$ be harmonic. It has domain $\mathcal{D}$, so is $2 \pi$ periodic in $x$.
Our coordinate change is then given by
\begin{align}
    \label{eqn:introduceXhatYhat:0}
    (X_0,Y_0) &= (\hat{X}_0,\hat{Y}_0) \circ S_0 \\
    \label{eqn:introduceXhatYhat:1}
    (X_1,Y_1) &= (\hat{X}_1,\hat{Y}_1) \circ S_1,
\end{align}
\end{subequations}
where the $(\hat{X}_i,\hat{Y}_i)$ are conformal mappings. This system has enough freedom for us to insist that for all $x$, we have $X_0(x,1)=X_1(x,1)$ and $Y_0(x,1)=Y_1(x,1)$; see Figure~\ref{fig:conformalMaps}.
We also define functions $\chi_i$ and $\eta_i$ which are $2 \pi$ periodic in $x$. Let
\begin{subequations}
\label{eqn:introduceChiEta}
\begin{align}
    \label{eqn:introduceChiEta:Chi}
    \chi_i(x,y) &= X_i(x,y)-\frac{1}{k}x\\\label{eqn:introduceChiEta:eta}
    \eta_i(x,y) &= Y_i(x,y),
\end{align}
\end{subequations}
where $2\pi/k$ is the $X$ period of the solution in physical space.

\begin{figure}
     \centering
     \begin{subfigure}[b]{0.9\textwidth}
         \centering
         \includegraphics[scale=0.65]{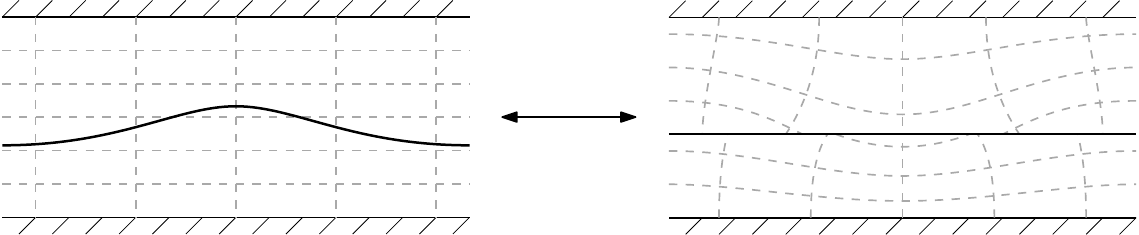}
	\caption{}
         \label{fig:conformalmaps:disagree}
	\bigskip
    \bigskip
     \end{subfigure}	
     \begin{subfigure}[b]{0.9\textwidth}
         \centering
         \includegraphics[scale=0.65]{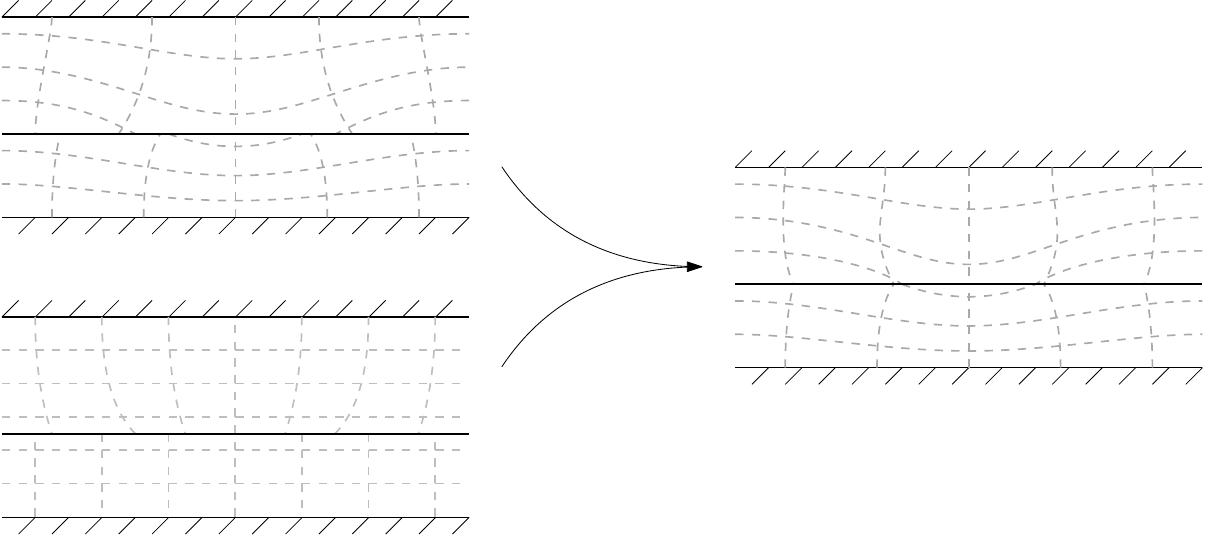}
	\caption{}
         \label{fig:conformalmaps:composition}
     \end{subfigure}
        \caption{A sketch of the issue our coordinate change rectifies. The interface is given by the black curve, and the grey dashed curves are lines on which the coordinates are constant.
        In (A) we see an example of conformal mappings disagreeing at the interface.
        We then see in (B) how composing a pair of conformal maps (upper left) with a pointwise-in-$y$ coordinate transform in the upper layer (lower left) yields a coordinate change which is continuous across the interface (right).}
        \label{fig:conformalMaps}
\end{figure}

Finally, we define the modified velocity field.
Let
\begin{subequations}
\begin{align}
    \label{eqn:defineq}q &= \frac 1 {2\pi} \int_0^1 \int_0^{2 \pi} U_0(X_0(x,y),Y_0(x,y)) \  dx \ dy \\
    \label{eqn:defineLilu}u_i(x,y)&= U_i(X_i(x,y),Y_i(x,y))-q\\
    v_i(x,y)&= V_i(X_i(x,y),Y_i(x,y)).
\end{align}
\end{subequations}
The parameter $q$ is related to the mass flux in the lower layer. Rather than simply using $U_0(X_0(x,y),Y_0(x,y))$ as our new independent variable, we have split it into its $0$th Fourier mode $q$, which we use as a bifurcation parameter, and the remainder $u_0$. 
Subtracting $q$ from $U_1(X_1(x,y),Y_1(x,y))$ ensures we retain continuity at the interface.

We can now rewrite our problem in terms of four known parameters, three unknown parameters, and nine unknown functions with domain $\mathcal{D}$. The known parameters are $\omega_0$, $\omega_1$, $k$, and $H$.
The unknown parameters are $h_0$, $h_1$, and $q$.
The functions are $u_0$, $v_0$, $\chi_0$, $\eta_0$, $u_1$, $v_1$, $\chi_1$, $\eta_1$, and $s$. They satisfy
\begin{subequations}\label{eqn:largeampstream2}
  \begin{alignat}{2}
	\label{eqn:largeampstream2:lowerdivfree}
  	0&=  h_0 u_{0x} + v_{0y} - h_0 \omega_0 \eta_{0x} \\
  	\label{eqn:largeampstream2:lowervort}
    0&=  u_{0y} - h_0 v_{0x} - \omega_0 \eta_{0y}  \\
    \label{eqn:largeampstream2:lowerCR1}
    0&=  h_0( k \chi_{0x} +1) - k \eta_{0y}  \\
    \label{eqn:largeampstream2:lowerCR2}
    0&=  \chi_{0y} + h_0 \eta_{0x} \\
    \label{eqn:largeampstream2:upperdivfree}
    0&=  k\eta_{1y} u_{1x} - k\eta_{1x} u_{1y} - k\chi_{1y} v_{1x} +(k\chi_{1x}+1) v_{1y}\\
    \label{eqn:largeampstream2:uppervort}
    0&=  -k\chi_{1y} u_{1x} +(k\chi_{1x}+1) u_{1y} - k\eta_{1y} v_{1x} +k\eta_{1x} v_{1y} - ((k\chi_{1x}+1)\eta_{1y}-k\chi_{1y}\eta_{1x})\omega_1  \\
    \label{eqn:largeampstream2:upperCR1}
    0&=  h_1 (k\chi_{1x}+1) - k\eta_{1x} s_y + k(s_x+1)\eta_{1y} \\
    \label{eqn:largeampstream2:upperCR2}
    0&=  (k\chi_{1x} +1) s_y + k h_1 \eta_{1x} - k\chi_{1y}( s_x+1) \\
    \label{eqn:largeampstream2:lap}
    0&=  \Delta s,
  \end{alignat}
in $\mathcal{D}$, with boundary conditions
  \begin{alignat}{2}
	\label{eqn:largeampstream2:kinbot}
  	0&=  v_0(x,0)\\
  	\label{eqn:largeampstream2:Ybot}
  	0&=  \eta_0(x,0)\\
  	\label{eqn:largeampstream2:kintop}
  	0&=  v_1(x,0)\\
  	\label{eqn:largeampstream2:Ytop}
  	0&=  \eta_1(x,0)-1\\
        \label{eqn:largeampstream2:sTop}
        0&=  s(x,0)\\
        \intertext{and}
  	\label{eqn:largeampstream2:ucts}
  	0&=  u_0(x,1) - u_1(x,1) \\
  	\label{eqn:largeampstream2:vcts}
  	0&=  v_0(x,1)-v_1(x,1)  \\
  	\label{eqn:largeampstream2:Xcts}
  	0&=  \chi_0(x,1)-\chi_1(x,1)  \\
  	\label{eqn:largeampstream2:Ycts}
  	0&=  \eta_0(x,1)-\eta_1(x,1)  \\
  	\label{eqn:largeampstream2:kinint}
  	0&=  (u_0(x,1)+q)\eta_{0x}(x,1)-v_0(x,1)\left(\chi_{0x}(x,1)+\frac{1}{k}\right) ,
  \end{alignat}
on the rigid walls and scalar constraints
    \begin{align}
        \label{eqn:largeampstream2:avgDepth}
        0&=  \frac{1}{2\pi} \int_0^{2\pi} (k \chi_{0x}(x,1)+1)\eta_0(x,1) \ dx - H\\
        \label{eqn:largeampstream2:flux}
        0&=  \int_0^1 \int_0^{2 \pi} u_0 \ dx \ dy .
    \end{align}
\end{subequations}
We can think of \eqref{eqn:largeampstream2:lowerdivfree} and \eqref{eqn:largeampstream2:upperdivfree} as incompressibility, and \eqref{eqn:largeampstream2:lowervort} and \eqref{eqn:largeampstream2:uppervort} as the vorticity conditions.
Note that \eqref{eqn:largeampstream2:lowerdivfree} and \eqref{eqn:largeampstream2:lowervort} are not obtained directly by applying the coordinate transformation to \eqref{eqn:largeampstream:divfree} and \eqref{eqn:largeampstream:lap}; some additional but straightforward rearranging is also needed.
The equations \eqref{eqn:largeampstream2:lowerCR1}, \eqref{eqn:largeampstream2:lowerCR2}, \eqref{eqn:largeampstream2:upperCR1}, and \eqref{eqn:largeampstream2:upperCR2} are derived from the Cauchy--Riemann equations and \eqref{eqn:introduceXhatYhat}.
The kinematic boundary conditions at the bottom, top, and interface become equations \eqref{eqn:largeampstream2:kinbot}, \eqref{eqn:largeampstream2:kintop}, and \eqref{eqn:largeampstream2:kinint} respectively.
Equations \eqref{eqn:largeampstream2:ucts}--\eqref{eqn:largeampstream2:Ycts} enforce continuity at the interface.
Applying the divergence theorem to the vector field $(0,Y)$ on $\Omega_0$, and enforcing the condition that the interface has average depth $H$, yields \eqref{eqn:largeampstream2:avgDepth}.
The flux condition \eqref{eqn:largeampstream2:flux} is derived from \eqref{eqn:defineq} and \eqref{eqn:defineLilu}.

It is worth noting that we work with the velocity vector $(U,V)$ rather than the stream function $\Psi$, and furthermore do not remove all occurrences of $\hat X_i$ and recover them later as the harmonic conjugate of $-\hat Y_i$.
The advantages of this are that the resulting system is first order, and, while large, is reasonably sparse.
We found that this sparsity was helpful in the very delicate computations in Section~\ref{sec:lopatinskii}. 
The use of computer algebra also assisted greatly with the repetition and inelegance inherent to calculations with large matrices. 

\subsection{Notation}
We now introduce some more compact notation. Let
\begin{align*}
    \zeta &= (u_0, v_0, \chi_0, \eta_0, u_1, v_1, \chi_1, \eta_1, s; h_0, h_1)\\
    \zeta^\circ &= (u_0, v_0, \chi_0, \eta_0, u_1, v_1, \chi_1, \eta_1, s),
\end{align*}
and $\mathcal{F}(\zeta,q)$ correspond to the right-hand side of equations \eqref{eqn:largeampstream2:lowerdivfree}--\eqref{eqn:largeampstream2:flux}. In other words, a solution of \eqref{eqn:largeampstream2} is $(\zeta,q)$ such that
\begin{align}
    \label{eqn:defineMathcalF}
     \mathcal{F}(\zeta, q)=0.
\end{align}
In order to describe functional and scalar equations separately, let $\mathcal{F}^\circ(\zeta,q)$ correspond to the right-hand side of the functional equations, \eqref{eqn:largeampstream2:lowerdivfree}--\eqref{eqn:largeampstream2:kinint}, and let $\mathcal{F}^s(\zeta,q)$ correspond to the right-hand side of the scalar equations, \eqref{eqn:largeampstream2:avgDepth} and \eqref{eqn:largeampstream2:flux}.
Thus $\mathcal{F} = (\mathcal{F}^\circ,\mathcal{F}^s)$.

We introduce function spaces for the domain and codomain of $\mathcal{F}$.
For spaces related to the domain we use $\mathcal{X}$s, and for spaces related to the codomain, $\mathcal{Y}$s.
A superscript $^\circ$ will generally mean ``without the scalar components'', or ``functional components only'', as with $\zeta^\circ$ and $\mathcal{F}^\circ$. 
This is an imprecise comment, but we will be more precise in the definitions that follow.
Given a bounded open set $\mathcal{O}$, a non-negative integer $n$, and $\gamma \in (0,1)$, let $C^{n,\gamma}(\overline {\mathcal{O}})$ be the set of functions with domain $\mathcal{O}$, whose derivatives of order $\leq n$ exist and have continuous extension to the boundary with finite $\gamma$ Hölder norm.

Let
\begin{subequations}\label{eqn:defineNonParitySpaces}
\begin{align}
    \label{eqn:defineXnpCirc}
    \mathcal{X}_\np^\circ &= (C^{2,\gamma}(\overline{\mathcal{D}}))^9\\
    \label{eqn:defineYnpCirc}
    \mathcal{Y}_\np^\circ &= \left[ (C^{1,\gamma}(\overline{\mathcal{D}}))^8 \times C^{0,\gamma}(\overline{\mathcal{D}}) \right] \times \left[(C^{2,\gamma}(\overline{\T}))^5 \right] \times \left[ (C^{2,\gamma}(\overline{\T}))^4 \times C^{1,\gamma}(\overline{\T}) \right].
\end{align}
\end{subequations}
These function spaces have no parity conditions, hence the use of the subscript $\np$.
The formula for $\mathcal{Y}_\np^\circ$ is given as the product of three spaces in square brackets.
The space in the first set of square brackets corresponds to \eqref{eqn:largeampstream2:lowerdivfree}--\eqref{eqn:largeampstream2:lap}, the second to \eqref{eqn:largeampstream2:kinbot}--\eqref{eqn:largeampstream2:sTop}, and the third to \eqref{eqn:largeampstream2:ucts}--\eqref{eqn:largeampstream2:kinint}.
The degrees of regularity are chosen to be the least regular Hölder spaces for which the arguments of Section~\ref{sec:lopatinskii} below hold.
These arguments apply the theory of elliptic systems to show that the associated linearised operators are Fredholm.

We now introduce parity constraints.
Let
\begin{subequations}
\begin{align}
    \label{eqn:defineXcirc}
    \mathcal{X}^\circ &= \{ \zeta^\circ \in \mathcal{X}_\np^\circ  \mid u_0, \eta_0, u_1, \eta_1 \text{ are even in $x$, and } v_0,\chi_0,v_1, \chi_1,s \text{ are odd in $x$} \}\\
    \label{eqn:defineX}
    \mathcal{X} &= \mathcal{X}^\circ \times \R^2 .
\end{align}
\end{subequations}
Let $\mathcal{Y}^\circ$ be the subset of $\mathcal{Y}^\circ_\np$ such that the elements have the same parity in $x$ as the elements of $\mathcal{F}^\circ(\mathcal{X})$.
Explicitly, components 1, 4, 5, 8, 9, 10, 12, 14, 16, 17 and 19 are odd in $x$, and 2, 3, 6, 7, 11, 13, 15, 18 are even in $x$.
The codomain $\mathcal{Y}$ is then given by $\mathcal{Y} =\mathcal{Y}^\circ \times \R^2$.
The parity conditions are imposed due to two symmetries of the problem, translation and reflection.
We have a translational symmetry in that given a solution to \eqref{eqn:largeampstream}, transformations of the form $U (\cdot, \cdot)  \mapsto U (\cdot + a, \cdot)$ give rise to
others. 
To eliminate these solutions, which are in some sense all the same, we insist $U$ is even in
$X$, and $V$ is odd in $X$.
This is motivated by the fact \eqref{eqn:largeampstream} has a reflectional symmetry, in that if $(U,V)$ solves \eqref{eqn:largeampstream}, then the functions $(U(-X,Y),-V(-X,Y))$ give another solution.

We seek $(\zeta,q) \in \mathcal{X} \times \R$ solving \eqref{eqn:largeampstream2}. 

\subsection{Shear solutions}
As with many problems of this type, we start by finding a family of trivial solutions to \eqref{eqn:largeampstream2}. For an arbitrary constant $c \in \R$, the shear flow given by 
\begin{equation}\label{eqn:defineShear}
    U=\begin{cases}
        \omega_0(Y-H)+c & Y \in [0,H]\\
        \omega_1(Y-H)+c & Y \in [H,1],
    \end{cases} \qquad V=0, \qquad  \Omega_0=\R\times(0,H), \qquad  \Omega_1=\R\times(H,1),
\end{equation}
solves the physical problem \eqref{eqn:largeampstream}. See Figure~\ref{fig:shear}.
\begin{figure}
     \centering
     \begin{subfigure}[b]{0.35\textwidth}
         \centering
         \includegraphics[width=\textwidth]{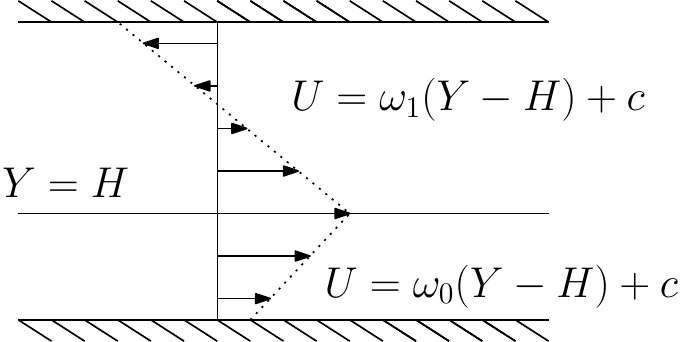}
     \end{subfigure}
        \caption{A shear solution in \eqref{eqn:defineShear}.}
        \label{fig:shear}
\end{figure}
Applying the coordinate change gives $q= c-\frac12\omega_0H$, and
\begin{align*}
     u_0^*&= \omega_0 H (y-\tfrac 12 ), &  v_0^* &=0, & \chi_0^*&= 0, &  \eta_0^* &=Hy,\\
      u_1^*&=\omega_1 (1-H)(1-y) + \tfrac 12 \omega_0 H ,&  v_1^* &=0  ,&  \chi_1^*&=0 ,&  \eta_1^* &=1-(1-H)y ,\\
      s^*&=0,\\
      h_0^*&=kH, & h_1^*&=k(1-H).
\end{align*}
As before, we write these more compactly as
\begin{equation}\label{eqn:defZetaStar}
     \zeta^\circ_*=(u_0^*,\ldots,s^*), \qquad \zeta_*=(u_0^*,\ldots,h_1^*).
\end{equation}
Thus, for any $q \in \R$, inserting $(\zeta_*,q)$ into  \eqref{eqn:largeampstream2} gives 0, so we have a family of trivial solutions.

As we will see in Proposition~\ref{prop:CrandallRabinowitzKernel} below,
the value of $q$ that we bifurcate from lies in the interval \[ \left(-\tfrac 12 \omega_0H, -\tfrac 12 \omega_0H + H(1-H) \right).\]
Shear flows with $q$ in this range exhibit some interesting properties.
Firstly, there exists at least one critical layer.
Secondly, if $\omega_0<1-2H$ or $\omega_0>2-2H$, then no global vorticity function exists.
These observations extend to nearby non-shear flows,
and can be shown by adapting the arguments from \cite[Proposition~2.4]{ourpaper}.

\subsection{Correspondence to physical solutions}\label{sec:correspondance}

We restrict the domain of $\mathcal{F}$ from $\mathcal{X} \times \R$ to a set $\mathcal{U}$ on which solutions to \eqref{eqn:largeampstream2} have various desirable properties.
The coordinate changes should be bijective, the interface should not intersect the walls or itself, and the fluid should not stagnate on the interface.
With these in mind, let $\mathcal{U}$ be the subset of $\mathcal{X} \times \R$ such that for all $(x,y) \in \mathcal{D}$, and all $(a,b) \in [0,4\pi]^2$,
\begin{subequations}
\begin{align}
    \label{eqn:defU:h} 0&<h_0, h_1\\
    \label{eqn:defU:s}0&< s_x(x,y)+1\\
    \label{eqn:defU:gradX} 0&< (\chi_{ix}(x,y)+1/k)^2+ \eta_{ix}(x,y)^2 \\
    \label{eqn:defU:interfaceStagnation} 0&<(u_0(x,1)+q)^2 + v_0(x,1)^2 \\
    \label{eqn:defU:wallIntersect} 0&<\eta_0(x,1)<1\\
    \label{eqn:defU:selfIntersect} 0&<|\sigma(a,b)|,
\end{align}
\end{subequations}
where $\sigma$ is defined as
\begin{equation}
\label{eqn:defineSigma}
\begin{aligned}
     \sigma\colon  & [0,4\pi]^2 \mapsto \R^2 \\
     \sigma(a,b)&= \begin{cases}
    \dfrac{ (X_0(a,1),Y_0(a,1)) - (X_0( b,1),Y_0(b,1))}{a-b} & a \neq b \\[2ex]
    (X_{0x}(a,1),Y_{0x}(a,1)) & a=b.
    \end{cases}
\end{aligned}    
\end{equation}

The conditions used to define $\mathcal{U}$ are all very natural.
If one of the $h_i=0$, then we would be seeking a conformal equivalence between $\Omega_i$ and $\varnothing$ which is clearly impossible. 
Much of our analysis holds for $h_i<0$, but we construct a curve of solutions starting in $\mathcal{U}$, so in order to access a solution with, say, $h_0<0$, we would have to pass through a solution with $h_0=0$.

The conditions \eqref{eqn:defU:s} and \eqref{eqn:defU:gradX} help ensure that the coordinate change is well-behaved.
If either of these quantities are zero in $\mathcal{D}$, then the coordinate change is not bijective, and if they are zero on $\partial \mathcal{D}$, then the interface will be too irregular to extend a conformal map.
Lemma~\ref{lem:bijectiveInU} contains further details.

The condition \eqref{eqn:defU:interfaceStagnation} rules out stagnation points on the interface, and will be needed several times to show that various systems are elliptic.
Intuitively, if equality held at some point, then the kinematic boundary condition \eqref{eqn:largeampstream2:kinint} would not enforce any behaviour on $\chi_0$ and $\eta_0$ at this point, leaving our system under-determined.
Alternatively, we can note that if we have a solution of \eqref{eqn:largeampstream2} corresponding to a solution $(U,V)$ of \eqref{eqn:largeampstream}, then we can define a stream function $\Psi(X,Y)$.
The interface is a level curve of $\Psi$, and \eqref{eqn:defU:interfaceStagnation} ensures that $|\nabla \Psi| \neq 0$ on the interface.
Thus, since $\Psi$ is globally $C^1$, the implicit function theorem implies that the interface is a $C^1$ curve. Then, applying \cite[Theorem 3.1$'$]{spruck:interfaceisanalytic} yields analyticity.

Equations \eqref{eqn:defU:wallIntersect} and \eqref{eqn:defU:selfIntersect} correspond to insisting the interface does not intersect a wall or itself respectively.
This will help us show that the coordinate transforms are bijective. It is not immediately obvious that \eqref{eqn:defU:selfIntersect} rules out all self-intersections since $\sigma$ only takes arguments in $[0,4\pi]$, rather than all of $\R$. We address this in the following lemma.

\begin{lem}
    Let $X_0,Y_0 \in C^1(\overline{\R\times [0,1]})$ with $X_0$ odd in $x$, $Y_0$ even in $x$, $X_0(x+2\pi,y) = X_0(x,y)+2\pi/k$, and $Y_0(x+2\pi,y) = Y_0(x,y)$. Recall the definition of $\sigma$, \eqref{eqn:defineSigma}.
    If the interface given by $(X_0(x,1),Y_0(x,1))$ self-intersects, then there exists $(\alpha,\beta) \in [0,4\pi]^2$ such that $\sigma(\alpha,\beta)=0$.
\begin{proof}
    Suppose $(X_0(a,1),Y_0(a,1))=(X_0(b,1),Y_0(b,1))$ for some $(a,b) \in \R^2$ with $a<b$.
    By periodicity, we can assume that $a \in [0,2\pi)$, and $b \in [2n\pi, 2(n+1)\pi)$, for some non-negative integer $n$.
    If $n=0$ or $n=1$, we are done.
    If not, then let $\tilde b = b-2n\pi$.
    Notice that $X_0(\tilde b,1)=X_0(a,1)-2n\pi/k$.
    Thus, we have $a, \tilde b \in [0,2\pi)$ with either $X_0(a,1)\geq 2\pi/k$, or $X_0(\tilde b,1)<0$.
    If the former, then since $X_0(0,1)=0$, by continuity there exists $x \in (0,2\pi)$ such that $X_0(x,1) = 2\pi/k$.
    If the latter, then since $X_0(2\pi,1)=2\pi/k$, there exists $x \in (0,2\pi)$ such that $X_0(x,1) = 0$.

    If we have $x\in (0,2\pi)$ such that $X_0(x,1) = 2\pi/k$, then
    \begin{align*}
        X_0(4\pi - x,1) &= -X_0( x,1) +4\pi/k\\
        &= 2\pi/k\\
        &= X_0( x,1),
    \end{align*}
    and $Y_0(4\pi - x,1) = Y_0( x,1)$, therefore, $\sigma(x,4\pi - x) = 0$.
    If we have $x\in (0,2\pi)$ such that $X_0(x,1) = 0$, a similar argument shows that $\sigma(2\pi+x, 2\pi-x)=0$.
\end{proof}
\end{lem}

As a consequence of defining $\mathcal{U}$ in the way we have, many of the mappings we defined are bijective at solutions of \eqref{eqn:largeampstream2}.

\begin{lem}\label{lem:bijectiveInU}
Let $(\zeta,q) \in \mathcal{U}$ solve \eqref{eqn:largeampstream2}. The $S_i$, defined in \eqref{eqn:introduceXhatYhat:BigS} are bijective. Furthermore, the maps $(\hat X_i,\hat Y_i)$ defined by \eqref{eqn:introduceXhatYhat:0} and \eqref{eqn:introduceXhatYhat:1} exist and are conformal. Consequently the coordinate mappings $( X_i, Y_i)$ are bijective, as desired.
\begin{proof}
We prove the result in the $i=1$ case. The $i=0$ case follows from an almost identical but simpler argument.
By assumption $s_x(x,y)>-1$, meaning $(s+x)_x$ is positive everywhere, and so $S_1$ is injective.
Now, we can define
\begin{equation}\label{eqn:S1inverse}
    (\hat{X}_1,\hat{Y}_1) = (X_1,Y_1)\circ S_1^{-1},
\end{equation}
and note that this definition agrees with \eqref{eqn:introduceXhatYhat:1}.
Inserting this definition into \eqref{eqn:largeampstream2:upperCR1} and \eqref{eqn:largeampstream2:upperCR2} shows that $(\hat{X}_1,\hat{Y}_1)$ satisfies the Cauchy--Riemann equations.
Applying the chain rule to \eqref{eqn:introduceXhatYhat} gives
\begin{equation*}
    (s_x+1)^2((\hat X_{1x}\circ S_1)^2+(\hat Y_{1x}\circ S_1)^2) = (\chi_{1x}+1/k)^2+ \eta_{1x}^2.
\end{equation*}
Thus \eqref{eqn:defU:s} and \eqref{eqn:defU:gradX} imply that $|\nabla \hat X_i|$ and $|\nabla \hat Y_i|$ are never zero.

We now use the Darboux--Picard Theorem; see, for example \cite[Theorem~9.16]{Burckel:darbouxPicard}.
This states that if $D \in \C$ is open and bounded, and $\partial D$ is a Jordan curve, $f \colon \overline{D} \to \C$ is a continuous function, holomorphic in $D$, and injective on $\partial D$, then the image of $f$ is the inside of the Jordan curve $f(\partial D)$, and $f$ is a conformal map onto its image.

Here, let $D = [0,2\pi] \times [h_0,h_0+h_1]$, and $f=(\hat X_1 , \hat Y_1)$. The only hypothesis left to verify is that $f$ is injective on $\partial D$. Splitting $\partial D$ into its 4 edges, let
\begin{align*}
    C_1&= \{ (0,y) \mid y \in (h_0,h_0+h_1) \}\\
    C_2&= \{ (x,h_0) \mid x \in [0,2\pi] \}\\
    C_3&= \{ (2\pi,y) \mid y \in (h_0,h_0+h_1) \}\\
    C_4&= \{ (x,h_0+h_1) \mid x \in [0,2\pi] \}.
\end{align*}
We show that $f$ is injective on each of these components, and that the images of these components are pairwise disjoint, and so $f$ is injective.
Consider $f|_{C_1}$. 
By parity $\hat X_1(0,y)=0$, which yields $0=\hat X_{1y}(0,y) = \hat Y_{1x}(0,y)$. 
It was observed earlier that $|\nabla \hat Y_1| \neq 0$, so $\hat Y_{1y}(0,y)$ can never be 0. 
Thus, $f(C_1)=\{0\} \times (Y_*,1)$ for some $Y_* \in (0,1)$, and $f|_{C_1}$ is injective.
Similarly, using periodicity, we see that $f(C_3)=\{2\pi/k\} \times (Y_{**},1)$ for some $Y_{**} \in (0,1)$, and $f|_{C_3}$ is injective.
The boundary condition \eqref{eqn:largeampstream2:Ytop} gives $f(C_4)=[0,2\pi/k] \times \{1\}$, and $f|_{C_4}$ is injective. 
These are all pairwise disjoint.
Condition \eqref{eqn:defU:selfIntersect} shows that the interface $f(C_2)$ is a curve which does not intersect itself, so we again use the fact $|\nabla \hat X_1| \neq 0$ to deduce that $f$ is injective on $C_2$. 
Condition \eqref{eqn:defU:wallIntersect} shows that $f(C_2)$ does not intersect $f(C_4)$. 
Parity, periodicity and \eqref{eqn:defU:selfIntersect} also give that $f(C_2)$ does not intersect $f(C_1)$ or $f(C_3)$, as an intersection here guarantees a self-intersection of the interface.

Therefore, $f$ is injective on $\partial D$, and $f(\partial D) = \partial \{ (X,Y) \in \Omega_1 \mid 0<X<2\pi/k\}$ so we apply Darboux--Picard to conclude that $(\hat X_1, \hat Y_1)$ gives a conformal equivalence between $D$ and $ \{ (X,Y) \in \Omega_1 \mid 0<X<2\pi/k\}$. 
It is then easy to show $(\hat X_1, \hat Y_1)$ also gives a conformal equivalence between $\mathcal{D}$ and $\Omega_1$.
\end{proof}
\end{lem}

\subsection{An abstract global bifurcation theorem}

In order to find large amplitude solutions, we use a theorem of global analytic bifurcation due to Buffoni and Toland \cite{Buffoni--Toland:Book}, based on work by Dancer \cite{DancerE.N.1973BTfA}. 
The proof of this abstract theorem combines results on real-analytic varieties with the implicit function theorem, both directly and in the form of Lyapunov--Schmidt reductions. 
This allows us to construct a curve of solutions which admits a local $\R$-analytic parametrisation. 
We follow this curve, using the theory of real-analytic varieties to move through any self-intersections, until the curve terminates, either by forming a closed loop, or through a combination of approaching $\partial \mathcal{U}$ and infinity.

\begin{lem}
    The set $\mathcal{U}$ is open, and $\mathcal{F} \colon \mathcal{U} \to \mathcal{Y}$ is $\R$-analytic.
\begin{proof}
    We first show $\mathcal{U}$ is open. The first five conditions \eqref{eqn:defU:h}--\eqref{eqn:defU:wallIntersect} clearly give rise to an open set, but the last one \eqref{eqn:defU:selfIntersect} is a little more subtle.
    The map
    \begin{align*}
        L \colon &\mathcal{X} \times \R \to C^0([0,4\pi]^2 )\\
        &(\zeta,q) \mapsto \sigma
    \end{align*}
    is bounded and linear, thus continuous. The set of non-vanishing functions is open in $C^0([0,4\pi]^2 )$, so its pre-image under $L$ is open also. Thus $\mathcal{U}$ is the intersection of six open sets, and therefore open.

    We now sketch a proof that $\mathcal{F} \colon \mathcal{U} \to \mathcal{Y}$ is $\R$-analytic. 
    It is straightforward to check that the image $\mathcal{F}(\mathcal{X} \times \R)$ is indeed contained in $\mathcal{Y}$.
    Analyticity then follows from the fact that $\mathcal{F}$ is a composition of polynomials and bounded linear maps, namely differentiation, integration, and the trace map.
\end{proof}
\end{lem}

We now state a minor adaptation of a theorem due to Buffoni and Toland, \cite[Theorem~9.1.1]{Buffoni--Toland:Book}.

\begin{thm}\label{thm:buffoniToland}
    Let $\mathcal{X}$, $\mathcal{Y}$ be Banach spaces over $\R$, let $\mathcal{U}$ be open in $\mathcal{X} \times \R$, and let $\mathcal{F} \colon \mathcal{U} \to \mathcal{Y}$ be $\R$-analytic. Suppose
    \begin{enumerate}[label=\rm(\roman*)]
        \item \label{largeamphypothesis1} There exists a trivial solution $\zeta_*$, and a non-empty, open interval $I \subseteq \R$ such that for all $q \in I$, we have $(\zeta_*, q) \in \mathcal{U}$ and $\mathcal{F}(\zeta_*,q)=0$.
        \item \label{largeamphypothesis2} For all $(\zeta,q) \in \mathcal{U}$ with $\mathcal{F}(\zeta,q)=0$, the operator $D_\zeta \mathcal{F}(\zeta,q)$ is Fredholm, and $D_\zeta \mathcal{F}(\zeta_*,q)$ is Fredholm of index 0 for some $q \in I$.
        \item \label{largeamphypothesis3} There exist $q_* \in I$ and $\dot{\rho}_* \neq 0$ such that
        \begin{subequations}\label{eq:CrandallRabinowitz}
\begin{align}
    \label{eq:CrandallRabinowitz:1Dkernel}
    \ker(D_\zeta \mathcal{F}(\zeta_*,q_*))&= \Span( \dot{\rho}_*)\\
    \label{eq:CrandallRabinowitz:transversality}
    D_{q}D_{\zeta} \mathcal{F}(\zeta_*,q_*)\dot{\rho}_* &\notin \Range( D_\zeta \mathcal{F}(\zeta_*,q_*)).
\end{align}
\end{subequations}
\end{enumerate}
Then there exists a continuous function $(Z,Q) \colon (-\infty, \infty) \to \mathcal{X} \times \R$ such that $Z(0)=\zeta_*$, $Z'(0)=\dot \rho$, $Q(0)=q_*$, and $\mathcal{F}(Z(\tau),Q(\tau))=0$ for all $\tau \in \R$. The function $(Z,Q)$ has a local $\R$-analytic reparametrisation around any $\tau$. Furthermore, one of the following occurs.
\begin{enumerate}[label=\rm(\alph*)]
    \item \label{alternative:blowup} For any compact set $K \subset \mathcal{U}$, there exists $T>0$ such that for all $\tau>T$, $(Z(\tau),Q(\tau)) \notin K$.
    \item \label{alternative:loop} $(Z,Q)$ forms a closed loop. In other words, there exists $T>0$ such that for all $\tau$ we have $(Z(\tau+T),Q(\tau+T))=(Z(\tau),Q(\tau))$.
\end{enumerate}
\end{thm}
\begin{rk}
This differs superficially from \cite[Theorem~9.1.1]{Buffoni--Toland:Book} in that we do not assume the interval $I$ is the whole real line, or that the trivial solution $\zeta_*=0$.
Additionally, hypothesis~\ref{largeamphypothesis2} replaces the stronger assumption that $D_\zeta \mathcal{F}(\zeta,q)$ is Fredholm of index 0 at all solutions, as in \cite{haziotWheeler}.
We also allow $Q'\equiv 0$ on a neighbourhood of 0, as in \cite{csv:global}.
Finally, our statement of \ref{alternative:blowup} gives a blow-up condition phrased in terms of compact sets rather than the norm and distance to $\partial \mathcal{U}$ of $(Z,Q)$. 
In this sense, loss of compactness for the bifurcation curve appears as an alternative rather than as a hypothesis, in the spirit of \cite{ChenWalshWheeler:StratifiedSolitaryWaves}.
\end{rk}
\begin{rk}
Recall the definition of $\zeta_*$ from \eqref{eqn:defZetaStar}. Letting $I = (-\frac 12 \omega_0 H, \infty)$, we see that $(\zeta_*,q) \in \mathcal{U}$ for all $q \in I$, and $\mathcal{F}(\zeta_*,q)=0$.
Thus, hypothesis~\ref{largeamphypothesis1} is satisfied.
\end{rk}

For completeness, we outline a proof of Theorem~\ref{thm:buffoniToland} in Appendix~\ref{sec:proofOfGlobBif}.

\section{Fredholm Properties}\label{sec:fredholmProperties}

The key functional analytic ingredient of the general bifurcation  Theorem~\ref{thm:buffoniToland} is that the linearised operators are Fredholm as maps between the relevant Banach spaces. In this section we relate this to the theory of elliptic systems  and hence verify hypothesis~\ref{largeamphypothesis2} of Theorem~\ref{thm:buffoniToland}.
Subsection~\ref{sec:lopatinskii} is dedicated to proving that $D_{\zeta^\circ}\mathcal{F}^\circ \colon \mathcal{X}^\circ_\np \to \mathcal{Y}^\circ_\np$
is an elliptic operator, while Subsection~\ref{sec:Fredind)} deals with its Fredholm properties.  
Lemma~\ref{lem:FredholmAndParity} then shows that ellipticity implies $D_{\zeta^\circ} \mathcal{F}^\circ \colon \mathcal{X}^\circ \to \mathcal{Y}^\circ$ is Fredholm.
Its index is found to be 0, and in Lemma~\ref{lem:finiteExtensionOfFredholm} we show that the same must therefore be true for $D_\zeta \mathcal{F} \colon \mathcal{X} \to \mathcal{Y}$.

Having linearised in $\zeta$ around some solution $(\zeta,q)$ of \eqref{eqn:largeampstream2}, we call our linearised variables $\dot{u_0},\ldots,\dot{h}_1$. Let $\dot{\zeta}= (\dot{u_0}, \ldots, \dot{h}_1)$, and let $\dot{\zeta}^\circ=(\dot{u_0}, \ldots, \dot{s})$.
We show that $D_{\zeta^\circ}\mathcal{F}^\circ \colon \mathcal{X}^\circ_\np \to \mathcal{Y}^\circ_\np$ is Fredholm by showing that it is elliptic.
The theory of elliptic systems is more complex than that for scalar equations, and many of the most useful results from scalar theory, such as the maximum principle, do not have analogues for systems.
Nevertheless, we still have access to Schauder estimates for elliptic systems due to Agmon, Douglis and Nirenberg \cite{ADN:prequel,ADN}.
We verify the hypotheses of these results following the presentation in \cite{Volpert:Book}; another good reference is \cite{wrl:elliptic}.
Throughout this section, lengthy calculations such as factoring the determinants of large matrices and calculating the residues of complicated rational functions are done with the aid of computer algebra.

\subsection{Ellipticity and the Lopatinskii condition}\label{sec:lopatinskii}

Ellipticity requires various hypotheses to be satisfied, as laid out for instance in \cite[Chapter 1, Section 2.2.1]{Volpert:Book}, which we verify here.
Given a differential operator such as $D_{\zeta^\circ} \mathcal{F}^\circ(\zeta,q)$ acting on a vector valued function $ \dot{\zeta}^\circ$, we first choose \emph{Douglis--Nirenberg numbers}, or  \emph{DN numbers}, associated with the regularity of components of $\dot{\zeta}^\circ$, and of  $D_{\zeta^\circ} \mathcal{F}^\circ(\zeta,q) \dot{\zeta}^\circ$.
These define a \emph{principal symbol} of the operator.
Roughly speaking, this is a way of finding the ``leading order" terms of  $D_{\zeta^\circ} \mathcal{F}^\circ(\zeta,q) $, and replacing the derivatives with Fourier variables.
Note, finding these ``leading order" terms is more subtle than only keeping the terms of highest order.
We expand on this later.
The operator is shown to be \emph{elliptic in the Douglis--Nirenberg sense}, meaning that if the Fourier variables are real valued and not all 0, the determinant of the principal symbol is non-zero.
In order to show \emph{proper ellipticity}, we show that if a particular Fourier variable is allowed to take on complex values, then the determinant of the principal symbol has the same number of roots in the upper and lower half-planes.
Finally, the \emph{Lopatinskii condition} for the boundary conditions is verified.
This requires a certain contour integral taken around half of these roots to be bounded away from 0.

\begin{prop}\label{prop:Lopatinskii}
    Let $(\zeta,q) \in \mathcal{U}$ solve \eqref{eqn:largeampstream2}.
    Then $D_{\zeta^\circ}\mathcal{F}^\circ(\zeta,q) \colon \mathcal{X}^\circ_\np \to \mathcal{Y}^\circ_\np$ is an \emph{L-elliptic} operator, in the sense of \cite[Definition~9.30]{wrl:elliptic}.
\begin{rk}
    In \cite{Volpert:Book}, L-ellipticity is known as \emph{ellipticity in the Agmon--Douglis--Nirenberg sense}.
\end{rk}
\begin{proof}[Proof of Proposition~\ref{prop:Lopatinskii}]
We give a matrix formula for $D_{\zeta^\circ} \mathcal{F}^\circ(\zeta,q) \dot{\zeta}^\circ$, that is, we find matrices of differential operators $A$, $B_0$, and $B_1$ such that $D_{\zeta^\circ} \mathcal{F}^\circ(\zeta,q) \dot{\zeta}^\circ = (A,B_0,B_1)\dot{\zeta}^\circ$.
The matrix $A$ corresponds to  the equations in the bulk \eqref{eqn:largeampstream2:lowerdivfree}--\eqref{eqn:largeampstream2:lap}, $B_0$ to the boundary conditions at $y=0$ \eqref{eqn:largeampstream2:kinbot}--\eqref{eqn:largeampstream2:sTop}, and $B_1$ to the boundary conditions at the interface \eqref{eqn:largeampstream2:ucts}--\eqref{eqn:largeampstream2:kinint}.
Let
\[A(x,y,\partial_x,\partial_y) = \begin{pmatrix}
M_0 & & 0\\
0 & & M_1
\end{pmatrix},
\]
where
\begin{align*}
    M_0 &= \begin{pmatrix}
 h_0 \partial_x & &\partial_y & &  0  & & -\omega_0 h_0 \partial_x
\\
 \partial_y & &-h_0 \partial_x & &  0 & & -\omega_0 \partial_y
\\
 0 & & 0 & & k h_0 \partial_x  & & -\partial_y
\\
 0 & & 0 & & \partial_y  & & h_0 \partial_x
\end{pmatrix}
\\
M_1&=\begin{pmatrix}
\eta_{1y} & -\chi_{1y} & v_{1y} & -u_{1y} & 0
\\
 -\chi_{1y} & -\eta_{1y} & -{\omega_1} \eta_{1y}+u_{1y} & {\omega_1} \chi_{1y}+v_{1y} & 0
\\
 0 & 0 & {h_1} & -s_y & \eta_{1y}
\\
 0 & 0 & s_y & {h_1} & -\chi_{1y}
\\
 0 & 0 & 0 & 0 & k^{-1} \partial_x
\end{pmatrix} k \partial_x\\
&\quad+\begin{pmatrix}
-k \eta_{1x} & k \chi_{1x}+1 & -k v_{1x} & k u_{1x} & 0
\\
 k \chi_{1x}+1 & k \eta_{1x} & k \left({\omega_1} \eta_{1x}-u_{1x}\right) & -k {\omega_1} \chi_{1x}-k v_{1x}-{\omega_1} & 0
\\
 0 & 0 & 0 & k \left(1+s_x\right) & -k \eta_{1x}
\\
 0 & 0 & -k \left(1+s_x\right) & 0 & k \chi_{1x}+1
\\
 0 & 0 & 0 & 0 & \partial_y
\end{pmatrix}\partial_y.
\end{align*}
Let
\[
B_0(x,\partial_x,\partial_y)=\begin{pmatrix}
0 & 1 & 0 & 0 & 0 & 0 & 0 & 0 & 0 \\
0 & 0 & 0 & 1 & 0 & 0 & 0 & 0 & 0 \\
0 & 0 & 0 & 0 & 0 & 1 & 0 & 0 & 0 \\
0 & 0 & 0 & 0 & 0 & 0 & 0 & 1 & 0 \\
0 & 0 & 0 & 0 & 0 & 0 & 0 & 0 & 1 \\
\end{pmatrix} ,
\]
 and
\[
B_1(x,\partial_x,\partial_y)=\left(\begin{matrix}
1 & 0 & 0 & 0 & -1 & 0 & 0 & 0 & 0 \\
0 & 1 & 0 & 0 & 0 & -1 & 0 & 0 & 0 \\
0 & 0 & 1 & 0 & 0 & 0 & -1 & 0 & 0 \\
0 & 0 & 0 & 1 & 0 & 0 & 0 & -1 & 0 \\
-\eta_{0x} & k^{-1} + \chi_{0x} & v_0 \partial_x & -(u_0+q) \partial_x & 0 & 0 & 0 & 0 & 0 \\
\end{matrix}\right),  \]
where the functions in the matrix $A$ are evaluated at $(x,y)$, and those in 
$B_1$ are evaluated at $(x,1)$.

As DN numbers we take
\begin{alignat*}{2}
    (s_1, \ldots, s_9)&=(-1,\ldots, -1,0)\\
    (t_1, \ldots, t_9)&=(2,\ldots, 2)\\
    (\sigma_1, \ldots, \sigma_5)&=(-2,\ldots,-2, -2)  & \qquad & \text{ at } y=0\\
    (\sigma_1, \ldots, \sigma_5)&=(-2,\ldots, -2, -1) & \qquad & \text{ at } y=1.
\end{alignat*}
This choice of DN numbers motivated our definition of $\mathcal{X}^\circ_\np$ and $\mathcal{Y}^\circ_\np$, with the $t_i$ corresponding the regularity of $\mathcal{X}^\circ_\np$, and the $s_i$ and $\sigma_i$ corresponding to the regularity of $\mathcal{Y}^\circ_\np$.

The order of the operator, given by $\sum (s_i + t_i)$, is 10. 
This is twice the number of boundary conditions, as required. 
Note, although it might appear there are ten boundary conditions, we in fact only have five, as each of the equations \eqref{eqn:largeampstream2:kinbot}--\eqref{eqn:largeampstream2:kinint} only covers $y=0$ or $y=1$, i.e., one of the two boundary components.
We define matrices $\hat{A}$, $\hat{B}_0$, $\hat{B}_1$, giving the principal symbols of $A$, $B_0$, and $B_1$.
If $A_{ij}$ is written in the form \[A_{ij}(x,y,\partial_x,\partial_y)=\sum_{|\alpha| \leq s_i+t_j} a^{ij}_\alpha(x,y) \partial^\alpha,\] 
where $\alpha$ is a multi-index, then let
\[ \hat{A}_{ij}(x,y,\xi,\nu)=\sum_{|\alpha| = s_i+t_j} a^{ij}_\alpha(x,y) (\xi,\nu)^\alpha.\]
Similarly, we define $\hat{B}_0$ and $\hat{B}_1$.
With our choice of DN numbers, no terms of $A$ or $B_0$ are eliminated, while $B_1$ loses only its $(5,1)$ and $(5,2)$ elements.

We next take the determinant of $\hat A$. This can be simplified by noting that we did not linearise around an arbitrary point in function space, but at a solution to \eqref{eqn:largeampstream2}. Thus \eqref{eqn:largeampstream2:lowerdivfree}--\eqref{eqn:largeampstream2:upperCR2} are used to eliminate various $y$ derivatives, and ultimately deduce
\begin{subequations}
\begin{equation}
\label{eqn:bigDet}
    \det(\hat A)=\frac{k^{3} (\nu^2+\xi^2) (\nu^2+h_0^2 \xi^2)^2  (h_1^2 \xi^2+(\nu (s_x+1)-\xi s_y)^2)^2  ((k  \chi_{1x}+1)^2+k^2 \eta_{1x}^2)}{(1+s_x)^2}.
\end{equation}

For $(\xi,\nu) \in \R^2 \setminus \{ 0 \}$, this determinant is strictly positive. Most of the factors in this determinant are straightforward to bound below, but $h_1^2 \xi^2+(\nu (s_x+1)-\xi s_y)^2$ is a little harder.
Note that \[\det(\hat{A}(x,y,\lambda \xi, \lambda \nu)) = \lambda^{10} \det(\hat{A}(x,y, \xi, \nu) ),\]
so it suffices to consider $\xi=\cos \theta$, $\nu=\sin \theta$ with $\theta \in \R$.
This substitution yields
\begin{align*}
    & h_1^2 \xi^2+((s_x+1)\nu -s_y\xi )^2 \\
    &\qquad =h_1^2 \cos^2 \theta+((s_x+1)\sin \theta -s_y\cos \theta )^2\\
    &\qquad =\frac{h_1^2+(s_x+1)^2+s_y^2}{2} - (s_x+1) s_y \sin 2 \theta+\frac{h_1^2-(s_x+1)^2+s_y^2}{2} \cos 2 \theta\\
    &\qquad \geq \frac 12 \left(h_1^2+(s_x+1)^2+s_y^2 - \sqrt{(h_1^2-(s_x+1)^2+s_y^2)^2+ 4(s_x+1)^2 s_y^2 } \right)\\
    &\qquad = \frac{2 h_1^2(s_x+1)^2}{h_1^2+(s_x+1)^2+s_y^2 +\sqrt{(h_1^2-(s_x+1)^2+s_y^2)^2+ 4(s_x+1)^2 s_y^2 }  }.
\end{align*}
Thus, $\det(\hat{A})$ is bounded below by
\begin{equation}
\label{eqn:bigDetLowerBound}
    \frac{4 k^3 \min(1,h_0^4)  h_1^4(s_x+1)^2 ((k \chi_{1x} + 1)^2 + k^2 \eta_{1x}^2)}{\left(h_1^2+(s_x+1)^2+s_y^2 +\sqrt{(h_1^2-(s_x+1)^2+s_y^2)^2+ 4(s_x+1)^2 s_y^2 }\right)^2}(\xi^2+\nu^2)^5.
\end{equation}
\end{subequations}
Therefore, $A$ is elliptic in the sense of Douglis--Nirenberg.

If we allow $\nu$ to take values in $\C$, then $\det(\hat A)= 0$ for
\[ \nu \in \left\{ i \xi ,-i \xi ,i \xi  h_0,-i \xi  h_0,\frac{\xi  s_y}{1+s_x}+\frac{i h_1 \xi}{1+s_x},\frac{\xi  s_y}{1+s_x}-\frac{i h_1 \xi}{1+s_x}\right\}. \]
The roots $i \xi$ and $-i\xi$ are simple, while the rest have multiplicity 2.
None of these are real, and there are the same number of roots in the upper and lower half planes.
Therefore we have proper ellipticity.

Finally consider the Lopatinskii condition. Let
\begin{align*}
    \Lambda_0(x,\xi)&= \frac{1}{2\pi i}\int_C  \hat{B}_0 (x,\xi,\nu) \hat{A}(x,0,\xi,\nu)^{-1} \ d \nu\\
    \Lambda_1(x,\xi)&= \frac{1}{2\pi i}\int_C  \hat{B}_1 (x,\xi,\nu) \hat{A}(x,1,\xi,\nu)^{-1} \ d \nu ,
\end{align*}
where $C$ is a simple closed contour in the upper half plane, enclosing the zeros of $\det(\hat A)$ with positive imaginary part, that is
\begin{align}\label{eqn:zerosInUpperHalfPlane}
    \left\{ i \xi ,i \xi  h_0,\frac{\xi  s_y}{1+s_x}+\frac{i h_1 \xi}{1+s_x}\right\},
\end{align}
where we have assumed $\xi>0$ without loss of generality.
In \cite{Volpert:Book} there is an extra term in the integrands, but for our purposes it is unnecessary.
These integrals are calculated using the residue theorem and computer algebra.
When performing this calculation, it is important to use the fact we linearised around a solution, as this affects the order and residues of the poles found at points in \eqref{eqn:zerosInUpperHalfPlane}.
We now check that the Lopatinskii constants $ e^i_{\partial \mathcal{D}}$ are non-zero.
These are defined as
\[ e^i_{\partial \mathcal{D}} = \inf_{x \in \T, \ |\xi|=1} \sum_\alpha |\mu_\alpha^i(x,\xi)|,  \]
where the $\mu^i_\alpha$ are the $5 \times 5$ minors of $\Lambda_i$ and $\alpha$ is an indexing variable.

For each of the $\Lambda_i$, consider the minor consisting of the odd numbered columns. We call this minor $\mu^i_{\alpha^*}$ and calculate that
\begin{align*}
    \det(\mu^0_{\alpha^*}(x,\xi))&=\frac{1}{32 k^{2} \left|1+k \chi_{1x}(x,0)\right||\xi|}\\
    \det( \mu^1_{\alpha^*}(x,\xi))&= \frac{\sqrt{(u_0(x,1)+q)^2+v_0(x,1)^2}}{16 \left(1+s_x(x,1)\right)^{2} k^{3}}  ,
\end{align*}
which, due to the definition of $\mathcal{U}$, give bounds away from 0 for $e^0_{\partial \mathcal{D}}$ and $e^1_{\partial \mathcal{D}}$ which are uniform in $x$.
\end{proof}
\end{prop}

\subsection{Fredholm index 0}\label{sec:Fredind)}
By combining Proposition~\ref{prop:Lopatinskii},  \cite[Theorem~9.32]{wrl:elliptic}, and \cite[Theorem~8.3 in Chapter~5]{Volpert:Book} we have for any $(\zeta,q)$ solving \eqref{eqn:largeampstream2} and any $\gamma \in (0,1)$, that
\begin{align*}
    D_{\zeta^\circ}\mathcal{F}^\circ(\zeta,q) \colon \mathcal{X}^\circ_\np \to \mathcal{Y}^\circ_\np 
\end{align*}
is Fredholm, where we recall the definitions of $\mathcal{X}^\circ_\np $ and $ \mathcal{Y}^\circ_\np $ from \eqref{eqn:defineNonParitySpaces}. 
This result is why $\mathcal{X}^\circ_\np $ and $ \mathcal{Y}^\circ_\np $ were chosen to be these particular products of Hölder spaces. 
Observe that the domain has its regularity given by the $t_i$, and the codomain has its regularity given by $-s_i$ for the factors corresponding to the bulk, and $-\sigma_i$ for those corresponding to the boundary.

We now show that this operator remains Fredholm when restricted to $ \mathcal{X}^\circ $ and $ \mathcal{Y}^\circ $.

\begin{lem}\label{lem:FredholmAndParity}
    Recall $\mathcal{X}^\circ$ and $\mathcal{Y}^\circ$ from \eqref{eqn:defineXcirc} and the following discussion.
    For $(\zeta,q) \in\mathcal{U}$ solving \eqref{eqn:largeampstream2}, the operator $ D_{\zeta^\circ}\mathcal{F}^\circ(\zeta,q) \colon \mathcal{X}^\circ \to \mathcal{Y}^\circ$ is Fredholm.
\begin{proof}
    In order to distinguish between different realisations of  $ D_{\zeta^\circ}\mathcal{F}^\circ(\zeta,q)$, with their different domains and codomains, define $T$ and $L$ as
    \begin{align*}
        T&= D_{\zeta^\circ}\mathcal{F}^\circ(\zeta,q) \colon   \mathcal{X}^\circ_\np \to \mathcal{Y}^\circ_\np \\
        L &= D_{\zeta^\circ}\mathcal{F}^\circ(\zeta,q) \colon \mathcal{X}^\circ \to \mathcal{Y}^\circ.
    \end{align*}
     We know that $T$ is Fredholm, and are trying to show that $L$ is Fredholm. Since $\ker(L) \subseteq \ker(T)$, we see $\ker(L)$ must be finite dimensional.

     Now consider the codimension of the image.
     Let $\check{\mathcal{X}}^\circ$ be the subset of $ \mathcal{X}^\circ_\np$ with the opposite parities to $\mathcal{X}^\circ$, i.e., the 1st, 4th, 5th, and 8th components are odd in $x$, and the rest are even.
     Similarly, $\check{\mathcal{Y}}^\circ$ denotes the subset of $\mathcal{Y}^\circ_\np $ with the opposite parities to $\mathcal{Y}^\circ$. 
     We have that $\mathcal{X}^\circ_\np = \mathcal{X}^\circ \oplus \check{\mathcal{X}}^\circ$, and $\mathcal{Y}^\circ_\np = \mathcal{Y}^\circ \oplus \check{\mathcal{Y}}^\circ$.

    Let $P_{\mathcal{X}^\circ}$ be the obvious projection onto $\mathcal{X}^\circ$, i.e.,
     \[ P_{\mathcal{X}^\circ} \colon (\dot{u}_0(x,y),\ldots,\dot{s}(x,y)) \mapsto \tfrac 12 (\dot{u}_0(x,y)+\dot{u}_0(-x,y),\ldots,\dot{s}(x,y)-\dot{s}(-x,y)).\]
    Note that $P_{\check{\mathcal{X}}^\circ}=I-P_{\mathcal{X}^\circ}$ gives a projection onto $\check{\mathcal{X}}^\circ$.
    Projections $P_{\mathcal{Y}^\circ}$ and $P_{\check{\mathcal{Y}}^\circ}$ can be defined similarly.
    Notice firstly that
     \[ P_{\mathcal{Y}^\circ}T \dot{\zeta}^\circ=P_{\mathcal{Y}^\circ}T(P_{\mathcal{X}} \dot{\zeta}^\circ) + P_{\mathcal{Y}^\circ}T(P_{\check{\mathcal{X}}^\circ} \dot{\zeta}^\circ) = T(P_{\mathcal{X}} \dot{\zeta}^\circ)= L(P_{\mathcal{X}} \dot{\zeta}^\circ) .\]
    We want to show that $L(\mathcal{X}^\circ)$ has finite codimension in $\mathcal{Y}^\circ$.
    The fact that $T$ is Fredholm means there exist $\alpha_1,\ldots,\alpha_n$ such that $\mathcal{Y}^\circ_\np = T(\mathcal{X}^\circ) \oplus T(\check{\mathcal{X}}^\circ) \oplus \Span( \alpha_1,\ldots,\alpha_n )$.
    Therefore, \begin{align*}
        \mathcal{Y}^\circ &=P_{\mathcal{Y}^\circ} T(\mathcal{X}^\circ) + P_{\mathcal{Y}^\circ}T(\check{\mathcal{X}}^\circ) + P_{\mathcal{Y}^\circ}(\Span( \alpha_1,\ldots,\alpha_n ))\\
        &= L(P_{\mathcal{X}^\circ}\mathcal{X}^\circ) + L(P_{\mathcal{X}^\circ}\check{\mathcal{X}}^\circ) + \Span( P_{\mathcal{Y}^\circ}\alpha_1,\ldots,P_{\mathcal{Y}^\circ}\alpha_n )\\
        &= L(\mathcal{X}^\circ) + \Span( P_{\mathcal{Y}^\circ}\alpha_1,\ldots,P_{\mathcal{Y}^\circ}\alpha_n ),
    \end{align*}
    and the result is proved.
\end{proof}
\end{lem}

We now show that $D_\zeta \mathcal{F}(\zeta,q) \colon \mathcal{X} \to \mathcal{Y}$ is also Fredholm, and in fact has the same index as $D_{\zeta^\circ} \mathcal{F}^\circ (\zeta,q) \colon \mathcal{X}^\circ \to \mathcal{Y}^\circ$.  
Recall the definitions of $\mathcal{F}^\circ$ and $\mathcal{F}^s$ from the discussion following \eqref{eqn:defineMathcalF}.
Observe that
\begin{align}
    \label{eqn:D_zetaF matrix}
D_\zeta \mathcal{F} = \begin{pmatrix}
    D_{\zeta^\circ} \mathcal{F}^\circ & D_{(h_0,h_1)} \mathcal{F}^\circ \\
    D_{\zeta^\circ} \mathcal{F}^s & D_{(h_0,h_1)} \mathcal{F}^s
\end{pmatrix}.
\end{align}
Intuitively, every entry of the matrix in \eqref{eqn:D_zetaF matrix} except the upper-left is finite-dimensional in some sense, and so should not affect the Fredholm property. We make this rigorous with the following ``bordering" lemma, a proof of which can be found in \cite[Lemma 2.3]{Beyn:borderinglemma}.

\begin{lem}\label{lem:finiteExtensionOfFredholm}
    Let $X$, $Y$ be general Banach spaces, let $L_1 \colon X \to Y$ be Fredholm, and let $L \colon X \times \R^n \to Y \times \R^m$ be the linear map given by
    \[ L = \begin{pmatrix}
        L_1 & L_2\\
        L_3 & L_4
    \end{pmatrix}.\]
    Then $L$ is Fredholm with $\Ind(L)=\Ind(L_1)+n-m$.
\end{lem}

Applying Lemma~\ref{lem:finiteExtensionOfFredholm} to $D_{\zeta^\circ} \mathcal{F}^\circ(\zeta,q)$ verifies the first part of hypothesis~\ref{largeamphypothesis2} of Theorem~\ref{thm:buffoniToland}, namely that $D_\zeta \mathcal{F}(\zeta,q)$ is Fredholm. It remains to show that  $D_\zeta \mathcal{F}(\zeta_*,q)$ is Fredholm of index 0.
To this end, we homotope the functional part $D_{\zeta^\circ} \mathcal{F}^\circ(\zeta_*,q)$ to an invertible operator.
We have
\begin{align}
    \label{eqn:linearisedAtShear}
    D_{\zeta^\circ}\mathcal{F}^\circ (\zeta_*,q) \dot{\zeta}^\circ=(A,B_0,B_1)\dot{\zeta}^\circ,
\end{align}
where 
the operator matrix $A$ has the block structure
\begin{align*}
A&= \begin{pmatrix}
    M_0 &0 \\
    0 & M_1
\end{pmatrix}\\
\intertext{with}
M_0 &= \begin{pmatrix}
k H \partial_x  & \partial_y  & 0 & -k H \omega_0 \partial_x  \\
\partial_y  & -k H \partial_x  & 0 & -\partial_y  \omega_0  \\
0 & 0 & k^2 H \partial_x  & -k \partial_y \\
0 & 0 & \partial_y  & k H \partial_x
\end{pmatrix}\\
M_1 &= \begin{pmatrix}
-k \left(1-H\right) \partial_x  & \partial_y  & 0 & k \left(1-H\right) {\omega_1} \partial_x  & 0 \\
\partial_y  & k \left(1-H\right) \partial_x  & 0 & -\partial_y  {\omega_1} & 0 \\
0 & 0 & k^2 \left(1-H\right) \partial_x  & k \partial_y  & -k \left(1-H\right) \partial_x  \\
0 & 0 & -k \partial_y  & k^2 \left(1-H\right) \partial_x  & \partial_y  \\
0 & 0 & 0 & 0 & \partial_x^2+\partial_y^2
\end{pmatrix}
,
\\
\intertext{and the boundary operators are}
B_0 &= \begin{pmatrix}
0 & 1 & 0 & 0 & 0 & 0 & 0 & 0 & 0 \\
0 & 0 & 0 & 1 & 0 & 0 & 0 & 0 & 0 \\
0 & 0 & 0 & 0 & 0 & 1 & 0 & 0 & 0 \\
0 & 0 & 0 & 0 & 0 & 0 & 1 & 0 & 0 \\
0 & 0 & 0 & 0 & 0 & 0 & 0 & 1 & 0 \\
\end{pmatrix} 
\\
B_1&=\begin{pmatrix}
1 & 0 & 0 & 0 & -1 & 0 & 0 & 0 & 0 \\
0 & 1 & 0 & 0 & 0 & -1 & 0 & 0 & 0 \\
0 & 0 & 1 & 0 & 0 & 0 & -1 & 0 & 0 \\
0 & 0 & 0 & 1 & 0 & 0 & 0 & -1 & 0 \\
0 & k^{-1} & 0 & - \left(\frac{1}{2}\omega_0H+q\right) \partial_x  & 0 & 0 & 0 & 0 & 0
\end{pmatrix}.
\end{align*}
Note that no terms are lost when we take the principal symbol, except from $B_1$, which only loses the $(5,2)$ entry. Therefore, if we homotope this component to 0, we do not change the principal symbol, and thus $(A,B_0,B_1)$ remains Fredholm.

We now homotope to an invertible operator in order to find the Fredholm index.
Let
\begin{align*}
    \tilde{A}(t)&= \begin{cases}
        \hat{A}|_{\omega_0=(1-t) \omega_0, \ \omega_1=(1-t) \omega_1, \ H=H_t, \ k= k_t} & \quad 0 \leq t \leq 1\\
        \diag \left(
        \begin{pmatrix}
            \xi  & \nu \\
            \nu  & -\xi  
        \end{pmatrix},
        \begin{pmatrix}
            2 \xi  & -2 \nu  \\
            \nu  & \xi 
        \end{pmatrix},
        \begin{pmatrix}
            -\xi  & \nu  \\
            \nu  & \xi 
        \end{pmatrix},
        \begin{pmatrix}
             2 \xi  & 2 \nu  & -\xi (2-t) \\
             -2 \nu  & 2 \xi  & \nu (2-t) \\
             0 & 0 & \xi^2+\nu^2
        \end{pmatrix} 
        \right)
        & \quad 1 < t \leq 2
    \end{cases}
\end{align*}
with boundary symbols
\begin{align*}
    \tilde{B}_0(t)&= \hat{B}_0\\
    \tilde{B}_1(t)&= \begin{cases}
        \begin{pmatrix}
1 & 0 & 0 & 0 & -1 & 0 & 0 & 0 & 0
\\
 0 & 1 & 0 & 0 & 0 & -1 & 0 & 0 & 0
\\
 0 & 0 & 1 & 0 & 0 & 0 & -1 & 0 & 0
\\
 0 & 0 & 0 & 1 & 0 & 0 & 0 & -1 & 0
\\
 0 & 0 & 0 & -\left(\left(1-t\right) \left(\frac{1}{2}\omega_0 H+q\right)+t\right) \xi  & 0 & 0 & 0 & 0 & 0
\end{pmatrix} & \quad 0 \leq t \leq 1\\
        \begin{pmatrix}
2-t & 0 & 0 & 0 & -1 & 0 & 0 & 0 & 0
\\
 0 & 2-t & 0 & 0 & 0 & -2+t & 0 & 0 & t-1
\\
 0 & 0 & 2-t & 0 & 0 & 0 & -1 & 0 & 0
\\
 t-1 & 0 & 0 & 2-t & 0 & 0 & 0 & -2+t & 0
\\
 0 & 0 & 0 & -\xi  & 0 & 0 & 0 & 0 & 0
\end{pmatrix} & \quad 1 < t \leq 2,
    \end{cases}\\
\end{align*}
where
\[ k_t= k(1-t) + 2t, \quad H_t=\frac{(1-t) H k+t}{k_t}. \]
Note $k_t$ and $H_t$ are both positive for $t \in [0,1]$, and $k_1=2$, $H_1=\frac 12$.  
The operators depend continuously on $t \in [0,2]$. 
We show the system is elliptic for all such $t$, so is Fredholm for all such $t$, and by the continuity of the index \cite[Theorem~9.12]{wrl:elliptic}, the index is constant with respect to $t$.
The same steps as in Lemma~\ref{prop:Lopatinskii} are applied, and the same DN numbers chosen. We see
\begin{align*}
    \det( \tilde{A}(t)) &= \begin{cases}
        k_t^{3} \left(H_t^2 k_t^2 \xi^2+\nu^2\right)^2 \left(\nu^2+\xi^2\right) \left( k_t^2 \xi^2 \left(1-H_t\right)^2+\nu^2\right)^2 & \quad 0 \leq t \leq 1\\
        8(\xi^2+\nu^2)^5 & \quad 1 \leq t \leq 2.\\
    \end{cases}
\end{align*}
This is non-zero for $(\xi,\nu) \neq 0$, and can be bounded below by
\begin{align*}
    \det( \tilde{A}(t)) & \geq \begin{cases}
        k_t^{3} \min(1,  k_t^4 H_t^4) \min(1,  k_t^4(1- H_t)^4)  (\nu^2+\xi^2)^5 & \quad 0 \leq t \leq 1\\
        8(\xi^2+\nu^2)^5 & \quad 1 \leq t \leq 2.\\
    \end{cases}
\end{align*}
Allowing $\nu$ to take values in $\C$, the roots of this determinant are $\{ i\xi, i k_t H_t \xi, i k_t(1- H_t) \xi  \}$ for  $0 \leq t \leq 1$, and $\{i \xi\}$ for  $1 \leq t \leq 2$.
Calculating the Lopatinskii matrices, and considering the same minor as in Lemma~\ref{prop:Lopatinskii}, we find lower bounds for the Lopatinskii constants,
\begin{align*}
    e^0_{\partial \mathcal{D}}(t) &\geq \begin{cases}
       \dfrac{1}{32 k_t^{2}|\xi|} & \quad 0 \leq t \leq 1\\[2ex]
        \dfrac{1}{128|\xi|} & \quad 1 \leq t \leq 2\\
    \end{cases}  \\
    e^1_{\partial \mathcal{D}}(t) &\geq \begin{cases}
      \dfrac{(1-t) \left(\tfrac 12 \omega_0 H+q\right)+t}{16 k_t^3} & \quad 0 \leq t \leq 1\\[2ex]
      \dfrac{(2-t)^4+(2-t)^3+2 (1-t)^2}{256} & \quad 1 \leq t \leq 2.\\
    \end{cases}
\end{align*}
Consider the associated family of operators
\begin{align*}
    &\tilde {L}\colon [0,2] \times \mathcal{X}^\circ \to \mathcal{Y}^\circ \\
    &\tilde {L}(t)\dot{\zeta}^\circ = (\tilde A (t), \tilde{B}_0 (t), \tilde{B}_1 (t)) \dot{\zeta}^\circ,
\end{align*}
and note that $\tilde{L}(0)=D_{\zeta^\circ}\mathcal{F}(\zeta_*^\circ,q)$ . Under the homotopy achieved by sending $t$ to 2, the Lopatinskii condition is retained throughout.
Thus, using arguments from Lemma~\ref{lem:FredholmAndParity}, and the same arguments from \cite{Volpert:Book,wrl:elliptic} as before, for all $t \in [0,2]$, $\tilde {L }(t)$ is Fredholm with index independent of $t$.

The endpoint operator $\tilde{ L}(2)$
has interior symbol
\begin{align*}
    \tilde{A}(2) &= \diag \left( \begin{pmatrix}
        \xi & \nu\\
        \nu & -\xi
    \end{pmatrix}, \begin{pmatrix}
        2\xi & - 2\nu\\
        \nu & \xi
    \end{pmatrix}, \begin{pmatrix}
        -\xi & \nu\\
        \nu & \xi
    \end{pmatrix}, \begin{pmatrix}
        2\xi &  2\nu\\
        -2\nu &2 \xi
    \end{pmatrix}, \begin{pmatrix}
        \xi^2+\nu^2
    \end{pmatrix} \right)
    \intertext{and boundary symbols}
    \tilde{B}_0 (2)= \hat{B}_0 &= \begin{pmatrix}
0 & 1 & 0 & 0 & 0 & 0 & 0 & 0 & 0 \\
0 & 0 & 0 & 1 & 0 & 0 & 0 & 0 & 0 \\
0 & 0 & 0 & 0 & 0 & 1 & 0 & 0 & 0 \\
0 & 0 & 0 & 0 & 0 & 0 & 0 & 1 & 0 \\
0 & 0 & 0 & 0 & 0 & 0 & 0 & 0 & 1 \\
\end{pmatrix} \\
     \tilde{B}_1 (2)&=  \begin{pmatrix}
0 & 0 & 0 & 0 & -1 & 0 & 0 & 0 & 0 \\
0 & 0 & 0 & 0 & 0 & 0 & 0 & 0 & 1 \\
0 & 0 & 0 & 0 & 0 & 0 & -1 & 0 & 0 \\
1 & 0 & 0 & 0 & 0 & 0 & 0 & 0 & 0 \\
0 & 0 & 0 & \xi & 0 & 0 & 0 & 0 & 0 \\
\end{pmatrix}. \\
\end{align*}
\begin{lem}
     The operator corresponding to the symbol $\tilde{ L}(2) \colon \mathcal{X}^\circ \to \mathcal{Y}^\circ$ is invertible.
\begin{proof}
    We begin by showing the kernel is trivial. Suppose $\tilde{L}(2) \dot\zeta =0$.
    First note that $\Delta \dot s =0$, $\dot s$ is bounded, and $\dot s (x,0)=\dot s (x,1)=0$, so $\dot s = 0$.
    For $\dot{\chi}_0$, and $\dot{\eta}_0$, note that $\dot{\eta}_0$ is harmonic, bounded, and has boundary conditions which imply  $\dot{\eta}_0(x,0)=0$,  $\dot{\eta}_0(x,1)=\lambda$ for some $\lambda \in \R$.
    This yields $\dot{\eta}_0=\lambda y$, and so $\dot{\chi}_0=\lambda x + \mu$. However, $\dot{\chi}_0$ is bounded and odd, so $\lambda=\mu=0$, therefore $\dot{\chi}_0=\dot{\eta}_0=0$.

    When considering the other components, we use the argument principle.
    Let $c$ be a simple closed contour going around the boundary of $[-\pi, \pi] \times [0,1]$, and let $f(x+iy)=\dot{u}_0(x,y) - i\dot{v}_0(x,y)$.
    Note that $\dot{v}_0(\pi,y)=\dot{v}_0(-\pi,y)=0$, since $\dot{v}_0$ is odd and $2\pi$ periodic in $x$. This, with the boundary conditions, implies that the image of $f \circ c$ is contained in $\R \cup i \R$.
    Thus, for any $z \notin \R \cup i \R$, the winding number of $f \circ c$ around $z$ is 0.
    The Cauchy--Riemann equations are satisfied by $f$, so by the argument principle, the image of $f$ is contained in $\R \cup i \R$. By the open mapping theorem we see then that $f$ must be constant, and by the boundary conditions, this constant must be 0.
    Similar arguments show that $\dot{u}_1$, $\dot{v}_1$, $\dot{\chi}_1$, and $\dot{\eta}_1$ are also $0$.

    We now show $\tilde{L}(2)$ has full range. Since it is Fredholm, and thus has closed range, it suffices to show the range contains a dense subset $\mathcal{Y}^\circ$. The set of smooth functions with the correct parities is such a subset.

    The Poisson equation with homogenous Dirichlet conditions can be solved on $\mathcal{D}$, so $\dot{s}$ is determined. Now consider $\dot{\chi}_0$ and $\dot{\eta}_0$.
    We want to show that if $f$ and $a$ are even in $x$ and smooth, and $g$ and $b$ are odd in $x$ and smooth, then
    \begin{equation}
    \label{eq:FullRangeAimingToSolve1}
        \dot{\chi}_{0x} - \dot{\eta}_{0y} =f, \quad \dot{\chi}_{0y} +  \dot{\eta}_{0x} = g, \quad \dot{\eta}_0(x,0)=a, \quad \dot{\eta}_{0x}(x,1)=b,
    \end{equation}
    has a solution $(\dot{\chi}_1,\dot{\eta}_1) \in C^{2,\gamma}(\overline{\mathcal{D}})$, with $\dot{\chi}_1$ odd, and $\dot{\eta}_1$ even.
    Note that since $b$ is odd and periodic, it has mean $0$, so its antiderivative $\beta(x)= \int_0^x b(t) \ dt$  is even and periodic.
    Let $\varphi$ solve
    \begin{align*}
        \Delta \varphi &= g_x -f_y , & \varphi(x,0)&=a , & \varphi(x,1)&=\beta.
    \end{align*}
    We define
    \begin{align*}
        \lambda &=-\frac{1}{2\pi}\int_0^{2\pi} f(t,0)+\varphi_y(t,0) \ dt,
    \end{align*}
    then try the ansatz
    \begin{align*}
        \dot{\eta}_0(x,y) &= \varphi(x,y) + \lambda y\\
        \dot{\chi}_0(x,y) &= \int_0^y g(x,t) -\varphi_x(x,t) \ dt + \int_0^x f(t,0)+\varphi_y(t,0) + \lambda \ dt,
    \end{align*}
    and see that this $\dot{\chi}_0$ and $\dot{\eta}_0$ solve \eqref{eq:FullRangeAimingToSolve1}, are $2\pi$ periodic, and have the required parities.

    Now consider $\dot{u}_1$ and $\dot{v}_1$, and note that the cases for $\dot{u}_0$, $\dot{v}_0$, $\dot{\chi}_1$ and $\dot{\eta}_1$ follow a very similar argument.
    We want to show that if $f$ and $a$ are odd in $x$ and smooth, and $g$ and $b$ are even in $x$ and smooth, then
    \begin{equation}
    \label{eq:FullRangeAimingToSolve2}
        \dot{u}_{1x} - \dot{v}_{1y} =f, \quad \dot{u}_{1y} +  \dot{v}_{1x} = g, \quad \dot{v}_1(x,0)=a, \quad \dot{u}_1(x,1)=b,
    \end{equation}
    has a solution $(\dot{u}_1,\dot{v}_1) \in C^{2,\gamma}(\overline{\mathcal{D}})$, with $\dot{u}_1$ even, and $\dot{v}_1$ odd.
    Let $\dot{u}_1$ and $\dot{v}_1$ solve
    \begin{subequations}
    \begin{align*}
        \Delta \dot{u}_1 &= f_x + g_y & \dot{u}_{1y}(x,0)&=g(x,0)-a_x & \dot{u}_1(x,1)&=b\\
        \Delta \dot{v}_1 &= g_x -f_y & \dot{v}_1(x,0)&=a & \dot{v}_{1y}(x,1)&=b_x -f(x,1).
    \end{align*}
    \end{subequations}

    Using the parities of $f$ and $g$, we see that both $\dot{u}_1(x,y)-\dot{u}_1(-x,y)$ and $\dot{v}_1(x,y)+\dot{v}_1(-x,y)$ are bounded harmonic functions, with boundary conditions implying they are 0 everywhere. Thus $\dot{u}_1$ is even in $x$ and $\dot{v}_1$ is odd in $x$.
    Now, let
    \[F(x+iy)=(\dot{u}_{1y}(x,y) +  \dot{v}_{1x}(x,y) -g(x,y))+i(\dot{u}_{1x}(x,y) - \dot{v}_{1y}(x,y) -f(x,y)).\]
    We show that $F$ is identically zero using a similar argument to when we showed the kernel is trivial.
    The Cauchy--Riemann equations are satisfied by $F$, and if we restrict the domain of $F$ to the boundary of $[-\pi,\pi] \times [0,1]$, its image is contained within $\R \cup i\R$.
    We conclude as before that by the argument principle $F=0$, and thus \eqref{eq:FullRangeAimingToSolve2} is satisfied.
\end{proof}
\end{lem}
Since $\tilde{ L}(2)$ is invertible, it is of of index 0.
The Fredholm index is locally constant, and since there exists a homotopy between $\tilde{ L}(2)$ and $D_{\zeta^\circ}\mathcal{F}^\circ (\zeta_*,q)$ along which the Fredholm property is satisfied, we conclude that $D_{\zeta^\circ}\mathcal{F}^\circ (\zeta_*,q)\colon \mathcal{X}^\circ \to \mathcal{Y}^\circ$ is Fredholm of index 0.
Finally, we apply Lemma~\ref{lem:finiteExtensionOfFredholm} with $n=m=2$ to conclude that $D_{\zeta}\mathcal{F} (\zeta_*,q)\colon \mathcal{X} \to \mathcal{Y}$ is Fredholm of index 0, as required.

\section{Local bifurcation}\label{sec:localBifurcation}

In this section we construct a local branch of non-trivial solutions using the classical Crandall--Rabinowitz theorem \cite{cr:simple}.
The bifurcation point is found at a root of a dispersion relation, which agrees with the dispersion relation we find in our previous work \cite{ourpaper} for solitary waves. To this end, we now verify hypothesis~\ref{largeamphypothesis3} of Theorem~\ref{thm:buffoniToland} and find $q_*$ and $\dot{\rho}_*$ such that \eqref{eq:CrandallRabinowitz} holds.
We in fact find a countable family $q_n$ of all possible bifurcation points, and the corresponding vectors $\dot{\rho}_n$.

\begin{prop}
\label{prop:CrandallRabinowitzKernel}
    The linearised operator $D \mathcal{F}_\zeta(\zeta_*,q)$ has a non-trivial kernel if and only if $q=q_n=\mathfrak d (nk)$, where
\begin{align}
    \label{eqn:dispersionRelation}
    \mathfrak{d}(k)&=\frac{1}{ k \left(\coth ( k (1-H))+\coth ( kH)\right)}-\frac{1}{2}\omega_0 H,
\end{align}
and $n$ is a positive integer.
We call the equation $q = \mathfrak d (k)$ the \emph{dispersion relation}.
Note that from \eqref{eqn:dispersionRelation} it is clear that $\mathfrak d (nk)> -\frac 12 \omega_0 H$ for all $k$, therefore $q_n \in I$, where $I$ is the interval from Theorem~\ref{thm:buffoniToland}.
  At $(\zeta_*, q_n)$ the kernel of the linearised operator is the span of $\dot{\rho}_n$, where 
\begin{align}\label{eqn:kernelAtShear}
    \dot{\rho}_n &= \begin{pmatrix}
         (c_1(n) \cosh( nk H y)+\omega_0 \sinh(nk H y))\cos nx\\
         c_1(n) \sinh(nk H y) \sin nx\\
         \cosh(nk H y) \sin nx\\
         \sinh(nk H y) \cos nx\\
        (c_2(n)\omega_1\sinh(nk  (1-H) y) -c_4(n) \cosh(nk  (1-H)  y)) \cos nx\\
         c_4(n) \sinh(nk (1-H)y) \sin nx \\
         \big(c_3(n)k^{-1}\sinh (ny)-c_2(n) \cosh (nk  (1-H) y)\big) \sin nx  \\
         c_2(n) \sinh(nk (1-H)y) \cos nx\\
         c_3(n)\sinh(ny)   \sin nx\\
         0\\
         0
        \end{pmatrix},
\end{align}
    and the constants $c_i(n)$ are defined in \eqref{eqn:kernelConstants}.
\begin{proof}
    Suppose $D_\zeta \mathcal{F}(\zeta_*,q)\dot{\zeta}=0$.
    We first show that $\dot{h}_0=\dot{h}_1=0$.
    Linearising \eqref{eqn:largeampstream2:lowerCR1}, \eqref{eqn:largeampstream2:Ybot}, and \eqref{eqn:largeampstream2:avgDepth} around $\zeta_*$ yields
    \begin{subequations}
    \begin{align}
        \label{eqn:linearisedCR1}
        k^2 H \dot{\chi}_{0x}-k\dot{\eta}_{0y}+\dot{h}_0&=0\\
        \label{eqn:linearisedYbot}
        \dot{\eta}_{0}(x,0)&=0\\
        \label{eqn:linearisedAvgDepth}
        \frac{1}{2\pi}\int_0^{2\pi} \dot{\eta}_0(x,1) \ dx &=0.
    \end{align}
    \end{subequations}
    Linearising \eqref{eqn:largeampstream2:avgDepth} actually gives another term in the integrand, $ k H \dot{\chi}_{0x}(x,1)$, but this integrates to 0 by periodicity.
    Integrating \eqref{eqn:linearisedCR1} we find
    \begin{align*}
        0&= \int_{y=0}^1 \int_{x=0}^{2\pi} k^2 H \dot{\chi}_{0x}-k\dot{\eta}_{0y}+\dot{h}_0 \ dx \ dy\\
        &=2 \pi \dot{h}_0 - k\int_{0}^{2\pi} \dot{\eta}_{0}(x,1) - \dot{\eta}_{0}(x,0) \ dx.
    \end{align*}
    Applying \eqref{eqn:linearisedYbot} and \eqref{eqn:linearisedAvgDepth} yields $\dot{h}_0 =0$. A similar argument considering \eqref{eqn:largeampstream2:upperCR1}, \eqref{eqn:largeampstream2:Ytop}, and \eqref{eqn:largeampstream2:avgDepth} shows that $\dot{h}_1 =0$.

    We now need only consider $D_{\zeta^\circ} \mathcal{F}^\circ(\zeta_*,q)\dot{\zeta}^\circ=0$ and $D_{\zeta^\circ} \mathcal{F}^s(\zeta_*,q)\dot{\zeta}^\circ=0$.
    Equation \eqref{eqn:linearisedAtShear} and the display that follows it give a formula for $D_{\zeta^\circ} \mathcal{F}^\circ(\zeta_*,q)\dot{\zeta}^\circ$.
    The linearised scalar constraints $D_{\zeta^\circ} \mathcal{F}^s(\zeta_*,q)\dot{\zeta}^\circ=0$ correspond to \eqref{eqn:linearisedAvgDepth} and
    \begin{align}
        \label{eqn:linearisedFlux}
        \int_0^1 \int_0^{2 \pi} \dot{u}_0 \ dx \ dy &=0.
    \end{align}

    Given $\dot{\zeta}^\circ \in \mathcal{X}^\circ$, let $\dot{u}_0^n(y)$ be the $n$th Fourier coefficient of $\dot{u}_0(x,y)$, where the Fourier transform is taken in $x$ only. More precisely, let
    \[\dot{u}_0^n(y)= \int_0^{2\pi} \dot{u}_0(x,y) \cos nx \ dx. \]
    Note $ \dot{u}_0(x,y)$ is even in $x$, so integrating against $\sin nx$, an odd function, would give $0$.
    Note also that we make no claims about the convergence of $\sum \dot{u}_0^n(y) \cos nx$.
    We similarly define $\dot{v}_0^n(y), \ldots, \dot{s}^n(y)$, integrating odd functions against $\sin nx$, and even ones against $\cos nx$. Let $\dot{\zeta}^{\circ n} =(\dot{u}_0^n(y), \ldots, \dot{s}^n(y)) $

    For all $n \in \mathbb{Z}$, we have
    \[\int_0^{2\pi}D_{\zeta^\circ} \mathcal{F}^\circ(\zeta_*,q)\dot{\zeta} \sin nx \ dx=\int_0^{2\pi}D_{\zeta^\circ} \mathcal{F}^\circ(\zeta_*,q)\dot{\zeta} \cos nx \ dx=0.\]
    For each fixed $n$, combining this with \eqref{eqn:linearisedAvgDepth} and \eqref{eqn:linearisedFlux} yields a system of ordinary differential equations for $\dot{u}_0^n(y), \ldots, \dot{s}^n(y)$.
    Briefly considering the zero modes  $\dot{u}_0^0(y), \ldots, \dot{s}^0(y)$ reveals they must all be 0.

    Now consider the non-zero modes. Since $q$ only appears in the the linearised kinematic boundary condition at the interface, i.e., the final row of $B_1$, we will enforce this last.
    Using \eqref{eqn:linearisedAvgDepth}, \eqref{eqn:linearisedFlux}, and every component of \eqref{eqn:linearisedAtShear} except  the final row of $B_1$, we ultimately deduce
    \begin{align*}
        \dot{\zeta}^{\circ n} &= \lambda_n \begin{pmatrix}
         c_1(n) \cosh(n k H y)+\omega_0 \sinh(n k H y)\\
         c_1(n) \sinh(n k H y)\\
         \cosh\left(n k H y\right)\\
         \sinh\left(n k H y\right)\\
        c_2(n)\omega_1\sinh\left(n k  \left(1-H\right) y\right) -c_4(n) \cosh\left(n k  \left(1-H\right)  y\right)\\
         c_4(n) \sinh\left(nk \left(1-H\right)y\right) \\
         c_3(n) k^{-1} \sinh \left(n y\right)-c_2(n) \cosh \left(n k  \left(1-H\right) y\right)  \\
         c_2(n) \sinh\left(nk \left(1-H\right)y\right)\\
         c_3(n)\sinh\left(ny\right)
        \end{pmatrix} ,
    \end{align*}
    where $\lambda_n$ is an unknown constant and
    \begin{subequations}
    \label{eqn:kernelConstants}
    \begin{align}
        c_1(n)&=-\frac{1}{\coth (n H k) + \coth (n (1-H)k) }\\
        c_2(n)&=\frac{\sinh (n H k)}{\sinh(n (1-H) k)}\\
        c_3(n)&=\frac{k \sinh(n H k)}{\sinh (n)}\left( \coth(n H k) + \coth (n  (1-H\right) k )  )\\
        c_4(n)&=-\frac{\sinh (n H k)}{\sinh(n (1-H) k) \left( \coth(n H k)+\coth(n  (1-H) k) \right)}.
    \end{align}
    \end{subequations}
    Applying the linearised kinematic boundary condition at the interface then yields that for each $n$, either $\lambda_n=0$ or $q=\mathfrak{d}(nk)$, where $\mathfrak d$ is defined in \eqref{eqn:dispersionRelation}.
    Since $\mathfrak d$ is monotone on $\R_{\geq 0}$, the equation $q=\mathfrak{d} (nk)$ can be satisfied for at most one value of $n$, which we call $m$.
    For $n \neq m$, we must have $\lambda_n=0$, and thus $\dot{\zeta}^{\circ n}=0$.

    Now, for fixed $y$, consider the function $\phi \colon x \mapsto \dot u _0(x,y)$.
    This function is bounded, so is certainly in $L^2(\T)$, and is even, so is $L^2$-orthogonal to $\sin n x$ for all $n$.
    Furthermore, we have just shown that $\phi$ is $L^2$-orthogonal to $\cos nx $ for all $n \neq m$.
    Thus, $\phi=\alpha \cos mx$ for some $\alpha \in \R$, and therefore, $\dot u _0(x,y)$ is of the form $\alpha(y) \cos mx$.
    A very similar argument holds for the other components of $\dot{\zeta}^\circ$, and thus if $\dot{\zeta}$ is in the kernel of $D_\zeta \mathcal{F}(\zeta_*,q_m)$, it must be of the form $(\alpha_1(y) \cos mx , \ldots, \alpha_9(y) \sin mx, 0,0)$.
    Considering \[D_\zeta \mathcal{F}(\zeta_*,q_m)(\alpha_1(y) \cos mx , \ldots, \alpha_9(y) \sin mx, 0,0)=0\] yields a system of ODEs for the $\alpha_i(y)$, which is in fact the same system we have already considered.
    Therefore, the kernel is contained in the span of $\dot{\rho}_m$, as defined in \eqref{eqn:kernelAtShear}.
    We see that $D_\zeta \mathcal{F}(\zeta_*,q_m)\dot{\rho}_m =0$ as required, and therefore, the kernel is exactly the span of this vector.
\end{proof}
\end{prop}

\begin{lem}\label{lem:transversality}
    For any $m$, the transversality condition \eqref{eq:CrandallRabinowitz:transversality} holds at $(\zeta_*,q_m)$. 
\begin{proof}
    Fix a positive integer $m$, and let $R =L'(q_m)\dot{\rho}_m$, where $L(q)=D_\zeta \mathcal{F}(\zeta_*,q)$ and $'$ denotes differentiation with respect to $q$. We assume that transversality does not hold, that is, there exists $\dot{\theta} \in \mathcal{X}$ such that
    \begin{equation}
        \label{eqn:nonTransverse}
       R = L(q_m) \dot{\theta}.
    \end{equation}
    and aim for a contradiction.

    Let $P^n_{\mathcal{X}} \colon \mathcal{X} \to \mathcal{X}$ and $P^n_{\mathcal{Y}} \colon \mathcal{Y} \to \mathcal{Y}$ be the projections onto the $n$th Fourier mode.
    More explicitly,
    \begin{align*}
        P^n_{\mathcal{X}} \begin{pmatrix}
            \dot{u}_0(x,y)\\
            \dot{v}_0(x,y)\\
            \vdots\\
            \dot{h}_1
        \end{pmatrix}&=\begin{pmatrix}
            \pi^{-1}\int_0^{2\pi} \dot{u}_0(t,y) \cos nt \ dt \cdot \cos nx\\
            \pi^{-1}\int_0^{2\pi} \dot{v}_0(t,y) \sin nt \ dt \cdot \sin nx\\
            \vdots\\
            \pi^{-1}\int_0^{2\pi} \dot{h}_1 \cos nt \ dt \cdot \cos nx\\
        \end{pmatrix}\\
        P^n_{\mathcal{Y}} \begin{pmatrix}
            \alpha_1(x,y)\\
            \alpha_2(x,y)\\
            \vdots\\
            \alpha_{21}
        \end{pmatrix}&=\begin{pmatrix}
            \pi^{-1}\int_0^{2\pi} \alpha_1(t,y) \sin nt \ dt \cdot \sin nx\\
            \pi^{-1}\int_0^{2\pi} \alpha_2(t,y) \cos nt \ dt \cdot \cos nx\\
            \vdots\\
            \pi^{-1}\int_0^{2\pi} \alpha_{21} \cos nt \ dt \cdot \cos nx\\
        \end{pmatrix},
    \end{align*}
    where the sine or cosine in the integrand is chosen to match the parity of the function it is being integrated against.
    Note that scalars only have a zero mode, so for $n \neq 0$, $P^n_{\mathcal{X}}$ and $P^n_{\mathcal{Y}}$ send the final 2 components of their arguments to 0.
    We see that $L(q)P^n_{\mathcal{X}}=P^n_{\mathcal{Y}}L(q)$ hence $L'(q)P^n_{\mathcal{X}}=P^n_{\mathcal{Y}}L'(q)$, so applying $P^m_{\mathcal{Y}}$ to \eqref{eqn:nonTransverse} yields $L'(q_m)P^m_{\mathcal{X}} \dot{\rho}_m = L(q_m) P^m_{\mathcal{X}} \dot \theta$.
    Note that $P^m_{\mathcal{X}} \dot{\rho}_m=\dot{\rho}_m$ and so $R = L(q_m) P^m_{\mathcal{X}} \dot \theta$.

    All components of $L(q)$ have no dependence on $q$, except for the 19th, the linearised kinematic boundary condition at the interface.
    Thus,
    \[L'(q_m)\zeta=(0,\ldots,0,-\dot{\eta}_{0x}(x,1), 0,0),\]
    where the non-zero component is the 19th. This means we are seeking a vector $P^1_{\mathcal{X}} \dot \theta $, which is necessarily of the form $(\beta_1(y)\cos mx, \ldots, \beta_9(y) \sin mx, 0, 0)$, such that we applying $L(q_*)$ gives 0 for every component except the 19th.
    Considering exactly the same system of ODEs as in Proposition~\ref{prop:CrandallRabinowitzKernel}, we see that $P^m_{\mathcal{X}} \dot \theta =\lambda\dot{\rho}_m$, for some $\lambda \in \R$.
    Thus $L(q_m)P^m_{\mathcal{X}} \dot \theta =0$.
    However, the 19th component of $R$ is $n\sinh(nkH)\sin nx$,
    and we have the desired contradiction.
\end{proof}
\end{lem}

Taking together Proposition~\ref{prop:CrandallRabinowitzKernel} and Lemma~\ref{lem:transversality}, we have checked the local bifurcation assumption in Theorem~\ref{thm:buffoniToland}.
Applying the Crandall--Rabinowitz bifurcation theorem \cite{cr:simple} and the implicit function theorem
gives a stronger corollary.
\begin{cor}\label{cor:allBifPoints}
For each positive integer $n$, there exists a local solution curve \[(Z_{\loc,n}(\tau), Q_{\loc,n}(\tau))\] defined for $\tau$ in some small interval $(-\eps_n, \eps_n)$. 
These satisfy $Z_{\loc,n}(\tau)=\zeta_* +\tau \dot{\rho}_n + \mathcal{O}(\tau^2)$ as $\tau \to 0$, and $Q_{\loc,n}(0)=q_n$.
Furthermore, there exists a neighbourhood of each bifurcation point $(\zeta_*,q_n)$ such that the only non-trivial solutions in this neighbourhood are those given by $(Z_{\loc,n}(\tau), Q_{\loc,n}(\tau))$. Additionally, if $q \notin \{ q_n \mid n \in \N \}$, then there exists a neighbourhood of $(\zeta_*,q)$ such that the only solutions in this neighbourhood are trivial.
\end{cor}

We have now satisfied all the assumptions of Theorem~\ref{thm:buffoniToland}, and can extend
any of these local solution curves to a global solution curve.
However, we will only do this for  $(Z_{\loc,1}, Q_{\loc,1})$, as this will give us various desirable monotonicity properties; see Section~\ref{sec:refining}. We refer to the global extension of $(Z_{\loc,1}, Q_{\loc,1})$ as  $(Z, Q)$
From here onwards, we refer to our critical parameter value as $q_*$ rather than $q_1$, and the corresponding vector in the kernel as $\dot{\rho}_*$ rather than $\dot{\rho}_1$.

\section{Ruling out a loop}\label{sec:refining}

We have shown all the hypotheses of Theorem~\ref{thm:buffoniToland}.
The conclusion of this theorem yields that one of the two alternatives \ref{alternative:blowup} blow-up, or \ref{alternative:loop} a loop, must occur. 
In this section we show it must in fact be the former.
Consider the \emph{nodal property} given by
\begin{equation}\label{eqn:nodalproperty}
    {v_0}(x,y) < 0 \text{ for all } (x,y) \in (0,\pi) \times (0,1],
\end{equation}
and let
\begin{equation*}
    \mathcal{N} = \{ (Z(\tau), Q(\tau)) \mid \tau>0, \ (Z(\tau), Q(\tau)) \text{ satisfies \eqref{eqn:nodalproperty} } \}.
\end{equation*}
We seek to show that $\mathcal{N}$ is open and closed in $ \{ (Z(\tau), Q(\tau)) \mid \tau>0 \}$.
This task is facilitated by the Hopf lemma, the maximum principle -- see for example \cite[Section 6.4.2]{Evans:pde} -- and Serrin's edge point lemma \cite[Lemma 1]{Serrin:edgepoint}, all of which require scalar ellipticity.

\begin{lem}
    Let $(\zeta,q) \in \mathcal{U}$ satisfy \eqref{eqn:largeampstream2}.
    Then $v_0$, $v_1$, $\eta_0$, $\eta_1$ each satisfy scalar elliptic equations.
\begin{proof}
    Firstly we see from \eqref{eqn:largeampstream2:lowerCR1} and \eqref{eqn:largeampstream2:lowerCR2} that
    \begin{align}
        h_0^2 {\eta_{0xx}} + {\eta_{0yy}} = 0,
    \end{align}
    and in a similar way, \eqref{eqn:largeampstream2:lowerdivfree} and \eqref{eqn:largeampstream2:lowervort} yield
    \begin{align}\label{eq:v0}
        h_0^2 {v_{0xx}} + {v_{0yy}} = 0.
    \end{align}
    The upper layer is more complicated. 
    If we take  \eqref{eqn:largeampstream2:upperCR1} and \eqref{eqn:largeampstream2:upperCR2}, their $x$ and $y$ derivatives, and \eqref{eqn:largeampstream2:lap}, we can eliminate $s_{yy}$ and all derivatives of $\chi_1$, and see
    \begin{align*}
        \frac{s_y^2+h_1^2}{s_x+1}{\eta_{1xx}}-2s_y{\eta_{1xy}}+(s_x+1){\eta_{1yy}}+\frac{((s_x+1)^2-s_y^2 -h_1^2) s_{xx} + 2 (s_x+1) s_{xy} s_y }{(s_x+1)^2}\eta_{1x}=0.
    \end{align*}
    Thinking of this as an equation for $\eta_1$ only, it is linear and elliptic. 
    This is because if we take the principal part, change to Fourier variables $\xi$ and $\nu$, and enforce $\xi^2 + \nu^2 = 1$ by letting $\xi = \cos \theta$ and $\nu = \sin \theta$, then
    \begin{align*}
        &\frac{s_y^2+h_1^2}{s_x+1}\xi^2 -2s_y \xi \nu +(s_x+1)\nu^2 \\
        &\qquad =\frac{s_y^2+h_1^2}{s_x+1}\cos^2 \theta-2s_y\sin \theta \cos \theta +(s_x+1)\sin^2 \theta \\
        &\qquad = \frac{(s_x+1)^2+s_y^2+h_1^2}{2(s_x+1)}+\frac{-(s_x+1)^2+s_y^2+h_1^2}{2(s_x+1)} \cos 2\theta-s_y \sin 2\theta\\
        &\qquad \geq \frac{(s_x+1)^2+s_y^2+h_1^2}{2(s_x+1)} - \frac{\sqrt{\left((s_x+1)^2+s_y^2+h_1^2)^2-4(s_x+1)^2h_1^2\right)}}{2(s_x+1)}\\
        &\qquad =\frac{2(s_x+1)h_1^2}{(s_x+1)^2+s_y^2+h_1^2+\sqrt{\left((s_x+1)^2+s_y^2+h_1^2)^2-4(s_x+1)^2h_1^2\right)} }.
    \end{align*}
    Thus the principal symbol is bounded away from 0.

    Finally, we show $v_1$ satisfies a scalar elliptic equation.
    Considering \eqref{eqn:largeampstream2:upperdivfree} and its $x$ and $y$ derivatives, and  \eqref{eqn:largeampstream2:uppervort} and its $x$ derivative, we can solve for all first and second order derivatives of $u_0$, so long as $\eta_{1x} \neq 0$.
Inserting these into the $y$ derivative of \eqref{eqn:largeampstream2:uppervort} yields
    \begin{align}
    \label{eqn:v1ScalarElliptic}
        \nonumber0&=k^2(\chi_{1y}^2+\eta_{1y}^2)  {v_{1x x}}-2 k((k \chi_{1x}+1) \chi_{1y}+k \eta_{1y} \eta_{1x})  {v_{1x y}}+((k \chi_{1x}+1)^2+k^2 \eta_{1x}^2) {v_{1y y}} \\
        &\quad +\frac{\alpha(x,y) v_{1x} + \beta(x,y) v_{1y}}{(k \chi_{1x}+1) \eta_{1y}-k \chi_{1x} \eta_{1y}},
    \end{align}
where $\alpha$ and $\beta$ can be expressed as polynomials of derivatives of $\chi_1$ and $\eta_1$. 
It may be of interest to note that the right-hand side of \eqref{eqn:v1ScalarElliptic} is the composition of the Laplacian and the coordinate transform $(\chi + 1/k, \eta)$.

We wish to remove this requirement that $\eta_{1x} \neq 0$. 
Considering \eqref{eqn:largeampstream2:upperdivfree} and its $x$ derivative, and  \eqref{eqn:largeampstream2:uppervort} and its $x$ and $y$ derivatives, we can solve for all first and second order derivatives of $u_0$, so long as $\eta_{1y} \neq 0$. 
Substituting this into the $y$ derivative of \eqref{eqn:largeampstream2:upperdivfree} gives us exactly \eqref{eqn:v1ScalarElliptic}. 
Since $(\zeta,q) \in \mathcal{U}$, at each $(x,y)$ at least one of $\eta_{1x}$ and $\eta_{1y}$ must be non-zero, and so \eqref{eqn:v1ScalarElliptic} must hold everywhere.

The operator acting on $v_1$ in \eqref{eqn:v1ScalarElliptic} is elliptic. 
The principal symbol $\hat A$, is given by
    \begin{align*}
        \hat A&=k^2(\chi_{1y}^2+\eta_{1y}^2)\xi^2 -2 k((k \chi_{1x}+1) \chi_{1y}+k \eta_{1y} \eta_{1x})  \xi \nu +((k \chi_{1x}+1)^2+k^2 \eta_{1x}^2) \nu^2.
    \end{align*}
Let 
\begin{align*}
    \alpha &= \sup| (k \chi_{1x}+1)^2+k^2 \chi_{1y}^2+k^2 \eta_{1x}^2+k^2 \eta_{1y}^2| \\
    \beta &= \inf |4k^2 \left((k \chi_{1x}+1) \eta_{1y}-k \chi_{1y} \eta_{1x} \right)^2|. 
\end{align*}
Note that $\beta$ is positive by \eqref{eqn:largeampstream2:upperCR1}, \eqref{eqn:largeampstream2:upperCR2}, \eqref{eqn:defU:gradX} and \eqref{eqn:defU:s}.  
Evaluating the expression for $\hat A$ at $\xi = \cos \theta$, $\nu = \sin \theta$ yields
    \begin{align*}
        2 \hat A&= (k \chi_{1x}+1)^2+k^2 \chi_{1y}^2+k^2 \eta_{1x}^2+k^2 \eta_{1y}^2 - 2k \left( (k \chi_{1x}+1) \chi_{1y}+k \eta_{1y} \eta_{1x} \right)  \sin 2 \theta\\
        & \quad -\left((k \chi_{1x}+1)^2-k^2 \chi_{1y}^2+k^2 \eta_{1x}^2-k^2 \eta_{1y}^2 \right) \cos 2 \theta\\
        &\geq \alpha - \sqrt{\alpha ^2 -\beta },\\
        &= \frac \beta{\alpha+\sqrt{\alpha^2-\beta}}
    \end{align*}
    which is bounded away from 0.
\end{proof}
\end{lem}

Before we prove the open and closed condition we quickly remark on two useful pieces of algebraic manipulation, which for now will appear entirely unmotivated.
Considering \eqref{eqn:largeampstream2:upperdivfree} at some point on the interface, then applying the continuity conditions \eqref{eqn:largeampstream2:ucts}--\eqref{eqn:largeampstream2:Ycts} yields
$k \chi_{0x} v_{1y}-k \chi_{1y} v_{0x}-k \eta_{0x} u_{1y}+k \eta_{1y} u_{0x}+v_{1y}=0$.
Using \eqref{eqn:largeampstream2:lowerdivfree}, \eqref{eqn:largeampstream2:lowerCR1}, \eqref{eqn:largeampstream2:lowerCR2} to eliminate $u_{0x},\chi_{0x},\chi_{0y}$ then gives that for $x$ arbitrary and $y=1$, we have
\begin{subequations}\label{eqn:usefulFactAtInterface}
\begin{equation}\label{eqn:usefulFactAtInterface:1}
    {h_0} {\omega_0} \eta_{0x} \eta_{1y}-{h_0} \chi_{1y} v_{0x}-{h_0} \eta_{0x} u_{1y}+\eta_{0y} v_{1y}-\eta_{1y} v_{0y}=0.
\end{equation}

Consider \eqref{eqn:largeampstream2:upperCR1} at some point on the interface. Applying the continuity conditions, then using \eqref{eqn:largeampstream2:lowerCR1} to eliminate $\chi_{0x}$ yields
\begin{equation}\label{eqn:usefulFactAtInterface:2}
    h_0 \left(s_x+1\right) \eta_{1y}-h_0 \eta_{1x} s_y+h_1 \eta_{0y}=0.
\end{equation}
\end{subequations}
We are now in a position to prove the open and closed conditions.
\begin{lem}[Closed condition]\label{lem:closedcondition}
    Suppose we have $(\zeta,q) \in \mathcal{U}$ solving \eqref{eqn:largeampstream2}, such that for all $(x,y) \in [0,\pi] \times [0,1]$, we have $ {v_0}(x,y) \leq 0$. Then either the nodal property \eqref{eqn:nodalproperty} holds, or else $\zeta$ is the trivial solution $\zeta_*$.
\begin{proof}
    Thanks to \eqref{eq:v0} and the strong maximum principle, either the strict inequality holds in $(0,\pi) \times (0,1)$, or else $v_0$ is constant.
    If $v_0$ is constant, then $(\zeta,q)$ corresponds to a solution of \eqref{eqn:largeampstream} where $V$ is constant in the lower layer.
    Some straightforward calculation shows that this must be a shear solution, and so $\zeta = \zeta_*$. 
    
    It remains to consider the interface in the case that the strict inequality holds in $(0,\pi) \times (0,1)$.
    We proceed by contradiction. Suppose there exists $x_* \in (0,\pi)$ such that $ {v_0}(x_*,1) = 0$. This is a maximum point of $v_0$, and so    $ v_{0x}(x_*,1) = 0$.
    By \eqref{eqn:largeampstream2:kinint}, and the fact that we do not have stagnation on the interface, we must have $\eta_{0x}(x_*,1)= 0$, i.e., we are at a stationary point on the interface. 
    This implies that $\eta_{0y}(x_*,1) \neq 0$, since $\nabla \eta_0$ cannot vanish as we are in $\mathcal{U}$.
    At $(x_*,1)$,  \eqref{eqn:usefulFactAtInterface:1} now simplifies to  $\eta_{0y} v_{1y}=\eta_{1y} v_{0y}$,
    and \eqref{eqn:usefulFactAtInterface:2} becomes $h_0 (s_x+1) \eta_{1y}=-h_1 \eta_{0y}$, which combine to give $h_0 (s_x+1) v_{1y}+h_1 v_{0y}=0$. 
    We know that $h_0$, $h_1$, and $(s_x+1)$ are all positive, so $v_{0y}(x_*,1)v_{1y}(x_*,1) \leq 0$.

    However, the maximum principle and \eqref{eqn:largeampstream2:kintop} give us that $v_1 \leq 0$ for all $x \in [0,\pi]$, and all $y$.
    In particular $(x_*,1)$ is a global maximum for both $v_i$, and so by the Hopf lemma, we have $v_{0y}(x_*,1),v_{1y}(x_*,1) > 0$.
    In other words we have shown that  $v_{0y}(x_*,1)$ and $v_{1y}(x_*,1)$ are two positive numbers with a non-positive product, a contradiction.
\end{proof}
\end{lem}

\begin{lem}[Open condition]\label{lem:opencondition}
    Suppose $(\zeta,q) \in \mathcal{U}$ solves \eqref{eqn:largeampstream2}, and satisfies the nodal property \eqref{eqn:nodalproperty}. Then all $(\tilde \zeta, \tilde q)$ solving \eqref{eqn:largeampstream2} sufficiently close to $(\zeta,q)$ in $\mathcal{X}\times \R$ satisfy \eqref{eqn:nodalproperty}.
\begin{proof}
    We first show that $(\zeta,q)$ satisfies $v_{0x}(0,1) \neq 0$. 
    Suppose for contradiction that $v_{0x}(0,1) = 0$.
    Note, that since we are at $x=0$, we are able to use the parity of the functions to deduce that many terms are 0.
    For example, the $\eta_i$ are even in $x$, so $\eta_{ix}(0,1)=0$.
    We can then use \eqref{eqn:usefulFactAtInterface:2} and the fact that $(\zeta,q) \in \mathcal{U}$ as before to deduce that $\eta_{0y}\eta_{1y} \leq 0$ at $(0,1)$. 
    This inequality is made strict by the fact that $\nabla \eta_i$ is never 0.
    We now consider the $v_i$, and use the Hopf lemma to see that ${v_{ix}}(0,y)<0$ for all $y \in (0,1)$. 
    Since ${v_{ix}}(0,1)=0$, we deduce ${v_{ixy}}(0,1)\geq 0$.

    Applying the parity conditions to the $x$ derivative of \eqref{eqn:largeampstream2:kinint} at $(0,1)$ gives $(u_0+q){\eta_{0xx}}=0$.
    Since there can be no stagnation at the interface, we deduce that ${\eta_{0xx}}(0,1)=0$.
    Inserting this, the parity conditions, and $v_{0x}(0,1) = 0$ into the $x$ derivative of \eqref{eqn:usefulFactAtInterface:1} gives $\eta_{0y} {v_{1xy}}=\eta_{1y} {v_{0xy}}$.
    This, with the sign conditions we have just shown implies that we must have ${v_{1xy}}={v_{0xy}}=0$.
    Therefore we know that all derivatives of $v_0$ of order at most two are 0 at $(0,1)$: $v_0$, $v_{0y}$, ${v_{0xx}}$, and ${v_{0yy}}$ by parity, $v_{0x}$ by assumption, and ${v_{0xy}}$ by the previous proof.
    However, this contradicts Serrin's edge point lemma \cite[Lemma~1]{Serrin:edgepoint} since by the nodal property $v_0 \not\equiv 0$.

    Thus we know that $v_{0x}(0,1)\neq0$, and since the nodal property is satisfied, we know it is negative. 
    The same argument shows that $v_{0x}(\pi,1)>0$, and we know from the nodal property that $v_0(x,1)<0$ for all $x \in (0,\pi)$.

    Now consider the sets
    \begin{align*}
        A&=\{f \in C^1([0,\pi]) \mid f(0)=f(\pi)=0 \} \\
        B&=\{ f \in A \mid f'(0)<0<f'(\pi)  , \ f(x)<0 \text{ for } x \in (0,\pi)\}.
    \end{align*}
    It is an easy exercise to show that $B$ is open in $A$, with respect to the $C^1$ norm.
    Thus, there exists $\eps>0$ such that for all odd $\tilde{v}_0$ with $\lVert v_0 - \tilde{v}_0 \rVert_{C^1(\mathcal{D})} < \eps$, we have $\tilde{v}_0(x,1)<0$ for all $x \in (0,\pi)$.
    In particular, let $(\tilde \zeta, \tilde q )$ solve \eqref{eqn:largeampstream2} and satisfy $\lVert (\zeta, q) - (\tilde \zeta, \tilde q) \rVert_{\mathcal{X} \times \R} < \eps$, and let $\tilde{v}_0$ be the second component of $\tilde \zeta$.
    Then, $\tilde{v}_0(x,1)<0$ for all $x \in (0,\pi)$.
    We also know from oddness that $\tilde{v}_0(0,y)=\tilde{v}_0(\pi,y)=0$ for all $y \in (0,1)$, and from \eqref{eqn:largeampstream2:kinbot} that $\tilde{v}_0(x,0)=0$ for all $x \in (0,\pi)$.
    Therefore, by the strong maximum principle, $(\tilde \zeta, \tilde q )$ satisfies the nodal property.
\end{proof}
\end{lem}

\begin{cor}\label{cor:noLoop}
    The solution $(Z(\tau),Q(\tau))$ satisfies the nodal peroperty \eqref{eqn:nodalproperty} for all $\tau > 0$. In particular, there is no $T>0$ such that $Z(T)=\zeta_*$, and so
    alternative~\ref{alternative:loop} cannot occur.
\begin{proof}   
    We know from Corollary~\ref{cor:allBifPoints} that for small $\eps$, at $Z(\eps)$ we have $v_0(x,1)=\eps \alpha \sin x + \mathcal{O}(\eps^2)$, for some constant $\alpha$ independent of $\eps$.
    From this we deduce that $(Z(\eps),Q(\eps)) \in \mathcal{N}$ for sufficiently small $\eps > 0$, by again considering the sets $A$ and $B$ from the end of Lemma~\ref{lem:opencondition} and the fact that $B$ is open in $A$.

    We now argue as in \cite[Theorem~13]{csv:global}.
    Suppose for contradiction that there exists $T>0$ such that $Z(T) = \zeta_*$. 
    There must in fact exist a minimal such $T$, since $Z(\eps)$ satisfies the nodal property for all small $\eps$.
    Consider the set \[\mathcal{C} = \{(Z(\tau),Q(\tau)) \mid \tau \in (0,T)\}.\]
    The intersection $\mathcal{C}\cap\mathcal{N}$ is non-empty, as it contains $(Z(\eps),Q(\eps))$, and by the open and closed conditions given in Lemmas~\ref{lem:opencondition} and \ref{lem:closedcondition}, it is open and closed in $\mathcal{C}$.
    Since $\mathcal{C}$ is a connected set we have $\mathcal{C}\cap\mathcal{N} = \mathcal{C}$, or in other words, every element of $\mathcal{C}$ satisfies the nodal property \eqref{eqn:nodalproperty}.
    Now consider the solutions $(Z(T-\eps),Q(T-\eps))$. 
    These are a family of non-shear solutions which tend to a shear solution as $\eps \to 0$, or in other words, $(Z(T),Q(T))$ is a bifurcation point. 
    Thus, recalling the local solution curves $(Z_{\loc,n},Q_{\loc,n})$ from Corollary~\ref{cor:allBifPoints}, we see that for some $n$, possibly after reparametrisation, for all small positive $\eps$ either \[(Z(T-\eps),Q(T-\eps)) = (Z_{\loc,n}(\eps),Q_{\loc,n}(\eps)),\] or \[(Z(T-\eps),Q(T-\eps)) = (Z_{\loc,n}(-\eps),Q_{\loc,n}(-\eps)).\]
    Appealing again to Corollary~\ref{cor:allBifPoints} give that $Z_{\loc,n}(\eps) = \zeta_*+\eps\dot{\rho}_n+\mathcal{O}(\eps^2)$, and so in order to satisfy the nodal property, we must have \[(Z(T-\eps),Q(T-\eps)) = (Z_{\loc,1}(\eps),Q_{\loc,1}(\eps)) = (Z(\eps),Q(\eps))\] 
    This however contradicts the proof of Theorem~\ref{thm:buffoniToland}, which classifies how the solution curve can self-intersect.

    The second statement is therefore proved. Setting $T = \infty$ in the definition of $\mathcal{C}$ and appealing as Lemmas~\ref{lem:opencondition} and \ref{lem:closedcondition}
    as above, we similarly obtain $\mathcal{C} \cap \mathcal{N} = \mathcal{C}$, which is the first statement.
\end{proof}
\end{cor}

\section{Uniform regularity and compactness}\label{sec:UniformReg}

In this section, we complete the proof of Theorem~\ref{thm:largeampmain}, and show that  alternative~\ref{alternative:blowup} implies \eqref{eqn:finalConclusion}. 
We find a family of sets which we call $E_M$ and show that these are compact in $\mathcal{U}$. They are in some sense the only compact sets we need to consider, as we show that any other compact set is a subset of some $E_M$.
The compactness arises as a consequence of a much stronger result, namely that in $E_M$ we have control over $\lVert (\zeta,q) \rVert_{\ell +\gamma}$ for arbitrary $\ell$.

We make use of the conformal maps $(\hat X_i, \hat Y_i)=(X_i,Y_i) \circ S_i^{-1}$, as in \eqref{eqn:introduceXhatYhat:0} and\eqref{eqn:introduceXhatYhat:1}.
Similarly, let $(\hat u_i, \hat v_i)$ be given by $(\hat u_i, \hat v_i)=(u_i,v_i) \circ S_i^{-1}$, or equivalently, $(\hat u_i, \hat v_i)=(U,V) \circ (\hat X_i, \hat Y_i)$.
Note that these inverses exist thanks to Lemma~\ref{lem:bijectiveInU}.
Let $\interface$ be the set $\{ (x,1) \mid x \in \T \}$, and let $\hat \interface$ be the set $\{ (x,h_0) \mid x \in \T \}$. Note that $\hat \interface = S_0(\interface)=S_1(\interface)$.
We use $\lVert \cdot \rVert_{\ell+\gamma}$ to mean the $C^{\ell,\gamma}$ norm on the function's domain of definition.
We use $\lVert \cdot \rVert_{\ell+\gamma,\interface}$ to mean the $C^{\ell,\gamma}(\interface)$ norm, and $\lVert \cdot \rVert_{\ell+\gamma,\tilde \interface}$ to mean the $C^{\ell,\gamma}( \tilde \interface)$ norm. 
Given $M>1$, define $E_M$ to be the set of $(\zeta,q) \in \mathcal{U}$ such that $\mathcal{F}(\zeta,q)=0$ and
\begin{equation}\label{eqn:bestBreakdownConditions}
    \frac 1M \leq \inf_{\tilde \interface}|\nabla \hat Y_i|^2, 
    \qquad 
    \frac 1M \leq \inf_{\interface}((u_0+q)^2+v_0^2),
    \qquad
    \sup_{\tilde \interface} | \nabla \hat{Y}_i|^2 \le M.
\end{equation}

We begin with the following straightforward lemma.
\begin{lem}\label{lem:closedBddImpliesBestBreakdown}
    Let $K$ be a compact subset of $ \{ (\zeta,q) \in \mathcal{U} \mid \mathcal{F}(\zeta,q)=0 \}$.
    There exists $M>1$ such that $K \subseteq E_M $.
\begin{proof}
    The map $(\zeta,q) \mapsto \inf_x( (u_0(x,1)+q)^2+v_0(x,1)^2)$ is continuous with respect to the norm on $\mathcal{X}\times \R$. 
    Thus, the image of $K$ under this map is compact and does not contain 0, giving the required bound. 
    To see the other two bounds, we apply the chain rule to the equations $X_i = \hat X_i \circ S_i$ and $Y_i = \hat Y_i \circ S_i$, yielding
    \[|(\nabla \hat Y_0) \circ S_0|^2 = X_{0x}^2 + Y_{0x}^2, \quad |(\nabla \hat Y_1) \circ S_1|^2 = \frac{X_{1x}^2 + Y_{1x}^2}{(s_x+1)^2}.\]
    Taking the infimum and supremum of these quantities correspond to continuous maps from $\mathcal{U}$ to $\R$, and arguing as before, we have the required bounds.
\end{proof}
\end{lem}

A series of lemmas are proven, which allow us to control  $\lVert (\zeta,q) \rVert_{\ell +\gamma}$ on $E_M$ for arbitrary $\ell$, in terms of $M$.
Throughout the following we appeal to \cite[Theorem~9.3]{ADN}, and \cite[Theorem~8.33]{GilbargTrudinger:book}.
We employ an inductive argument, and begin with some steps analogous to a base case.
\begin{lem}\label{lem:inductiveBaseCase}
There exist various functions which we refer to interchangeably as $f$, increasing in their arguments,  such that for any $M>1$ we have the following bounds in $E_M$:
\begin{enumerate}[label=\rm(\alph*)]
    \item \label{basecase:boundsOnBoundary} $ \inf|\nabla \hat Y_i| = \inf_{\tilde{\interface}}|\nabla \hat Y_i|, \quad  \sup|\nabla \hat Y_i| = \sup_{\tilde{\interface}}|\nabla \hat Y_i|$,
    \item \label{basecase:h} $h_0 \in \left[ \dfrac H {\sup | \nabla \hat Y_0|}, \dfrac H {\inf | \nabla \hat Y_0|} \right], \ h_1 \in \left[ \dfrac {1-H} {\sup | \nabla \hat Y_1|}, \dfrac{1- H} {\inf | \nabla \hat Y_1|} \right]$,
    \item \label{basecase:uv} $\lVert \hat u_i,\hat v_i \rVert_\gamma \leq f( \sup |\nabla \hat Y_i|)$, and $|q|$ can be bounded in terms of the $\omega_i$ only,
    \item \label{basecase:s} $|s_x+1|, \dfrac{1}{|s_x+1|} \leq  \dfrac { \sup |\nabla \hat Y_i |}{ \inf |\nabla \hat Y_i |}$,
    \item \label{basecase:chieta} 
    $(\chi_{ix}+1/k)^2+\eta_{ix}^2, \dfrac{1}{(\chi_{ix}+1/k)^2+\eta_{ix}^2} \leq f \left(\sup |\nabla \hat Y_i |,\dfrac 1{ \inf |\nabla \hat Y_i |}, \sup |s_x+1 |, \dfrac 1{\inf |s_x+1 |} \right)$.
\end{enumerate}

\begin{proof}

We begin by proving \ref{basecase:boundsOnBoundary}. Let $w$ be a function on $\T \times [-h_0,h_0]$, given by $ w(x,y) = \log(|\nabla \hat Y_0(x,|y|)|^2)$.
It can be shown straightforwardly using \eqref{eqn:largeampstream2:Ybot} that $w$ is twice differentiable even at $y=0$. From there, we see it is harmonic, and thus achieves its extrema at the boundary $y=\pm h_0$. Therefore, $|\nabla \hat Y_0|$ achieves its extrema on $\tilde \interface$.  Similarly reflecting $\hat Y_1$ over the line $y=h_0+h_1$ gives the result for  $|\nabla \hat Y_1|$.

We now prove \ref{basecase:h}.
By the intermediate value theorem, and the fact the interface has average depth $H$, there exists $x$ such that $\hat Y _0 (x,h_0) =H$.
By the mean value theorem, there exists $y$ such that $\hat{Y}_0(x,h_0)=h_0 \hat{Y}_{0y}(x,y)$.
Thus, \[h_0 \geq \frac H {\sup | \nabla \hat Y_0|}.\]

We know  that $(\hat X_0, \hat Y_0)$ has a conformal inverse $(\hat X_0^{-1}, \hat Y_0^{-1})$ by Lemma~\ref{lem:bijectiveInU}, which satisfies $\sup |\nabla ( \hat Y_0^{-1} ) | \inf |\nabla \hat Y_0| =1$.
Let $(X_{\mathrm{min}},Y_{\mathrm{min}}) \in \Gamma$ be a point of minimum height, i.e., $Y_{\mathrm{min}} \leq \hat Y_0(x,h_0)$ for all $x \in \T$.
This implies that the line segment between the points $(X_{\mathrm{min}},0)$ and $(X_{\mathrm{min}},Y_{\mathrm{min}})$ is contained in $\Omega_0$.
By the mean value theorem, there exists $Y_* \in (0,Y_{\mathrm{min}})$ such that $h_0=(\hat Y_0^{-1})_Y(X_{\mathrm{min}},Y_*) Y_{\mathrm{min}}$. Thus \[h_0 \leq Y_{\mathrm{min}} \sup |\nabla (\hat Y_0^{-1})|\leq \frac H {\inf|\nabla \hat Y_0|} .\]
The $h_1$ case follows similarly.

We now prove \ref{basecase:uv}.
By Lemma~\ref{lem:bijectiveInU}, $(\zeta,q)$ corresponds to a solution $(U,V)$ of \eqref{eqn:largeampstream}. It can be easily shown there exists a stream function $\Psi$, satisfying $\Psi_Y=U$, $\Psi_X=-V$, and $\Delta \Psi = \omega_0 \mathbf{1}_{\Omega_0}+\omega_1 \mathbf{1}_{\Omega_1}$, where $\mathbf{1}_A$ is the indicator function for a set $A$.
Pick $p>2/(1-\gamma)$. We see
 $ \omega_0 \mathbf{1}_{\Omega_0}+\omega_1 \mathbf{1}_{\Omega_1} \in L^p$, thus $\Psi \in W^{2,p}$, and so $U,V \in W^{1,p} \subseteq C^\gamma$.
In particular, there exists some constant $c$ depending only on the $\omega_i$ and our choice of $p$, such that $\lVert U,V \rVert_{\gamma} \leq c$.
The definition of $q$, \eqref{eqn:defineq}, gives that that $|q| \leq \sup| U|$.
Finally, we see that 
\[ \sup_{(x,y),(\tilde x,\tilde y) \in \mathcal{D}} \frac{|\hat u_i(x,y)-\hat u_i(\tilde x,\tilde y)|}{|(x,y)-(\tilde x,\tilde y)|^\gamma} \leq \sup_{(X,Y),(\tilde X,\tilde Y) \in \R \times [0,1]} \frac{|U(X,Y)-U(\tilde X,\tilde Y)|}{|(X,Y)-(\tilde X,\tilde Y)|^\gamma} \sup| \nabla \hat X_i, \nabla \hat Y_i |^\gamma , \]  and we have the required result.

We now prove \ref{basecase:s}.
Since $s_x+1$ is harmonic, it achieves its extrema on the boundary.
At $y=0$, $s$ is identically 0.
At $y=1$, we can write the continuity conditions \eqref{eqn:largeampstream2:Xcts} and \eqref{eqn:largeampstream2:Ycts} in terms of $\hat X_i$, $\hat Y_i$, and $S_i$, then differentiate with respect to $x$ to yield $\hat X_{0x}(x,h_0)=\hat X_{1x}(S_1(x,1))(s_x(x,1)+1)$ and similarly for $\hat Y_i$.
Squaring, summing and rearranging these gives  \begin{equation}\label{eq:sxOnBoundary}
		\frac{ \hat X_{0x}(x,h_0)^2+\hat Y_{0x}(x,h_0)^2}{\hat X_{1x}(S_1(x,1))^2+\hat Y_{1x}(S_1(x,1))^2}=(s_x(x,1)+1)^2.
	\end{equation}
This readily gives us the required upper and lower bounds for $s_x+1$.

Finally, we prove \ref{basecase:chieta}. This can be shown straightforwardly by applying the chain rule to the equations $X_i = \hat X_i \circ S_i$ and $Y_i = \hat Y_i \circ S_i$.
\end{proof}
\end{lem}

We now move on to the inductive part.
\begin{lem}\label{lem:inductiveStep}
There exist various functions which we refer to interchangeably using $f$, increasing in their arguments,  such that for any solution of \eqref{eqn:largeampstream2}, and any integer $\ell \geq 0$, we have the following bounds:
\begin{enumerate}[label=\rm(\alph*)]
    \item \label{inductivestep:XhatYhat}  $ \lVert \hat{Y}_i \rVert_{ \ell+1 + \gamma} \leq f\Big(\lVert \hat{u}_i , \hat{v}_i \rVert_{\ell+\gamma},\dfrac 1{\inf_{\interface} |u,v|} ,   \ell\Big)$.
    \item \label{inductivestep:s} $\lVert s \rVert_{\ell+1+\gamma} \leq f\Big(\lVert \hat Y_i \rVert_{\ell+1+\gamma}, \lVert s \rVert_{\ell+\gamma}, \dfrac 1{\inf |\nabla \hat Y_i |^2}, \ell\Big)$.
    \item \label{inductivestep:chieta} $\lVert \chi_i, \eta_i \rVert_{ \ell +1+ \gamma} \leq f\Big(\lVert \hat{Y}_i  , s\rVert_{\ell+1+\gamma}, h_i, \dfrac1{h_i},\ell\Big)$.
    \item \label{inductivestep:uv} $\lVert  {u}_i ,  {v}_i \rVert_{ \ell +1+ \gamma} \leq f\Big(\lVert  \hat Y_i, {\chi}_i  ,  {\eta}_i, s \rVert_{\ell+1+\gamma}, \dfrac 1{\inf |\nabla \hat Y_1|}, \dfrac 1{\inf|s_x+1|},  h_i, \dfrac1{h_i}, \ell\Big)$.
    \item \label{inductivestep:uhatvhat} $\lVert \hat{u}_i, \hat{v}_i \rVert_{ \ell +1+ \gamma} \leq f\Big(\lVert \hat{u}_i,\hat{v}_i  \rVert_{\ell+\gamma},\lVert u_i, v_i \rVert_{\ell+1+\gamma},\lVert s \rVert_{\ell+1+\gamma}, \dfrac 1{\inf|s_x+1|},  h_i,\dfrac1{h_i}, \ell\Big)$.
\end{enumerate}
\begin{rk}
    The right-hand side of each estimate depends only on quantities which can either be bounded by a previous estimate or by Lemma~\ref{lem:inductiveBaseCase}, or else involve a norm with one fewer degrees of regularity. 
    Furthermore, once \ref{inductivestep:uhatvhat} has been shown to hold, we can then apply the first estimate with an extra degree of regularity, and proceed inductively.
\end{rk}

\begin{proof}[Proof of Lemma~\ref{lem:inductiveStep}]
We first prove \ref{inductivestep:XhatYhat}.    
Both cases $i=0,1$ follow the same argument, so we suppress the dependence on $i$.
We see that $\hat Y$ satisfies
\begin{align*}
    \Delta \hat Y(x,y) = 0, \quad \hat{Y}_x(x,0)=0, \quad \hat{u}(x,1)\hat{Y}_x(x,1)-\hat{v}(x,1)\hat{Y}_y(x,1)=0.
\end{align*}
Thinking of $\hat u , \hat v $ as fixed coefficient functions, this is a linear elliptic system of equations. The Lopatinskii constant is given by $\frac 12 \sqrt{\hat u(x,h_0) ^2 +\hat v(x,h_0) ^2 }$ on $\tilde \interface$, and $\frac 12$ at $y=0$,
so \cite[Theorem~9.3]{ADN} gives us the required result.

We now prove \ref{inductivestep:s}. Note, for the $\ell = 0$ case, $C^\gamma$ control on $s$ is attained from the $C^1$ control shown in estimate~\ref{basecase:s} in Lemma~\ref{lem:inductiveBaseCase}.
    Let \[g(a,b,c,d)=\sqrt{\frac{a^2+b^2}{c^2+d^2}}-1.\] By \eqref{eq:sxOnBoundary} we have that on $\interface$,
\[s_x=g(\hat X_{0x} \circ S_0,\hat Y_{0x} \circ S_0,\hat X_{1x} \circ S_1,\hat Y_{1x} \circ S_1).\]
The set $\R^4 \supseteq E=\{ a^2+b^2+c^2+d^2 \leq 2M, a^2+b^2 \geq M^{-1}, c^2+d^2 \geq M^{-1} \} $ is compact, and $g$ is smooth on this set, so $g$ has finite $C^n$ norm for all $n$. This bound depends on $M$ and $n$.
Hence,
\begin{align*}
    \lVert s \rVert_{C^{\ell+1+\gamma}(\interface)} &\leq \lVert s \rVert_{C^0(\interface)} + \lVert s_x \rVert_{C^{\ell+\gamma}(\interface)}\\
    &= \lVert s \rVert_{C^0(\interface)} + \lVert g(\hat Y_{0y} \circ S_0,\hat Y_{0x} \circ S_0,\hat Y_{1y} \circ S_1,\hat Y_{1x} \circ S_1) \rVert_{C^{\ell+\gamma}(\interface)}\\
    &\leq f( \lVert g \rVert_{C^{\ell  + \gamma}(E)},\lVert \nabla \hat Y_{i}  \rVert_{C^{\ell  + \gamma}(\tilde \interface)}, \lVert D S_i  \rVert_{C^{\ell -1  + \gamma}( \interface)}).
\end{align*}
This yields a bound of the required form which holds on $\interface$. We extend it to the whole domain using \cite[Theorem~8.33]{GilbargTrudinger:book} for the $C^{1,\gamma}$ norm, and \cite[Theorem~9.3]{ADN} for higher order norms.

We now prove \ref{inductivestep:chieta}. The bound on $\eta_i$ can be shown straightforwardly from the fact that $\eta_i = \hat Y_i \circ S_i$. 
The bound on $\chi_i$ is a little harder, as the $\hat X_i$ are unbounded. 
For $x \in [0,2\pi]$, we see
\begin{align*}
    |\chi_1(x,y)| &= \left| \int_0^x \partial_t \big(\hat X_1 ( S_1(t,y)) - t/k \big) \ dt \right|\\
    &= \left|\int_0^x(s_x(t,y)+1)\hat Y_{1y}(S_1(t,y)) + 1/k \ dt \right|\\
    &\leq 2\pi ( \sup |s_x+1| \sup|\nabla \hat Y_1| +1/k),
\end{align*}
yielding a $C^0$ bound on $\chi_1$.
A $C^0$ bound on $\chi_0$ follows similarly but more simply.
Bounding the seminorm part of $\lVert\chi_i\rVert_{\ell+1+\alpha}$ can be done by differentiating \eqref{eqn:introduceXhatYhat:0} and \eqref{eqn:introduceChiEta:Chi}, as any derivative of $\hat X_i$ can be written in terms of $\hat Y_i$, any derivative of $S_i$ can be written in terms of $s$ and the $h_i$, and any derivative of $x/k$ is a constant, or 0.
This gives our required bound.

We now prove \ref{inductivestep:uv}.
Thinking of $\chi_i, \eta_i$ as fixed coefficient functions, the fluid equations \eqref{eqn:largeampstream2:lowerdivfree}, \eqref{eqn:largeampstream2:lowervort}, \eqref{eqn:largeampstream2:upperdivfree}, \eqref{eqn:largeampstream2:uppervort} with the kinematic boundary conditions \eqref{eqn:largeampstream2:kinbot}, \eqref{eqn:largeampstream2:kintop}, and the continuity of $u$ \eqref{eqn:largeampstream2:ucts} and $v$ \eqref{eqn:largeampstream2:vcts} is a linear elliptic system of equations.
As in the proof of Lemma~\ref{prop:Lopatinskii}, the linear system is written in the form of a matrix, the principal part is taken (which in this case changes nothing), we convert to Fourier variables, and call this matrix $\hat A$. Explicitly,
    \[\hat A = \begin{pmatrix}
        h_0  \xi &  \nu & 0 & 0 \\
         \nu & -h_0  \xi & 0 & 0\\
        0 & 0 & -k \eta_{1x}  \nu - \frac{((k \chi_{1x} +1) h_1 - k \eta_{1x} s_y)  \xi}{s_x + 1} & (k \chi_{1x} +1)  \nu - \frac{ ((k \chi_{1x} +1) s_y + k \eta_{1x} h_1)  \xi}{s_x + 1} \\
        0 & 0 & (k \chi_{1x} +1)  \nu - \frac{((k \chi_{1x} +1) s_y +k \eta_{1x} h_1)  \xi}{s_x + 1} & k \eta_{1x}  \nu + \frac{((k \chi_{1x} +1) h_1 - k\eta_{1x} s_y)  \xi}{s_x + 1}
    \end{pmatrix}.
\]
    Notice $y$ derivatives of $\chi_i$ and $\eta_i$ are eliminated in favour of $x$ derivatives, using \eqref{eqn:largeampstream2:upperCR1} and \eqref{eqn:largeampstream2:upperCR2}. 
    We see that
    \begin{align*}
        \det(\hat A) &= \frac{  (h_0^{2} \xi^{2}+\nu^2) ((k\chi_{1x}+1)^{2}+k^2 \eta_{1x}^{2})\left( h_1 ^2 \xi^2+(\nu  (s_x+1)-s_y\xi)^2\right)}{(s_x+1)^2}\\
        &=k^2 (h_0^{2} \xi^{2}+\nu^2)(\hat X_{1x}^2+\hat Y_{1x}^{2}) \left( h_1 ^2 \xi^2+(\nu  (s_x+1)-s_y\xi)^2\right) .
    \end{align*}
    We need upper and lower bounds on this determinant when $\xi^2+\nu^2=1$. 
    A suitable upper bound is $k^2(1+h_0^2)\lVert \hat Y_1 \rVert_{1+\gamma}^2(h_1^2+(1+\lVert s \rVert_{1+\gamma})^2 ) $.
    We now seek a lower bound, and using the same argument we used to bound \eqref{eqn:bigDet} bounded below by \eqref{eqn:bigDetLowerBound}, we see that
    \begin{align*}
    \det(\hat A) &\geq    \frac{2 k^2 \min(1,h_0^2)  h_1^2(\hat X_{1x}^2+\hat Y_{1x}^{2}) }{h_1^2+(s_x+1)^2+s_y^2 +\sqrt{(h_1^2-(s_x+1)^2+s_y^2)^2+ 4(s_x+1)^2 s_y^2 }} \\
    &\geq  \frac{2 k^2 \min(1,h_0^2)  h_1^2 \inf|\nabla \hat Y_1|^2}{h_1^2+(1+\lVert s \rVert_{C^{1+\gamma}})^2 +\sqrt{(h_1^2+\lVert s \rVert^2_{C^{1+\gamma}})^2+ 4(1+\lVert s \rVert_{C^{1+\gamma}})^2 \lVert s \rVert_{C^{1+\gamma}}^2 }}.
    \end{align*}
    The Lopatinskii constant at the wall is given by $ (4k|s_x+1| |\nabla \hat Y_i|)^{-1}$, and at the interface by  $(2k|s_x+1| |\nabla \hat Y_i|)^{-1}$.
    Therefore \cite[Theorem~9.3]{ADN} gives us the required result.

    Finally, we prove \ref{inductivestep:uhatvhat}.
    This is straightforward for $i=0$, so we focus on the $i=1$ case and suppress dependence on $i$ in the notation.
    Since $\hat u$, $\hat v$ are harmonic, applying  \cite[Theorem~8.33]{GilbargTrudinger:book} or \cite[Theorem~9.3]{ADN}, allows us to reduce to considering the case on the boundary.
    This means we need only consider derivatives in $x$. We see
    \begin{align}\label{eq:hatremoving}
        \partial_x^{\ell+1} u = \partial_x^{\ell+1} (\hat u \circ S) = ((\partial_x^{\ell+1} \hat u) \circ S ) \cdot (s_x+1)^{\ell+1} + p( (D^{\leq \ell}_x \hat u ) \circ S, D^{\leq \ell+1}_x S),
    \end{align}
    where $p$ is some polynomial, and $D^{\leq n}_x \varphi$ is the first $n$ derivatives in $x$ of a function $\varphi$.
    Furthermore,
    \begin{align*}
        \lVert \hat u\rVert_{\ell+1+\gamma,\hat{\mathcal{T}}} &= \sup_{\hat{\mathcal{T}}}|\hat u|+ \sup_{x, \tilde x \in \T} \frac {|\partial_x^{\ell+1} \hat u(x,1) - \partial_x^{\ell+1} \hat u( \tilde x,  1)| }{ \left| x- \tilde x\right|^{\gamma}}\\
        &=\sup_{\mathcal{T}}| u|+\sup_{x, \tilde x \in \T } \frac {|\partial_x^{\ell+1} \hat u(S(x,h_0)) - \partial_x^{\ell+1} \hat u(S( \tilde x, h_0))| }{ \left|x-\tilde x+ s(x,h_0)- s(\tilde x,h_0)\right|^{\gamma}}\\
        &\leq \sup_{\mathcal{T}}| u|+ \frac {1}{\inf |s_x+1|^\gamma} \lVert(\partial_x^{\ell+1} \hat u) \circ S\rVert_{\gamma,\mathcal{T}}\\
        &=\sup_{\mathcal{T}}| u|+\frac {1}{\inf |s_x+1|^\gamma} \left\lVert \frac{\partial_x^{\ell+1}  u-p( (D^{\leq \ell}_x \hat u ) \circ S, D^{\leq \ell+1}_x S)}{(s_x+1)^{\ell+1}} \right\rVert_{\gamma,\mathcal{T}},
    \end{align*}
    where the last line uses \eqref{eq:hatremoving}. 
    This gives the claimed bound for $y=1$. 
    The case $y=0$ follows similarly, so we can extend the claimed bound into the whole domain.
\end{proof}
\end{lem}

\begin{cor}\label{cor:BestBreakdownImpliesAnyRegularity}
    Let $\ell \ge 1$ be an integer, $\gamma \in (0,1)$, and $M > 1$.
    There exists $f(\ell,M) > 0$ such that for all $(\zeta,q) \in E_M$, we have $ \lVert (\zeta,q) \rVert_{\ell+\gamma} \leq f(\ell,M)$.
\end{cor}

\begin{prop}\label{prop:EMcompact}
    The set $E_M$ is compact.
\begin{proof}
    Arguing as in the proof of Lemma~\ref{lem:closedBddImpliesBestBreakdown}, $E_M$ is the intersection of pre-images of closed sets in $\R$ under continuous maps, thus is itself closed.
    Furthermore, Corollary~\ref{cor:BestBreakdownImpliesAnyRegularity} implies that $E_M$ is bounded with respect to the $C^{3,\gamma}$ norm. 
    The Arzel\`a--Ascoli theorem then implies that $E_M$ is relatively compact with respect to the $C^{2,\gamma}$ norm. 
    Therefore $E_M$ is compact.
\end{proof}
\end{prop}

We now have everything we need to prove our main result. 
\begin{proof}[Proof of Theorem~\ref{thm:largeampmain}] 
By Theorem~\ref{thm:buffoniToland} we have a global solution curve $(Z,Q)$ which satisfies either alternative~\ref{alternative:blowup} or \ref{alternative:loop}.
Corollary~\ref{cor:noLoop} then gives that in fact alternative~\ref{alternative:loop} cannot occur, and so $(Z,Q)$ must eventually forever leave any compact set.
Proposition~\ref{prop:EMcompact} shows that $E_M$ -- the set on which the left hand side of \eqref{eqn:finalConclusion} is bounded below -- is compact. 
Therefore, since $(Z,Q)$ will forever leave any $E_M$, the limit \eqref{eqn:finalConclusion} must hold.
Finally, the claimed monotonicity of the solutions follows from Corollary~\ref{cor:noLoop} and the kinematic boundary condition \eqref{eqn:largeampstream:kinint}.
\end{proof}

\section{Numerical solutions}\label{sec:numerics} 
Having proved the existence of large-amplitude solutions, we now numerically compute solution branches to explore their qualitative behaviour. We begin by reformulating the problem as non-local equations confined to the a-priori unknown interface. A numerical procedure to then discretise and solve this system is presented. Typical solution behaviours are shown, with particular focus given to how solution branches terminate. 

\subsection{Numerical method}\label{sec:numerics:reformulation}

We restate the problem in a way that is more amenable to numerical methods.
Recall the non-dimensionalisation from Section~\ref{sec:intro:setup} that gives a channel height of unity, and likewise a vorticity jump across the interface of unity as well.
Let the interface be parametrised by the curve
$\eta(t)$, where as before, $\eta$ is non-self-intersecting, and does not touch the walls.
We use $X(t)$ and $Y(t)$ to refer to the horizontal and vertical components of $\eta$ respectively.
Note, these are unrelated to the coordinates $X_i(x,y)$ and $Y_i(x,y)$ introduced in Section~\ref{sec:largeampprelims:reformulation}, which we do not refer to again in this section.
Let $\eta$ have period $2\pi$ in $t$, and its image have period $2\pi/k$ horizontally, where we choose the value of $k$.
In other words, for all $t$ we have
\[(X(t+2\pi),Y(t+2\pi))=(X(t)+2\pi/k,Y(t)).\]
Let it also be even, in that $(X(-t),Y(-t))=(-X(t),Y(t))$.
This implies that $X(0)=0$, and $X(\pi)=-X(-\pi)=\pi/k$. 
We also insist that $|(X', Y')|$ is constant.
This curve gives us 2 regions:
\begin{align*}
    \Omega_0&=\{ (x,y) \text{ below } \eta, \ y>0 \}\\
    \Omega_1&=\{ (x,y) \text{ above } \eta, \ y<1 \},
\end{align*}
as we have in Section~\ref{sec:intro:setup}.
Recall that $H$ is the average height of the curve, in that the area of $\Omega_0 \cap \{(x,y) \mid x \in (0,\pi/k) \}$ is $\pi H / k$.
Recall also that $\omega_0$ and $\omega_1$ are constants satisfying $\omega_0-\omega_1=1$. 

We can now restate \eqref{eqn:largeampstream} in terms of a stream function $\Psi$.
Given three parameters $k$, $H$, and $\omega_0$, we seek a curve $(X(t),Y(t))$ and $\Psi \in C^1(\overline{\R\times(0,1)}) \cap C^2(\overline{\Omega_0}) \cap C^2(\overline{\Omega_1})$ satisfying
\begin{subequations}\label{eqn:largeampstream3}
\begin{align}
    \label{eqn:largeampstream3:vort}\Delta \Psi &= \begin{cases}
\omega_0 & \text{in } \Omega_0\\
\omega_1 & \text{in } \Omega_1
\end{cases}\\
    \Psi_x(x,0) &= 0\\
    \Psi_x(x,1) &=0\\
    \label{eqn:largeampstream3:kinint}\Psi(X(t),Y(t))&=0\\
    \label{eqn:avgDepth}\int_0^\pi Y(t)X'(t) \ ds &= \pi H/k,
\end{align}
\end{subequations}
with $\Psi$ being $2\pi/k$ periodic and even in $x$, and $X$ and $Y$ having the parity and periodicity conditions already discussed.

We reformulate the problem in such a way that we seek functions with a one-dimensional domain.
This dimension reduction makes the numerics vastly more efficient, as if we want mesh points to be a distance $1/N$ apart, then covering $\Omega_j$ over one period requires $\mathcal{O}(N^2)$ mesh points, whereas covering the interface over one period requires only $\mathcal{O}(N)$ mesh points.
To this end, we define a modified velocity field. 
On $\Omega_j$, let
\begin{align*}
    \tilde {u}_j&= \Psi_y - \omega_j (y-H)\\
    \tilde{v}_j&=-\Psi_x\\
    f_j&=\tilde{u}_j-i\tilde{v}_j.
\end{align*}
Identifying $\R^2$ and $\C$, we see that if $\Psi$ satisfies \eqref{eqn:largeampstream3}, then $f_j$ is a holomorphic function in $\Omega_j$. 
We wish to express $\hat{u}_j$ and $\hat{v}_j$ exclusively it terms of their values at the interface. This is done through use of the residue theorem, yielding
\begin{equation}\label{eqn:lowerHolomorphic}
    F(t_*)=-\frac{k}{ \pi} \int_0^\pi F(t)g\big(Z(t),Z(t_*) \big)Z'(t)+\overline{F(t)}g \big(-\overline{Z(t)},Z(t_*) \big) \overline{Z'(t)} \ ds,
\end{equation}
and
\begin{equation}\label{eqn:upperHolomorphic}
\begin{aligned}
    F(t_*)+Y(t_*) - H&=-\frac{k}{\pi}\int_0^1 (F(t)+Y(t)-H)g \big( k, i - Z(t), Z(t_*)-i \big) Z'(t) \\
    &\qquad \qquad+ (\overline{F(t)}+Y(t)-H) g \big( k, i + \overline{Z(t)}, Z(t_*)-i \big) \overline{Z'(t)} \ ds,
\end{aligned}
\end{equation}
where
\begin{align*}
    g(z,w)&=\frac{1}{(e^{-ikz}-e^{ikw})(e^{-ikz}-e^{-ik w})} - \frac{1}{(e^{ikz}-e^{ikw})(e^{ikz}-e^{-ik w})}\\
    \nonumber Z(t)&=X(t)+iY(t)\\
    \nonumber F(t)&=f_0(Z(t)).
\end{align*}
These equations come about through straightforward but tedious calculations, see Appendix~\ref{sec:numericsAlg} for more detail.

We also know that $\Psi(X(t),Y(t))=0$, so differentiating with respect to $t$ gives
\begin{align}
    \label{eqn:the other BC}  X'(t) \Im F(t) + Y'(t) (\Re F(t) +\omega_0 Y(t)) &=0.
\end{align}
Furthermore, through choice of parametrisation, we have that
\begin{equation}\label{eqn:constantSpeed}
    X'(t)^2+Y'(t)^2=(L/\pi)^2,
\end{equation}
for some constant $L$, where $L$ is the arc length of the interface over half a wavelength. 
Here $L$ is unknown and will be determined as part of our solution.
Finally, we enforce the average depth condition \eqref{eqn:avgDepth}.

Equations \eqref{eqn:lowerHolomorphic}--\eqref{eqn:constantSpeed} and \eqref{eqn:avgDepth} give 4 functional equations and one scalar equation which we solve to find 4 functions and a scalar: the real and imaginary parts of $F$, which we call $u(t)$ and $v(t)$, the parametrisation of the interface, $X(t)$ and $Y(t)$, and the arc length $L$. 
We exploit the periodicity of the domain by expressing the unknown functions in terms of Fourier series, that is we write
\begin{align*}
    u(t) &=\sum_{n=0}^\infty \hat{u}_n \cos(nt), & X(t) &= \sum_{n=1}^{\infty} \hat{X}_n \sin(nt),\\
    v(t) &= \sum_{n=1}^\infty \hat{v}_{n} \sin(nt), & Y(t) &= \sum_{n=0}^{\infty} \hat{Y}_n \cos(nt).
\end{align*}
In the above, we have used that $u$ and $Y$ are even, while $v$ and $X$ are odd and with zero mean.

We now truncate the above series in terms of a size parameter $N$. This gives a solution vector $\hat{\phi}$  consisting of the finitely-many Fourier coefficients.
We take the first $N$ terms of $\hat{u}$, the first $N-1$ terms of $\hat{v}$ and $\hat{X}$, the first $N$ terms of $\hat{Y}$, and $L$. 
Hence, $\hat{\phi}\in \R^{4N-1}$, and to form a closed discrete system we must use \eqref{eqn:lowerHolomorphic}--\eqref{eqn:avgDepth} to produce $4N-1$ equations for $\hat \phi$ to satisfy.
We discretise $t$ in to $N+1$ equally spaced mesh points $t_I$ given by
\begin{align}
 t_I &= \frac{(I-1)\pi}{N}, & i=1,\ldots,N-1,
\end{align}
with a corresponding set of $N$ midpoints, $\{  \frac{\pi}{2N}, \frac{3\pi}{2N}, \ldots, \frac{(2N-1)\pi}{N}\} $.
Our first $N-1$ equations are obtained by enforcing the real part of \eqref{eqn:lowerHolomorphic}, with $t_*$ evaluated at all midpoints except the first, and the integral in $t$ is performed using the trapezium rule on mesh points. 
We obtain our the next $N-1$ equations in a similar manner, taking the real part of \eqref{eqn:upperHolomorphic}, with $t_*$ evaluated at all midpoints except the first, and the integral in $t$ is performed using the trapezium rule on mesh points. 
The next $N-1$ equations are \eqref{eqn:the other BC} with $t$ evaluated at all midpoints except the first. 
The next $N$ equations are \eqref{eqn:constantSpeed} with $t$ evaluated at all midpoints, and the penultimate equation is \eqref{eqn:avgDepth} evaluated using the trapezium rule.

Notice that we do not enforce \eqref{eqn:lowerHolomorphic}, \eqref{eqn:upperHolomorphic} or \eqref{eqn:the other BC} at the first midpoint, as we require the same number of equations as unknowns.
It is checked for all converged solutions that  \eqref{eqn:lowerHolomorphic}, \eqref{eqn:upperHolomorphic} and \eqref{eqn:the other BC} are satisfied at the first midpoint to the same order of error as the other midpoints.

We now have $4N-2$ equations for $4N-1$ unknowns. 
There are two choices for our final equation. 
We may enforce a particular amplitude, i.e.~$Y(0)-Y(\pi)=A$ for some chosen $A$. Alternatively, we may insist that our solution $\hat \phi$ lies some distance $d$ from some other vector $\hat \theta$, i.e., $\sum (\hat \phi _n - \hat \theta_n)^2 = d^2$. 

We solve using a Newton method, terminating once the $L^\infty$ norm of the residual is smaller than some critical size, of the order $10^{-10}$.
In order to speed up computation, once the error is sufficiently small (of the order $10^{-4}$), we reuse the value of the Jacobian matrix from the previous iteration. 
To construct a branch of solutions we iterate the solver, finding each successive solution by enforcing a particular distance from the previous one, or incrementally increasing the enforced amplitude. 
At first, a small amplitude solution is recovered using an initial guess constructed from linear theory. 
We then use continuation to compute larger amplitude solutions, using a linear extrapolation of the two previously computed solutions as the initial guess.
For any discrete solution $\hat \phi$, we can recover the fluid velocity away from the interface by applying the residue theorem similarly to as in \eqref{eqn:lowerHolomorphic} and \eqref{eqn:upperHolomorphic}. 
We can then numerically compute derivatives of this velocity field and verify it solves \eqref{eqn:largeampstream:divfree} and \eqref{eqn:largeampstream:lap} to an acceptable level of error.

Note that  the limiting solutions cannot be recovered numerically with the described method. This is because as the limiting solution is approached via continuation, a larger and larger number of Fourier coefficients are required to get a converged solution to the given tolerances, until eventually $N$ becomes impractically large. Nonetheless, solutions close to limiting can be recovered.

\subsection{Results}

\begin{figure}
        \subfloat[]{
        \includegraphics[width=0.45\linewidth,trim={0 30 0 30},clip]{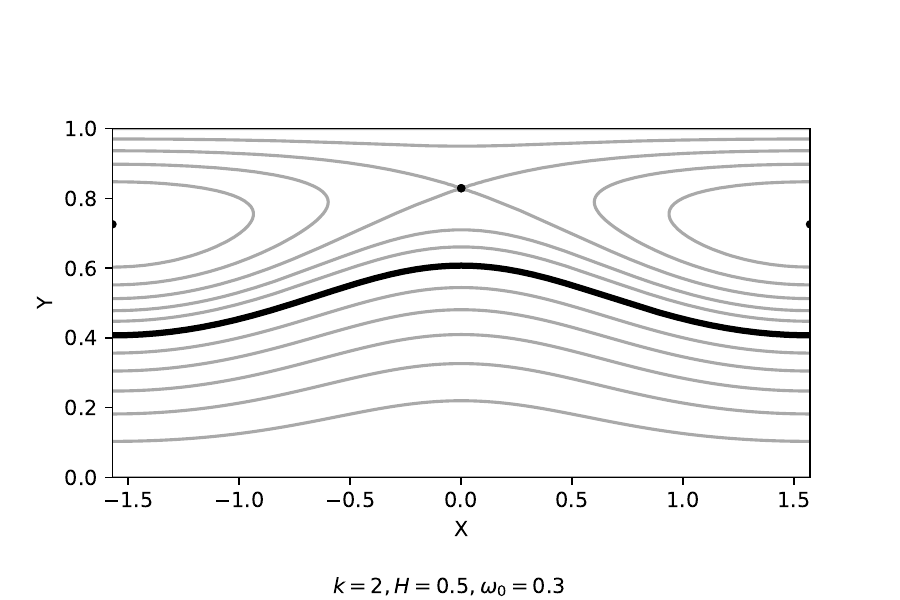}
        \label{fig:perturbative:cateye}}\hfill
        \subfloat[]{\includegraphics[width=0.45\linewidth,trim={0 30 0 0},clip]{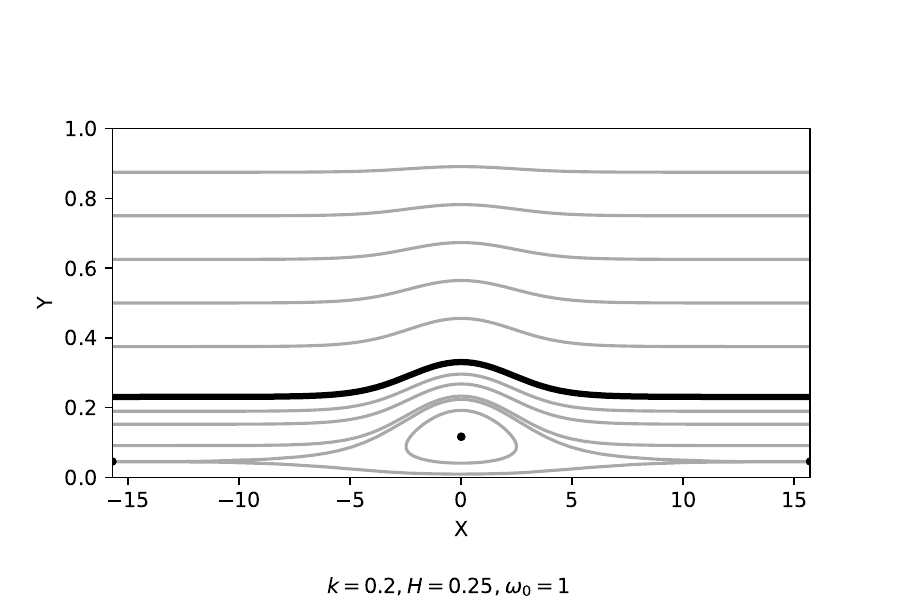}    
        \label{fig:perturbative:centre}}\\
        \subfloat[]{\raisebox{3ex}{\includegraphics[width=0.45\linewidth,trim={0 30 0 30},clip]{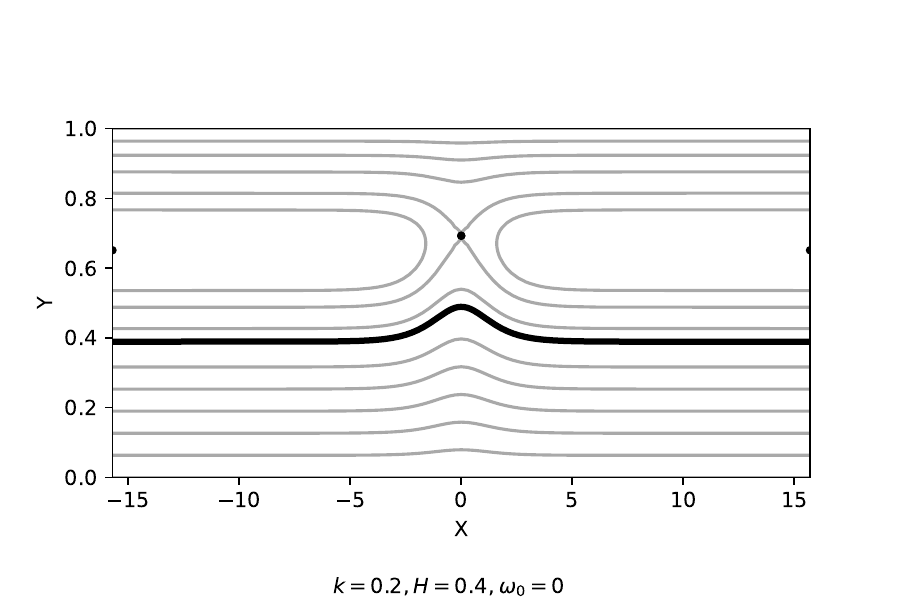}}
         \label{fig:perturbative:saddle}}\hfill
         \subfloat[]{\includegraphics[width=0.45\linewidth]{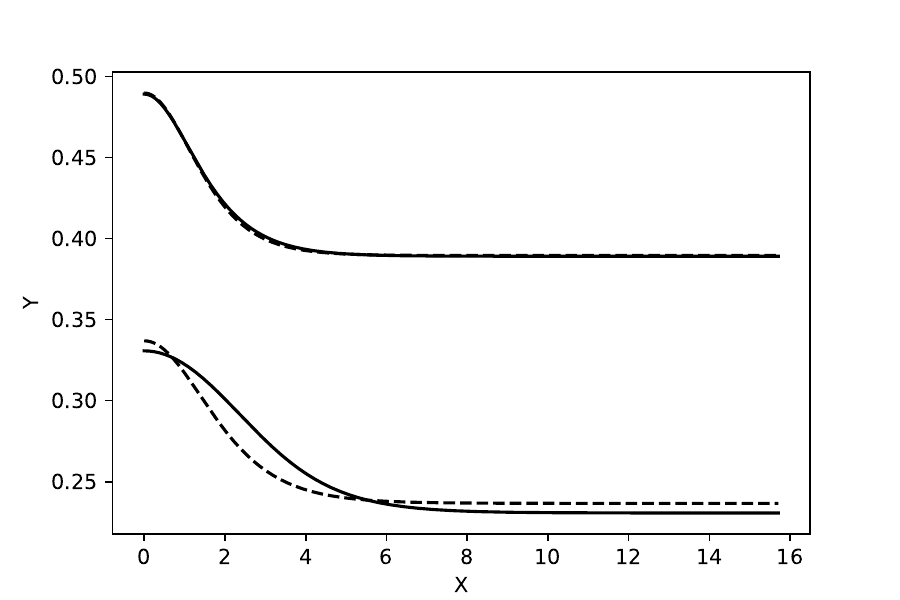}
         \label{fig:perturbative:compare}}
        \caption{Some perturbative solutions and their streamlines. The parameter tuple $(k,H,\omega_0,A)$ is equal to $(2,0.5,0.3,0.2)$ in panel~(A), $(0.2,0.25,1,0.1)$ in panel~(B), and $(0.2,0.4,0,0.1)$ in panel~(C).
        The solution in (A) exhibits a cat's-eye structure in the upper layer. Those in (B) and (C) are approximately shear when $x$ is not close to 0. They have stagnation points at $x=0$ which qualitatively agree with the predictions of \cite{ourpaper}. 
        In panel~(D) we see the interfaces of the solutions in (B) and (C) (solid) and the leading order term of the solution predicted by \cite{ourpaper} (dashed) over half a period.
        The solution from (C) has much better agreement 
        than the solution from (B), due to $|\theta|$ being considerably smaller in (B) than in (C), 0.25 versus 0.8.}
        \label{fig:perturbative}
\end{figure}
\begin{figure}
    \subfloat[]{\raisebox{3ex}
        {\includegraphics[width=0.48\textwidth,trim={0 30 0 30},clip]{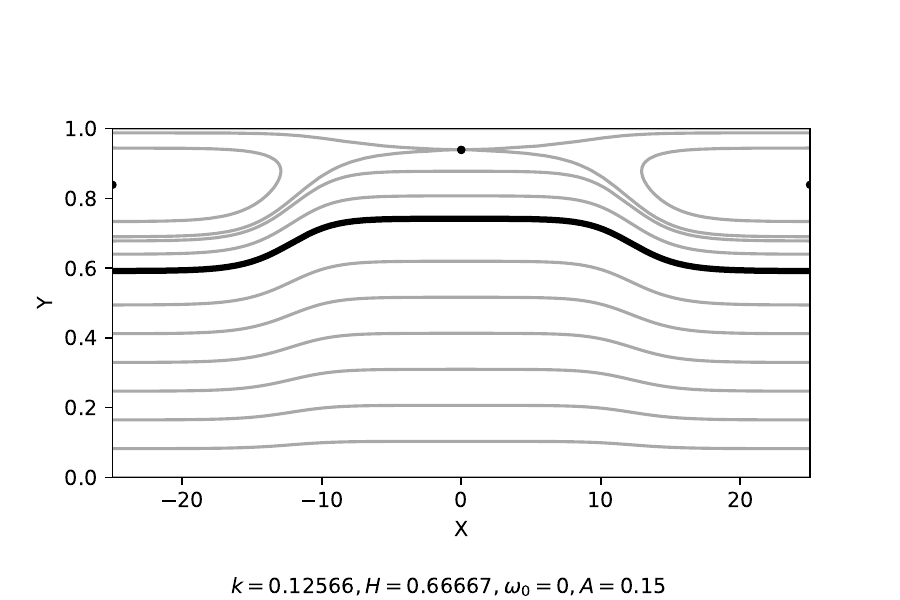}}\label{fig:perturbativebore:streamfn}}\hfill
    \subfloat[]{
         \includegraphics[width=0.48\textwidth,trim={0 0 0 20},clip]{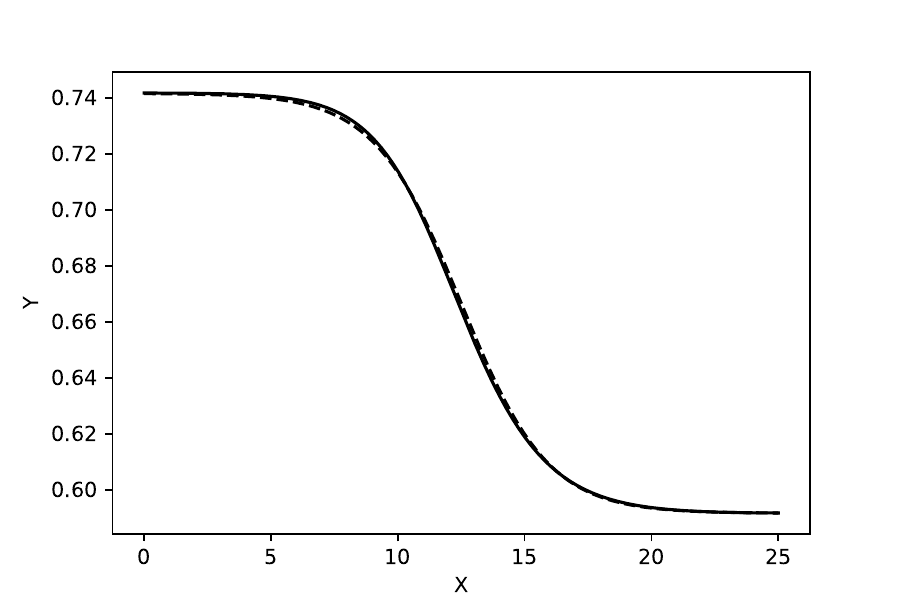}
         \label{fig:perturbativebore:compare}}
    \caption{In panel~(A) we have a perturbative periodic bore with parameters $(k,H,\omega_0,A) = (\pi/25,2/3,0,0.15)$. In panel~(B) we see the numerical solution (solid line) agrees well with the leading order term of the solution predicted by \cite{ChenWalshWheeler:WithoutPhaseSpace} (dashed).}
    \label{fig:perturbativebore}
\end{figure}

We present these results as an exploration of parameter space.
For each $(k,H,\omega_0) \in \R \times (0,1) \times \R$, there is a branch of solutions of \eqref{eqn:largeampstream3} with parameters $k$, $H$, and $\omega_0$, whose interface is strictly monotone in each half-period, and this branch is unique.
Furthermore, the amplitude $A$ increases monotonically along all such branches we have computed.
Thus, any solution  of \eqref{eqn:largeampstream3} with monotone interface is determined uniquely by a quadruple of parameters  $(k, H, \omega_0,A)$.
We note also a symmetry on solution space, namely, if $\Psi$ is a solution to \eqref{eqn:largeampstream3} with parameters $(k, H, \omega_0,A)$, then there exists a ``reflected'' solution given by $(x,y) \mapsto -\Psi(x,-y)$, which has parameters $(k,1- H, 1-\omega_0,-A)$.

We will discuss three kinds of solutions: \emph{perturbative}, then \emph{intermediate}, and finally \emph{limiting} solutions. 
Perturbative solutions are those which are a small perturbation from a shear flow, limiting solutions are those at the end of their solution branches, and intermediate solutions are those between perturbative and limiting.
Note that we do not refer to perturbative solutions as ``small amplitude", as it is possible for intermediate and limiting solutions to have small values of $A$, despite not being close to a shear flow in $C^1$. 
Note also that we do not take numerical limits, so when we refer to specific numerical limiting solutions, these will be solutions very close to the end of their solution branch, rather than exactly at it. 
 
Perturbative solutions display a \emph{cat's-eye structure} in at least one layer, with the other layer either being stagnation free, or also exhibiting a cat's-eye structure.  
By a cat's-eye structure we mean alternating saddle points and centres, with each pair of consecutive saddle points linked by two streamlines encircling the centre.
The saddle points occur above crests of the interface, or below troughs, and the centres occur below crests or above troughs.
One could also show this analytically, using the local bifurcation from Section~\ref{sec:localBifurcation} and the nodal analysis from Section~\ref{sec:refining}.
We show perturbative solutions with stagnation points in the upper layer only and lower layer only in Figure \ref{fig:perturbative}(A) and \ref{fig:perturbative}(B) respectively.
As $k$ becomes small, i.e.~when wavelengths are long, the linear approximation given by the local bifurcation becomes less accurate. 
However, we see increasingly good agreement with the solitary solutions in \cite{ourpaper}.
Note, \cite{ourpaper} makes the additional assumption that the quantity $\theta$ is non-zero, where $\theta = 3h+\omega_0-2$ and $h$ is the height at infinity of the interface of a solitary solution. 
Since $H$ roughly corresponds to $h$ in our long wavelength case, we expect the approximation to be less accurate when $3H+\omega_0-2$  is close to 0. This can be seen in Figure \ref{fig:perturbative}(D), where the solution with $\theta=0.25$ from panel $(B)$ has worse agreement with the long-wave theory \cite{ourpaper} than that of $\theta=0.8$ from panel $(C)$. 

We conjectured in \cite{ourpaper} that in the infinite wavelength case with $\theta=0$, perturbative solutions are bores.
This is demonstrated rigorously in  \cite[Corollary 7.4]{ChenWalshWheeler:WithoutPhaseSpace} for the particular case of $H=\frac 23$, $\omega_0=0$, and our results agree well with this as shown by Figure~\ref{fig:perturbativebore}.

As we move along a solution branch, it is possible for intermediate solutions to display stagnation in a layer that was initially stagnation-free.
This takes the form of a centre appearing either from the lower wall below the peak at $x=0$, or from the upper wall above the trough at $x=\pi/k$. This is demonstrated in Figure~\ref{fig:stagappears}, which shows waves with parameter values such that as the amplitude increases from the solution in panel (A) to panel (B), a stagnation point appears in the lower layer.
\begin{figure}
    \centering
    \begin{subfigure}[t]{0.49\textwidth}
        \centering
        \includegraphics[width=\textwidth,trim={0 30 0 30},clip]{smallcateye.pdf}
        \caption{}
        \label{fig:stagappears:notyet}
     \end{subfigure}
     \begin{subfigure}[t]{0.49\textwidth}
         \centering
         \includegraphics[width=\textwidth,trim={0 30 0 30},clip]{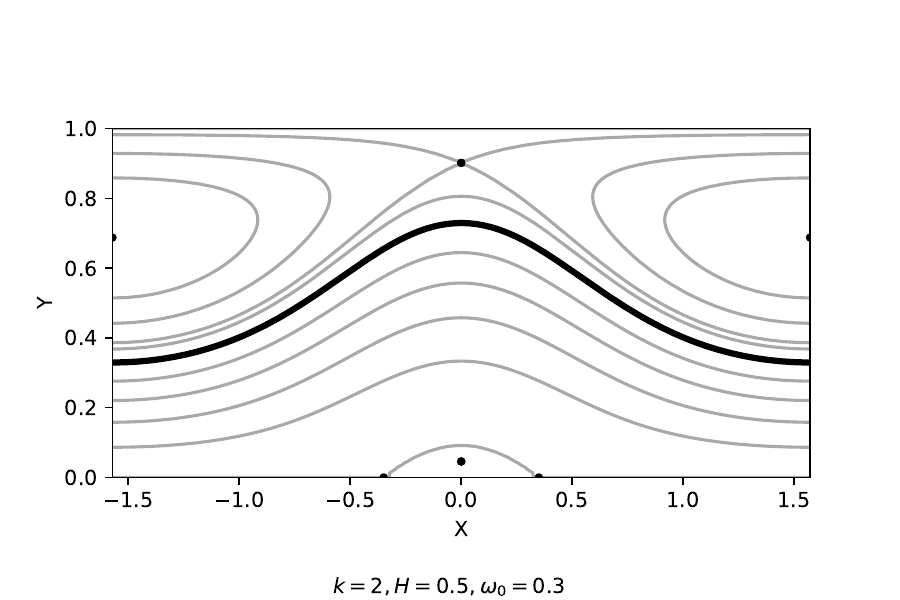}
         \caption{}        
         \label{fig:stagappears:appeared}
     \end{subfigure}
        \caption{Stagnation appears in the lower layer when the amplitude is increased from 0.2 in panel~(A) to 0.4 in panel~(B). For both solutions, $(k,H,\omega_0)=(2,0.5,0.3)$.}
        \label{fig:stagappears}
\end{figure}
\begin{figure}
    \centering
    \begin{subfigure}[t]{0.79\textwidth}
        \centering
        \includegraphics[width=\textwidth,trim={0 30 0 20},clip]{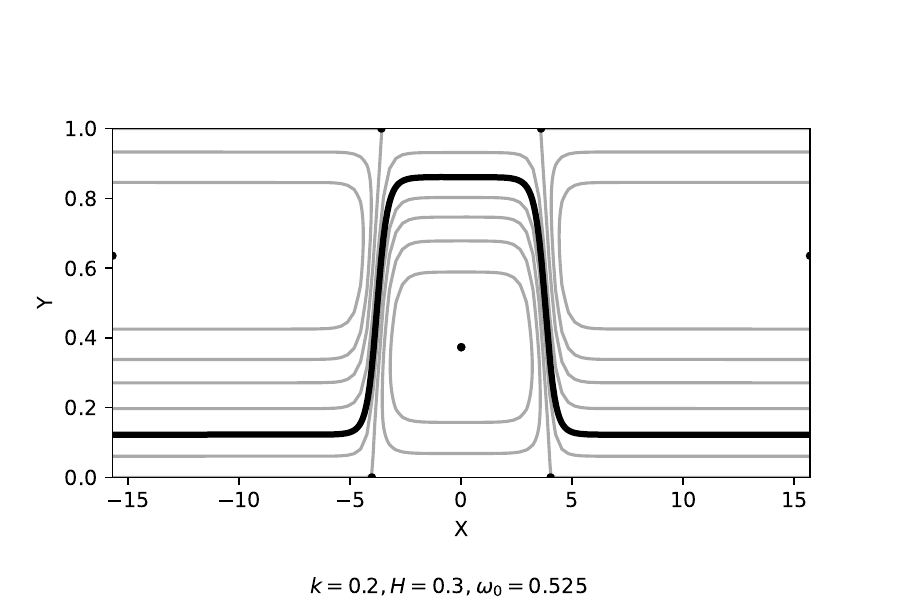}
     \end{subfigure}
        \caption{A solution with $(k,H,\omega_0,A) = (0.2,0.3,0.525,0.73873)$. The flow is approximately shear for $x \in [-2,2]$ and for $x \in [5\pi-10,5\pi+10]$.}
        \label{fig:intermediatebore}
\end{figure}
\begin{figure}
    \subfloat[]{
        \includegraphics[width=.45\textwidth,clip]{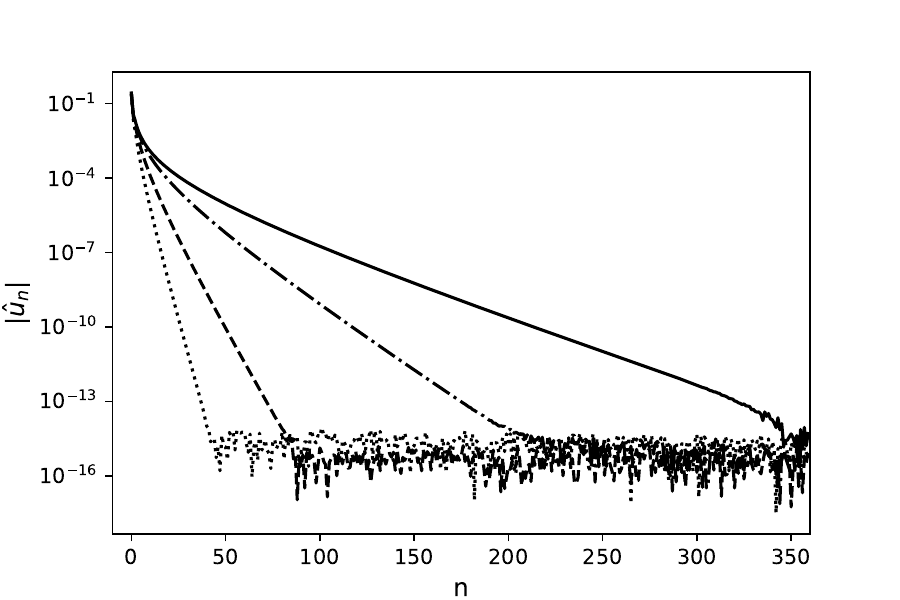}
        \label{fig:coefs:coefs}} \hfill
    \subfloat[]{
        \includegraphics[width=0.45\textwidth,trim={0 15 0 40},clip]{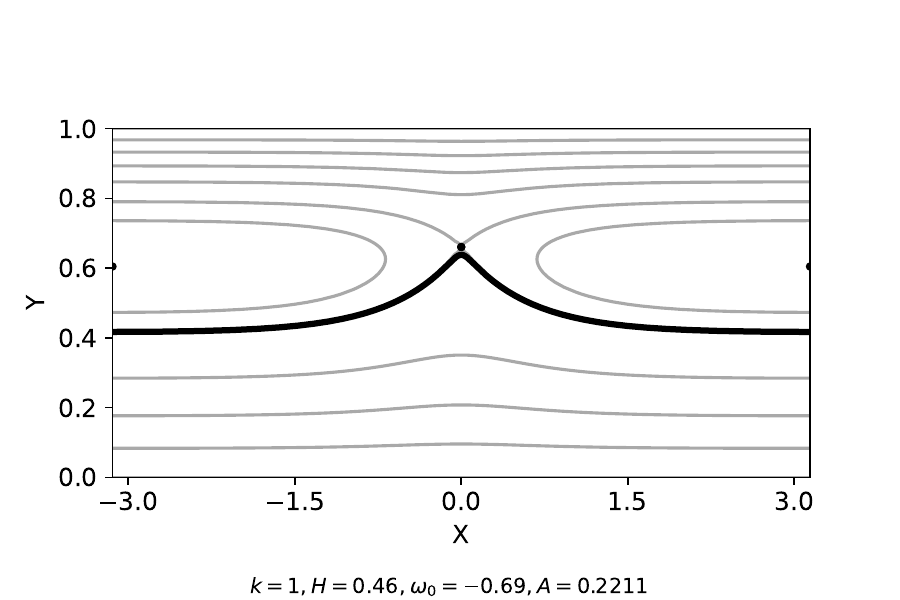}    
        \label{fig:coefs:streamlines}}\\
        \caption{  In panel~(A) we see the how the Fourier coefficients of the horizontal velocity decay along the branch with parameters $(k,H,\omega_0)=(1,0.46,-0.69)$. The solutions have size parameter $N=360$. The dotted line corresponds to $A=0.1$, the dashed line to $A = 0.15$, the dot-dashed line to $A=0.2$, and the solid line to $A=0.2211$
        In panel~(B) we have the streamlines of the last solution.}
        \label{fig:coefs}
\end{figure}

For long wavelengths, many intermediate solutions have two distinct regions per period, in each of which the flow behaves like a shear flow.
We call such a solution a \emph{periodic bore}.

\begin{figure}
    \subfloat[]{
        \includegraphics[width=0.31\textwidth,trim={0 15 0 50},clip]{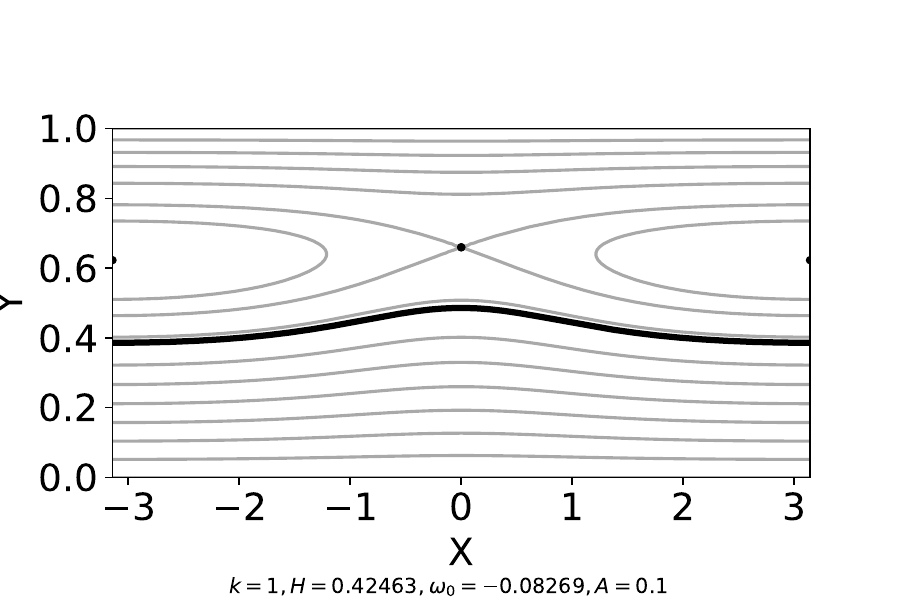}
        \label{fig:saddleSplitting:smallCorner}}
    \subfloat[]{
        \includegraphics[width=0.31\textwidth,trim={0 15 0 40},clip]{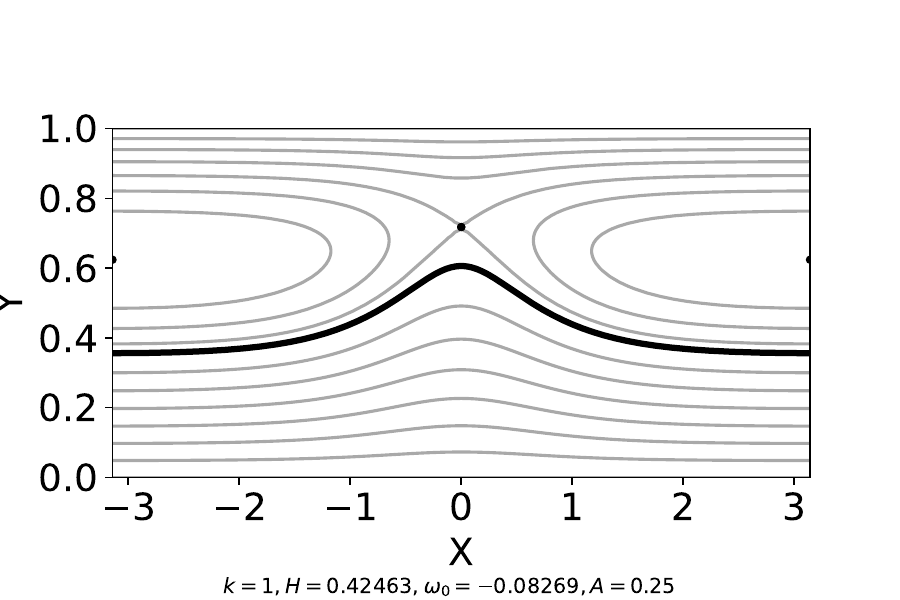}    
        \label{fig:saddleSplitting:midCorner}}
    \subfloat[]{
        \includegraphics[width=0.31\textwidth,trim={0 15 0 40},clip]{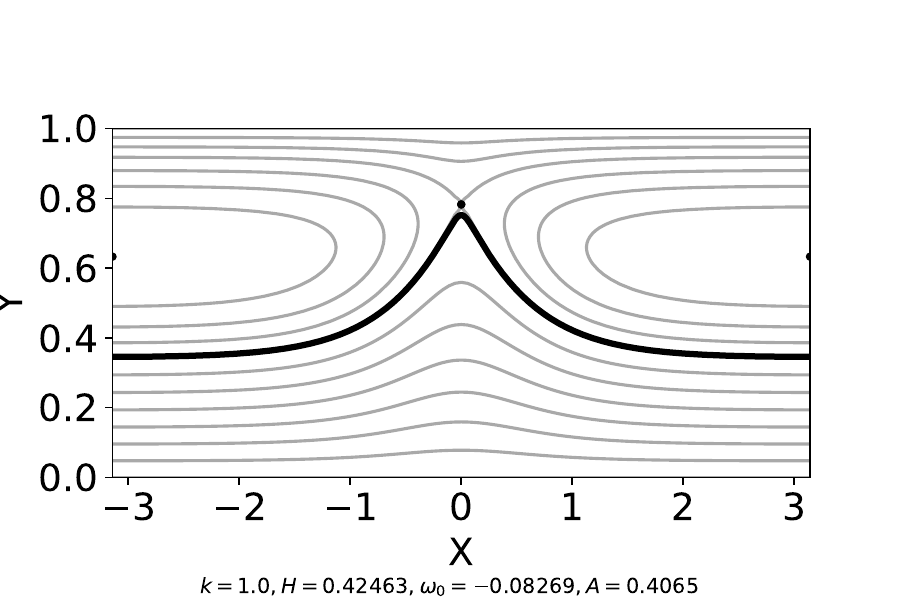}
        \label{fig:saddleSplitting:largeCorner}}\\
    \subfloat[]{
         \includegraphics[width=0.31\textwidth,trim={0 15 0 50},clip]{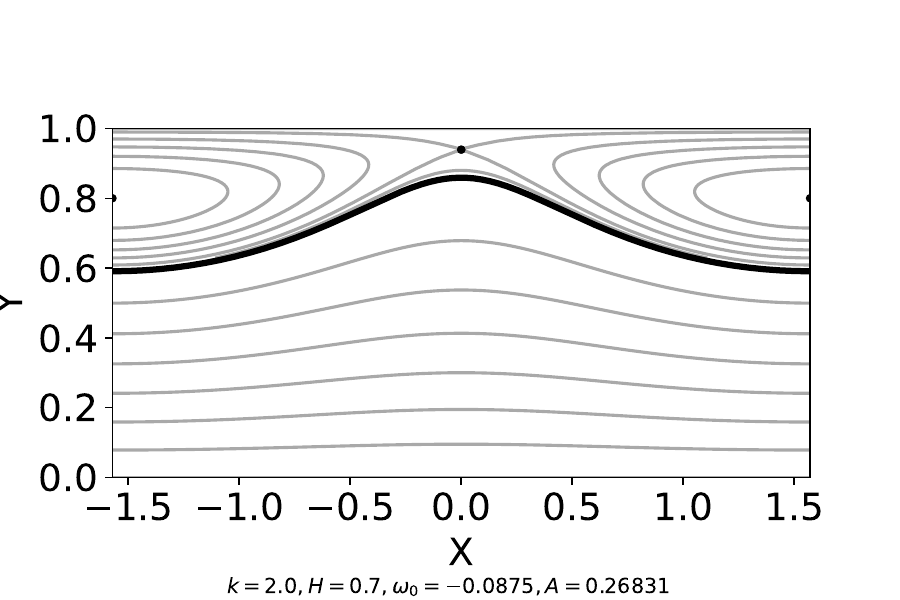}
        \label{fig:saddleSplitting:smallstagtouchwall}}
    \subfloat[]{
         \includegraphics[width=0.31\textwidth,trim={0 15 0 40},clip]{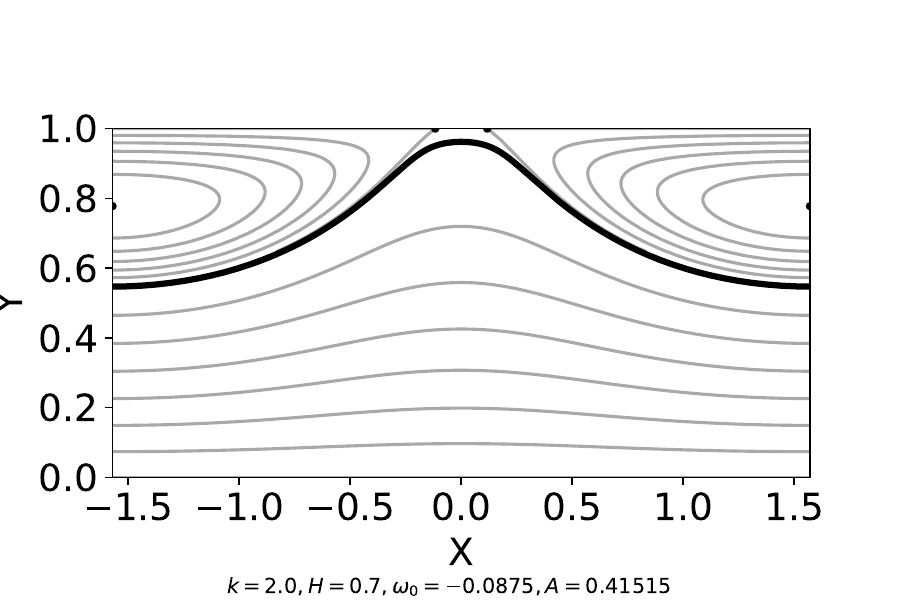}
         \label{fig:saddleSplitting:midstagtouchwall}}
    \subfloat[]{
        \includegraphics[width=0.31\textwidth,trim={0 15 0 40},clip]{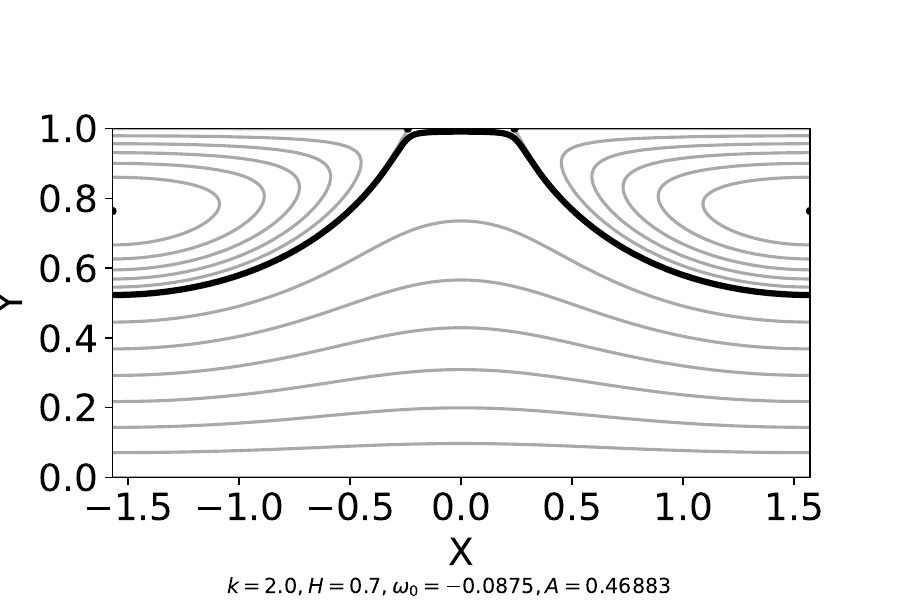}
         \label{fig:saddleSplitting:largestagtouchwall}}
        \caption{ On the upper row we have solutions along a branch with $(k,H,\omega_0)=(1,0.42463, -0.08269)$, with the amplitude increasing from $0.1$ in panel~(A) to $0.25$ in panel~(B) to $0.4065$ in panel~(C). 
        The saddle point collides with the interface, thus the limiting solution is of Type I.
        On the lower row we have solutions along a branch with $(k,H,\omega_0)=(2,0.7, -0.0875)$. In panel~(D) we have a perturbative solution with $A=0.25$. 
        As the amplitude increases, the saddle point collides with the upper wall and splits into two saddle points, as we see in panel~(E), which has $A=0.4$. 
        There is now no stagnation between the crest and the upper wall, which allows a Type II solution to develop as the limit, demonstrated in panel~(F), which has $A=0.46883$.}
        \label{fig:saddleSplitting}
\end{figure}
To better understand these periodic bores, we make a brief digression to discuss \emph{conjugate flows}.
Given constants $h$, $c$, we can define a corresponding shear flow 
\begin{equation*}
    \Psi^*_{h,c}(y) = \begin{cases}
        \omega_0(y-h)^2+c(y-h) & y \in [0,h]\\
        \omega_1(y-h)^2+c(y-h) & y \in [h,1].\\       
    \end{cases}
\end{equation*}
In order for $\Psi$ to be a periodic bore, we must have constants $h$, $\tilde h $, $c$, and $\tilde c$,  such that for $x $ near $0$ we have $\Psi(x,y) \approx \Psi^*_{h,c}(y)$, and for  $x $ near $\pi/k$ we have $\Psi(x,y) \approx \Psi^*_{\tilde h,\tilde c}(y)$, where the approximation is in a $C^1$ sense.
Now observe that for any $\Psi$ solving \eqref{eqn:largeampstream3}, we can define three invariants, i.e., quantities which do not depend on $x$. 
These are the mass flux in the lower layer, given by $\Psi(X(t),Y(t))-\Psi(X(t),0)$, the mass flux in the upper layer, given by $\Psi(X(t),1)-\Psi(X(t),Y(t))$, and the flow force,
\begin{equation}
\label{eqn:flowforce}
    \int_0^{Y(t)} \left(  \frac{\Psi_x^2 - \Psi_y^2}{2}+\omega_0 y  \Psi_y \right) \ dy +\int_{Y(t)}^1 \left(  \frac{\Psi_x^2 - \Psi_y^2}{2}+\omega_1 y  \Psi_y \right) \ dy,
\end{equation} 
where the functions in \eqref{eqn:flowforce} are evaluated at $(X(t),y)$.
Thus we seek shear flows $\Psi^*_{ h, c}$ and $\Psi^*_{\tilde h,\tilde c}$ with $\tilde h<h$, equal mass fluxes in both layers, and equal flow force. 
We call such shear flows \emph{conjugate} flows.
Any shear flow is clearly conjugate to itself, but more interestingly, distinct shear flows $\Psi^*_{ h, c}$ and $\Psi^*_{\tilde h,\tilde c}$ are conjugate if and only if
\begin{subequations}\label{eqn:conj}
\begin{align}
    \label{eqn:conj:c}c &= \tfrac 19 ( 3h+\omega_0-2)^2+h(1-h)\\
    \label{eqn:conj:ctilde}\tilde c &= \tfrac 19 ( 3\tilde h+\omega_0-2)^2+\tilde h(1-\tilde h)\\
    \label{eqn:conj:htilde}\tilde h&=-h+\tfrac23 (2-\omega_0).
\end{align}
In Figure~\ref{fig:intermediatebore}, we present a typical profile of a periodic bore. Taking the value of $h$ to be the interface height at $x=0$, equation \eqref{eqn:conj:htilde} predicts a value of $\tilde{h}=0.122310$ up to 6 decimal places. Looking at the interface at $x=\pi/k$, we find an interface height of $Y=0.122296$, in superb agreement with the conjugate flow theory. We can construct a lower bound for the periodic bore amplitude as follows.
Monotonicity of the interface imposes the further restriction that 
\begin{equation}
    \label{eqn:conj:H}0<\tilde h < H < h < 1.
\end{equation}
\end{subequations}
Note, it is not necessarily true that $\frac 12( h + \tilde h)=H$, as the region in which $\Psi \approx \Psi^*_{h,c}$ may not be the same size as the region in which $\Psi \approx \Psi^*_{\tilde h,\tilde c}$. 
By monotonicity, the amplitude of $\Psi$ is $h-\tilde h$.
Thus, if $H \neq \frac13 (2-\omega_0) $, then  by \eqref{eqn:conj:htilde} and \eqref{eqn:conj:H}, periodic bores must have amplitude greater than $2|H-\frac13 (2-\omega_0)|$.
Once this minimum amplitude has been reached, we do indeed see periodic bores in the numerics.
Solutions below this minimum amplitude fit better in the perturbative regime, and are best described by the solitary solutions found in \cite{ourpaper}.
If however, $H=\frac13(2-\omega_0)$, then no such minimum amplitude exists, and we can find perturbative periodic bores as discussed.

\begin{figure}
    \centering
    \begin{subfigure}[t]{0.3\textwidth}
        \centering
        \includegraphics[width=\textwidth,trim={0 15 0 50},clip]{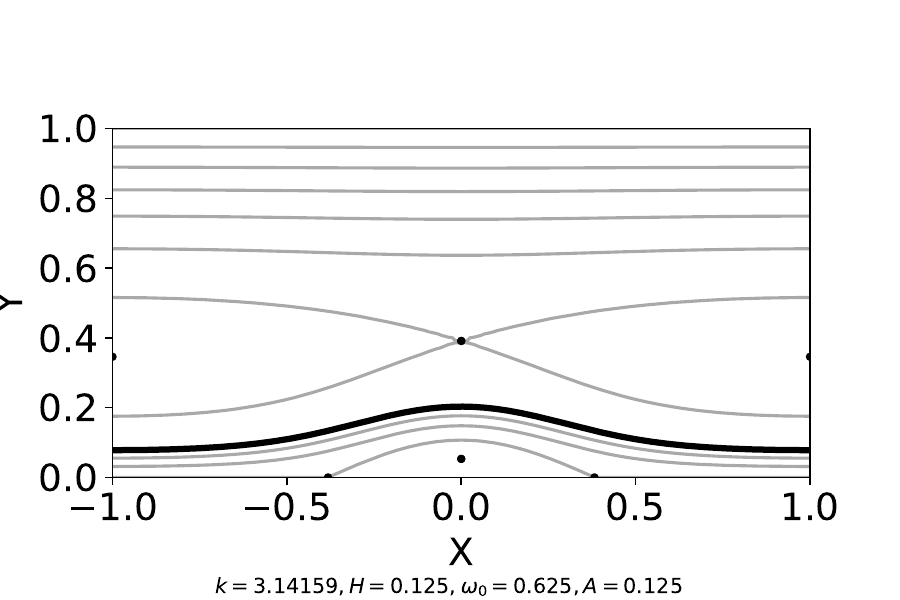}
        \caption{}
        \label{fig:unclear:small}
    \end{subfigure}
    \begin{subfigure}[t]{0.3\textwidth}
            \centering
        \includegraphics[width=\textwidth,trim={0 15 0 40},clip]{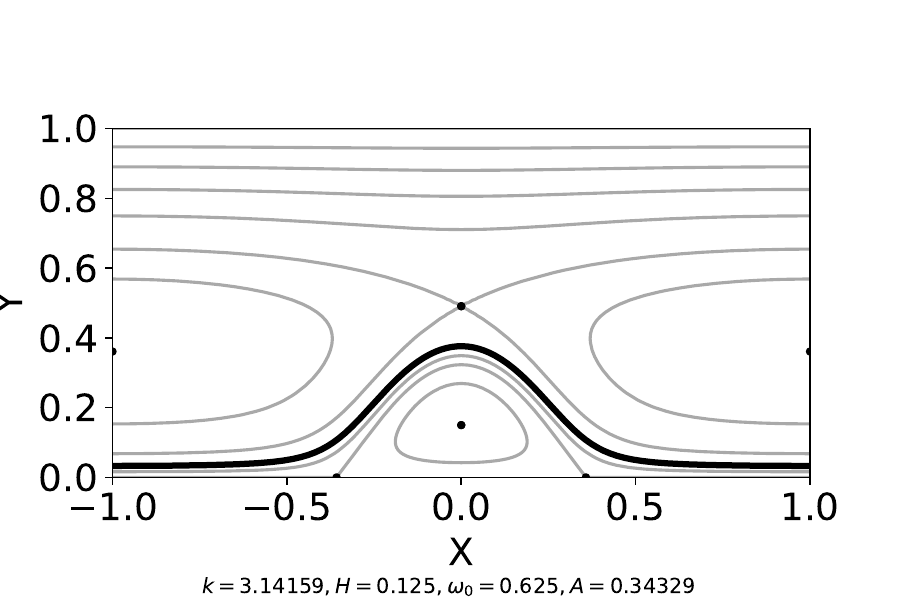}    \caption{}        \label{fig:unclear:mid}
    \end{subfigure}
    \begin{subfigure}[t]{0.3\textwidth}
            \centering
        \includegraphics[width=\textwidth,trim={0 15 0 40},clip]{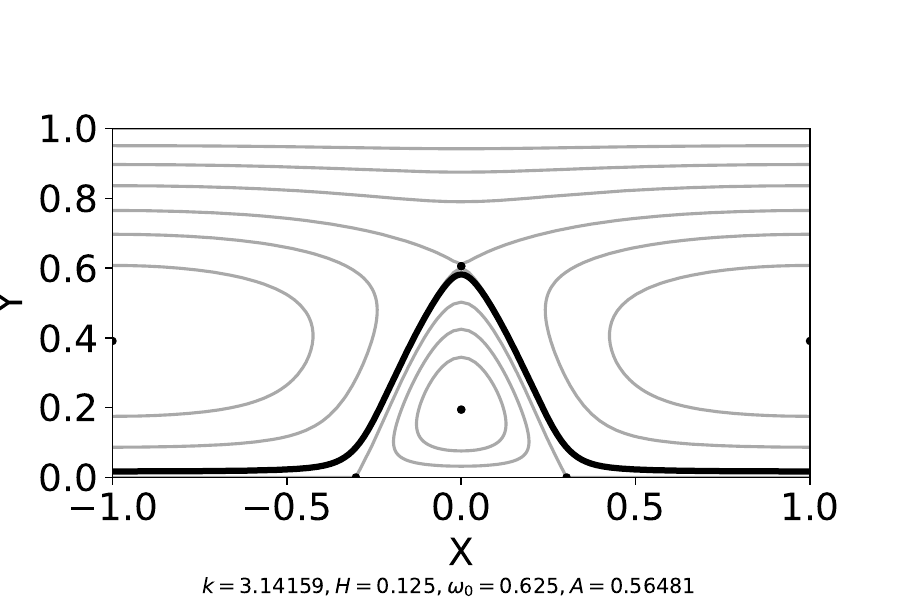}
	    \caption{}        \label{fig:unclear:big}
     \end{subfigure}
        \caption{ Some solutions along a branch with $(k,H,\omega_0)=(\pi,0.125,0.625)$.  As the amplitude increases from $0.125$ in panel~(A) to $0.3$ in panel~(B), it initially seems we have a Type II solution, since the distance from the trough to the lower wall is much smaller than from the peak to the saddle point. However, when the amplitude is further increased to $0.56481$ in panel~(C), the crest of the interface becomes close to a saddle point in the upper layer, making it unclear which of these causes termination of the solution branch.}
        \label{fig:unclear}
\end{figure}
\begin{figure}
    \subfloat[]{
        \includegraphics[width=0.48\textwidth,trim={0 20 0 20},clip]{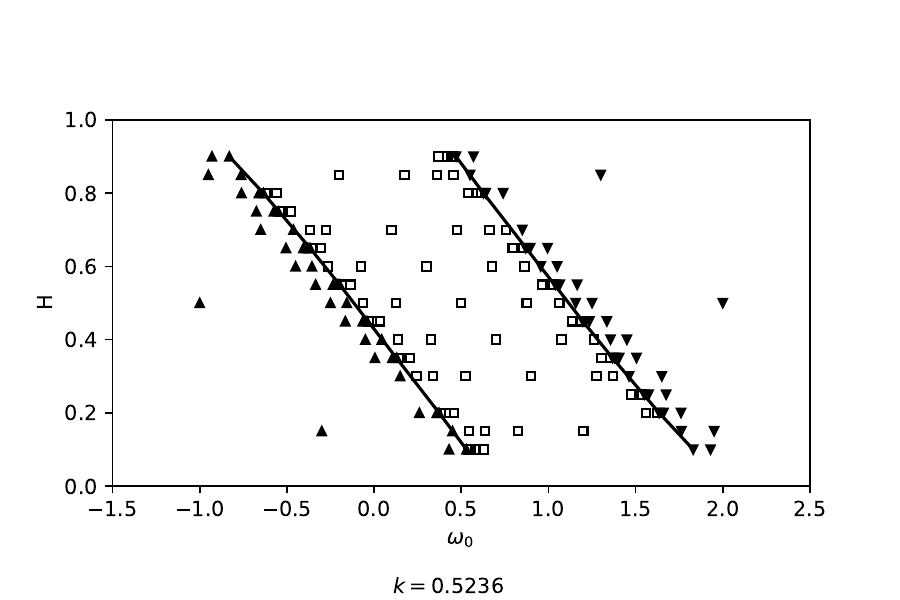}
        \label{fig:paramspace:piby6}}\hfill
    \subfloat[]{
        \includegraphics[width=0.48\textwidth,trim={0 20 0 20},clip]{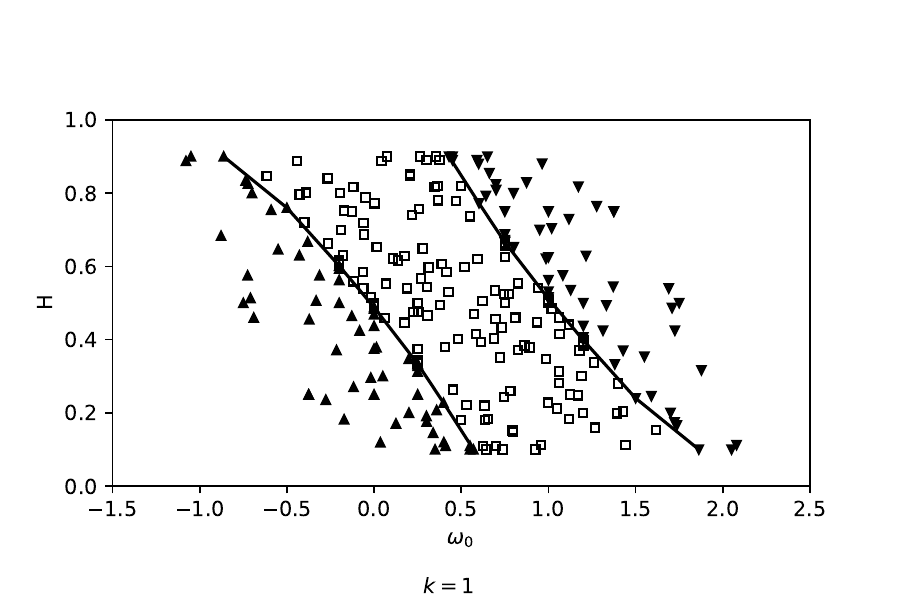}     
        \label{fig:paramspace:1}}\\
    \subfloat[]{
        \includegraphics[width=0.48\textwidth,trim={0 20 0 20},clip]{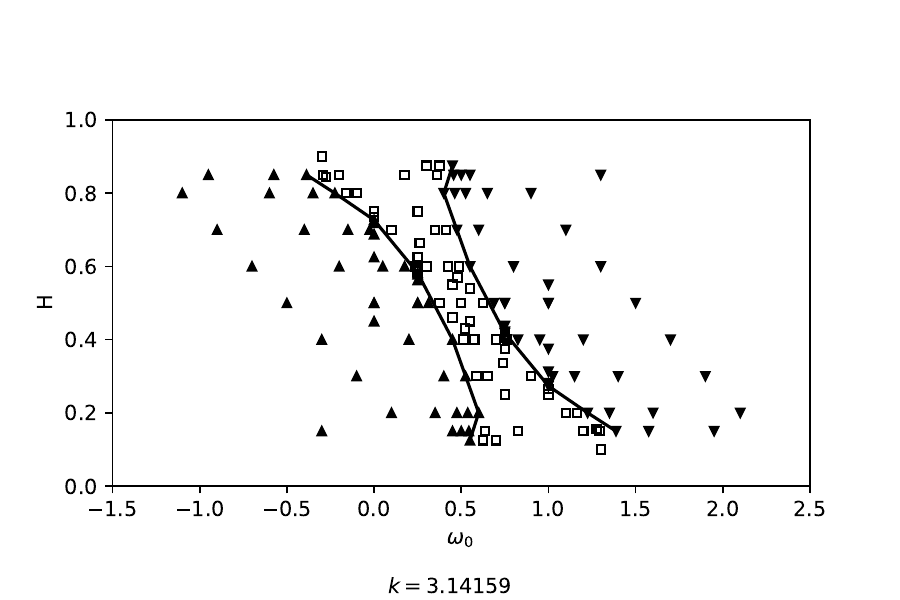}
        \label{fig:paramspace:pi}}\hfill
    \subfloat[]{
         \includegraphics[width=0.48\textwidth,trim={0 20 0 20},clip]{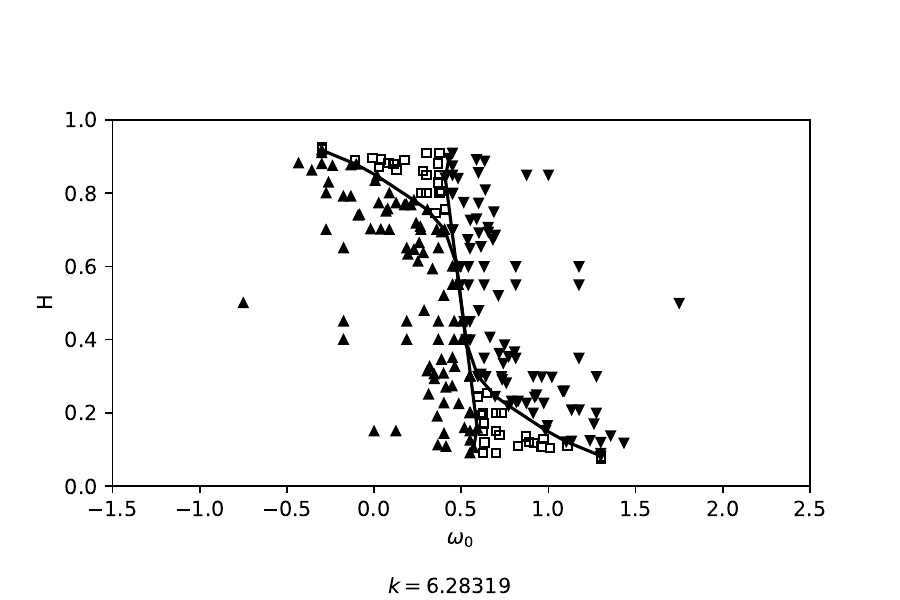}
         \label{fig:paramspace:k2pi}}
    \caption{ The nature of limiting solutions in different regions of parameter space. In panel~(A) we have $k=\pi/6$, in panel~(B)  $k=1$, in panel~(C)  $k=\pi$, in panel~(D)  $k=2\pi$. Upwards pointing triangles correspond to Type I solutions with a corner at their crest, downwards pointing triangles correspond to Type I solutions with a corner at their trough, and squares correspond to Type II solutions. The lines give the approximate boundary between regions of parameter space, and are generated by taking the convex hull of the region of Type I solutions with a corner at their crest. }
        \label{fig:paramspace}
\end{figure}

We now discuss the formation of limiting solutions. 
As previously stated, limiting solutions cannot be computed by our numerical scheme, due to a slowing of the decay of the coefficients of the Fourier expansions. We demonstrate this in figure \ref{fig:coefs},  where in panel (A) we show the Fourier coefficients of the horizontal velocity $\hat{u}_n$ for four solutions of increasing amplitude with $N=360$. The largest amplitude solution is shown in panel $(B)$. It has much slower decay of the coefficients, and we believe this solution is near-limiting. We justify this by observing that as the amplitude increases, the stagnation point in the upper layer above the crest approaches the interface. For the solution in panel (B), the stagnation point is very close to the interface, and we conjecture that as the amplitude increases further, the stagnation point will reach the interface, and the resulting wave profile will form a corner at $x=0$. Unfortunately, impractically large $N$ is required to obtain converged solutions beyond the wave shown in panel (B). This demonstrates we can recover near-limiting waves, but not the limiting wave itself. Nonetheless, this allows us to predict the nature of the limiting wave.
 Our computations through parameter space suggest that there are two types of limiting solution: those that form corners, and those that approach a wall, which we refer to as \emph{Type I} and \emph{Type II} solutions respectively.

There are several ways a solution branch may approach its limiting solution. 
We describe how termination criteria can be achieved at the crest of the interface, but all of these have direct analogues which can occur at the trough.
If there is a saddle point in the flow above the crest at $x=0$, then this may collide with the interface as the amplitude increases, resulting in a Type I solution, like that shown in Figure~\ref{fig:coefs}(B). We show the process of the stagnation moving towards the interface in Figure~\ref{fig:saddleSplitting}(A)--(C), with the solution in panel (C) being close to a Type I solution.
Alternatively, the crest may approach the upper wall if there are no stagnation points between.
This lack of stagnation occurs in one of two ways: either for all solutions along the bifurcation curve, there is no stagnation above the crest, or there is initially a saddle above the crest, but as the amplitude increases, the saddle point collides with the upper wall.
After this collision, it will split into two saddle points, which move along the upper wall away from $x=0$, leaving no stagnation points at $x=0$ in the upper layer. 
This results in a Type II solution, and typical solutions along such a branch are shown in  Figure~\ref{fig:saddleSplitting}(D)--(F).
\begin{figure}
    \subfloat[]{
        \includegraphics[width=.45\textwidth,trim={0 15 0 40},clip]{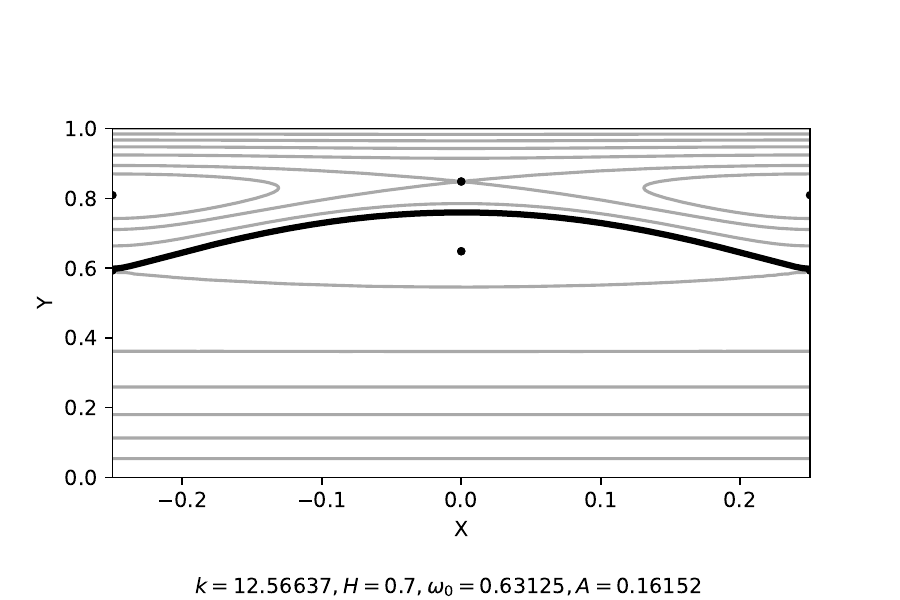}
        \label{fig:deep:H07}} \hfill
    \subfloat[]{
        \includegraphics[width=0.45\textwidth,trim={0 15 0 40},clip]{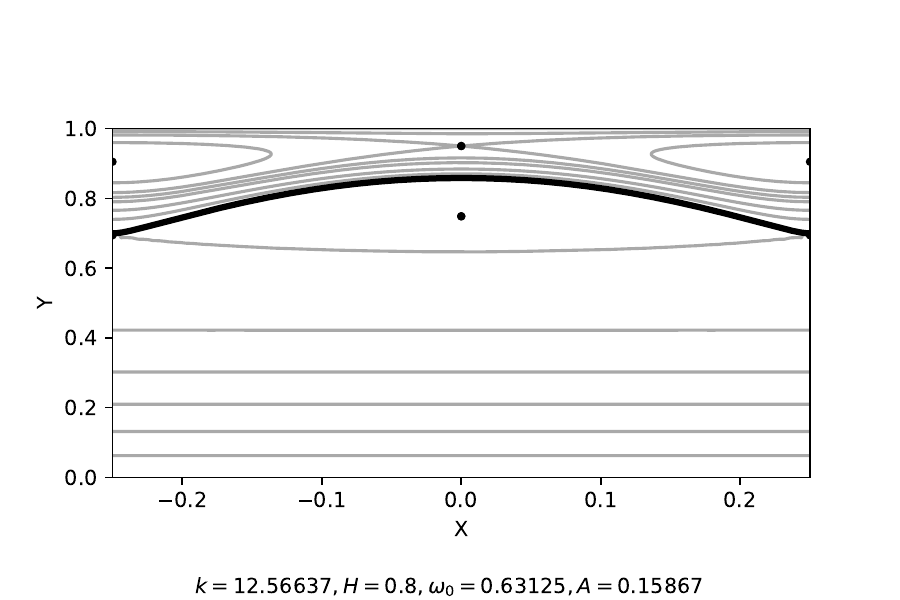}    
        \label{fig:deep:H08}}\\
    \subfloat[]{
         \includegraphics[width=0.45\textwidth,trim={0 0cm 0 1cm},clip]{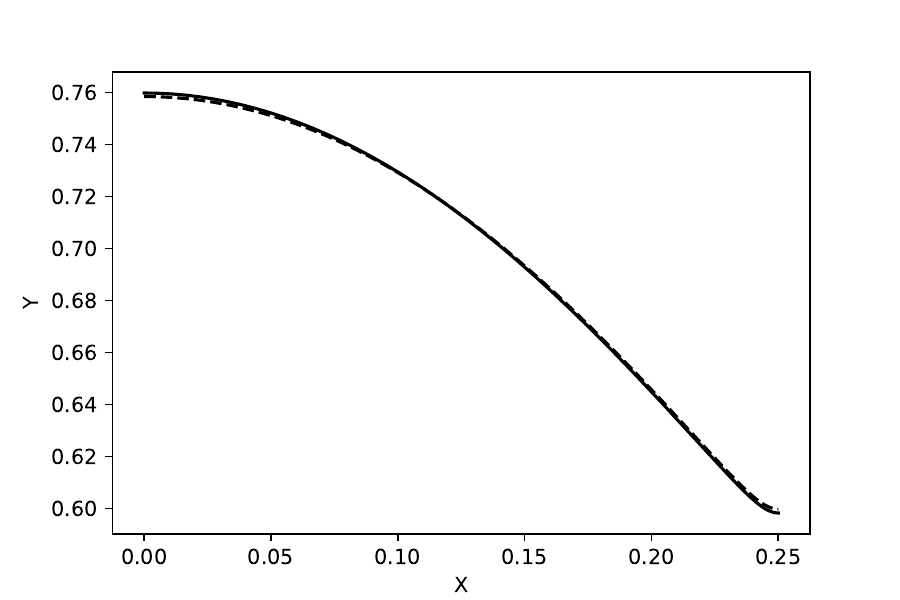}
        \label{fig:deep:compare}}
        \caption{ Two solutions with $k=4\pi$, $\omega_0=0.63125$. In panel~(A) we have $H=0.7$, and in panel~(B) $H=0.8$. In panel~(C),  we see that the interface from (A) (solid line) agrees well with the interface from (B) translated down by 0.1 (dashed). The velocity along the interface for each solution agrees similarly well.}
        \label{fig:deep}
\end{figure}

There is however a caveat. As the crest becomes very close to the upper wall, it can become infeasible to resolve the limiting behaviour at the trough. 
For example, if there is a saddle a small distance below the trough, as the amplitude increases, this distance may be bounded below, meaning we have a Type II solution, or this distance may become arbitrarily small, meaning we have a Type I solution that at first appeared to be a Type II solution.
Alternatively, if the crest is approaching the upper wall and there is no saddle below the trough, it is unclear whether the trough will approach the lower wall or tend to some other limit.
Figure~\ref{fig:unclear} shows solutions along a branch where this difficulty in distinguishing the limiting behaviour is observed.

Additionally, for certain critical parameter values, more than one termination criterion may occur at once.
This can manifest as multiple criteria occurring at the same place, for example a corner approaching the wall as it forms, or criteria occurring at different places, for example, a corner forming at the crest as the trough approaches the wall.
In Figure~\ref{fig:paramspace}, we present a sweep of parameter space, showing whether the solution branch with the given parameters limits to a Type I or Type II solution.
For $k$ less than some critical value around 5, Type II solutions are found in one connected region, with Type I solutions to either side. This behaviour is observed  Figure~\ref{fig:paramspace}(A)--(C).
Solutions on the left-hand boundary of this region have a corner at their crest and touch one of the walls.
If $H$ is less than some critical value then the interface will touch the lower wall, if it is greater than this value, the corner will touch the upper wall.
Solutions on the right-hand boundary are ``reflections" of solutions on the left-hand boundary, in the sense described at the start of this section. 
As $|\omega_0|\to\infty$, the amplitude of the corresponding Type I solutions become smaller, we believe unboundedly so.

As the value of $k$ increases, the region of Type II solutions becomes ``pinched" around $H=0.5$, eventually splitting into two regions. This can be seen in Figure~\ref{fig:paramspace}(D).
Solutions on the boundary between the regions of upward pointing corners and downwards pointing corners have corners at the peak and trough. 

In the short wavelength limit ($k \to \infty$), we find that $\omega_0$ generally has a far greater impact than $H$ on the behaviour of Type I solutions. This is because Type I solutions do not ``see" the wall, unless they are particularly close to it, so varying $H$ in such a case simply translates the interface of the solution vertically, and leaves $u(t)$, $v(t)$ and $X(t)$ largely unaffected.
In other words, if $k$ is large, then for a Type I solution which is not close to a wall, only $\hat Y_0$ is significantly affected by small variations in $H$ (see Figure~\ref{fig:deep}).

\subsection*{Acknowledgements}
AD would like to acknowledge funding from EPSRC NFFDy Fellowships (EPSRC grants EP/X028607/1).
JS received support through EPSRC, EP/T518013/1.

\appendix
\section{Proof of Theorem~\ref{thm:buffoniToland}}\label{sec:proofOfGlobBif}
\begin{proof}
    This very closely follows the argument in \cite{Buffoni--Toland:Book}, we focus on the small differences.
    Firstly, by the real-analytic Crandall--Rabinowitz theorem \cite[Theorem~8.3.1]{Buffoni--Toland:Book},  hypotheses~\ref{largeamphypothesis1}--\ref{largeamphypothesis3} yield the existence of a local curve of solutions, i.e., functions $Z_{\mathrm{loc}}, Q_{\mathrm{loc}}$ with domain $(-\eps, \eps)$ such that $(Z_{\mathrm{loc}}(0), Q_{\mathrm{loc}}(0))=(\zeta_*,q_*)$, and  $\mathcal{F}(Z_{\mathrm{loc}}(\tau), Q_{\mathrm{loc}}(\tau)) \equiv 0$.
    We seek to extend this to a global curve of solutions by joining \emph{distinguished arcs} together.
    Given a function $\mathcal{G}$ with domain $\tilde{ \mathcal{U}}$, we let
    \[\mathcal{A} = \{(\zeta,q) \in \tilde{ \mathcal{U}} \mid \mathcal{G}(\zeta,q)=0, \ D_\zeta \mathcal{G} (\zeta,q) \text{ invertible }   \},\]
    and define a distinguished arc of $\mathcal{G}$  to be a maximal connected subset of $\mathcal{A}$.

    If $Q_{\mathrm{loc}}'$ is identically 0, it is possible that no non-trivial solutions lie in a distinguished arc.
    To avoid this, we investigate the distinguished arcs of a new function $\mathcal{G}$. To define $\mathcal{G}$, we introduce a bounded linear functional $l \colon \mathcal X \to \R$.
    If $Q_{\mathrm{loc}}'(0) \neq 0$, let $l \equiv 0$.
    If, however, $Q_{\mathrm{loc}}'(0)= 0$, let $l$ satisfy $l(\dot{\rho}_*)=1$, where $\dot{\rho}_* = \dot{\rho}_1$, and $\dot{\rho}_1$ is defined in \eqref{eqn:kernelAtShear}.
    Let $\mathcal{G}(\zeta, q)=\mathcal{F}(\zeta, q + l(\zeta-\zeta_*))$.
    Letting $\tilde{\mathcal{U}} = \{ (\zeta, q - l(\zeta-\zeta_*)) \mid (\zeta, q) \in \mathcal{U} \}$, we see $\tilde {\mathcal{U}}$ is open, and $\mathcal{G} \colon \tilde{\mathcal{U}} \to \mathcal{Y}$ is $\R$-analytic.

    Notice that
    \[D_\zeta\mathcal{G}(\zeta,q) \dot \zeta =D_\zeta\mathcal{F}(\zeta,q+l(\zeta-\zeta_*)) \dot \zeta +D_q \mathcal{F}(\zeta,q+l(\zeta-\zeta_*)) l( \dot \zeta).\]
    Furthermore, for any $q$, the derivative $D_q\mathcal{F}(\zeta_*,q)$ is the zero operator, and so $\mathcal{G}$ satisfies hypotheses~\ref{largeamphypothesis1}--\ref{largeamphypothesis3} with the same values of $\zeta_*$, $q_*$, and $\dot {\rho}_*$. 
    Defining $\tilde{Z}_{\mathrm{loc}}(\tau) =Z_{\mathrm{loc}}(\tau)$ and $\tilde{Q}_{\mathrm{loc}}(\tau) =Q_{\mathrm{loc}}(\tau) - l(Z_{\mathrm{loc}}(\tau)-\zeta_*)$, we see that $\mathcal{G}(\tilde Z_{\mathrm{loc}}(\tau),\tilde Q_{\mathrm{loc}}(\tau))$ is identically 0.
    We apply \cite[Proposition~8.3.4(b)]{Buffoni--Toland:Book} to $\mathcal{G}$. 
    We know from its definition that $\tilde Q_{\mathrm{loc}}'(0) \neq 0$, thus for all sufficiently small $\tau>0$, $(\tilde Z_{\mathrm{loc}}(\tau),\tilde Q_{\mathrm{loc}}(\tau))$ lies on a distinguished arc of $\mathcal{G}$.
    Call this arc $\mathcal{A}_0$.

    By the implicit function theorem, see for example \cite[Theorem~4.5.3]{Buffoni--Toland:Book}, on any distinguished arc, $\zeta$ can be written as a function of $q$.
    We can thus reparametrise such that we have a bijection $(\tilde Z, \tilde Q) \colon (0,1) \to  \mathcal{A}_0$, with $(\tilde Z(\tau), \tilde Q(\tau))=(\tilde Z_{\mathrm{loc}}(\tau), \tilde Q_{\mathrm{loc}}(\tau))$ for all $0<\tau<\eps$.
    By the definition of a distinguished arc, we see $\mathcal{G}(\tilde Z(\tau), \tilde Q(\tau)) =0$.
    As $\tau \nearrow 1$, if alternative~\ref{alternative:blowup} occurs, we are done.
    If not, we show that $\mathcal{A}_0$ connects to another distinguished arc $\mathcal{A}_1$, or that $\mathcal{A}_0 \cup \{(\zeta_*,q_*)\} $ forms a loop.

    If alternative~\ref{alternative:blowup} does not occur, then there exists a sequence of $\tau_n \nearrow 1$ such that $(\tilde Z(\tau_n), \tilde Q(\tau_n))$ remains in a compact subset of $\tilde{\mathcal{U}}$.
    Passing to a convergent subsequence,  $(\tilde Z(\tau_n), \tilde Q(\tau_n)) \to (\zeta_1,q_1) \in \tilde {\mathcal{U}}$.
    We know $\mathcal{G}(\zeta_1,q_1)=0$ by continuity, and so hypothesis~\ref{largeamphypothesis2} gives us that $D_\zeta \mathcal{G}(\zeta_1,q_1)$ is Fredholm.
    The Fredholm index is locally constant with respect to the operator norm, and $D_\zeta \mathcal{G}(\tilde Z(\tau_n), \tilde Q(\tau_n)) \to D_\zeta \mathcal{G}(\zeta_1,q_1)$ in this norm. Thus, $D_\zeta \mathcal{G}(\zeta_1,q_1)$ is Fredholm of index 0.

    We now proceed exactly as in the proof of \cite[Theorem~9.1.1]{Buffoni--Toland:Book}, and conclude we have functions $\tilde Z, \tilde Q$  satisfying the conclusion of the theorem for $\mathcal{G}$ and $\tilde{\mathcal{U}}$. An inverse of an earlier transformation gives functions $Z$, $Q$
    satisfying the conclusion of the theorem for $\mathcal{F}$ and $\mathcal{U}$.
\end{proof}

\section{Derivation of the numerical formulation}\label{sec:numericsAlg}

We can extend to larger domains by reflecting $\Omega_0$ over the line $y=0$, and reflecting $\Omega_1$ over the line $y=1$.
Let
\begin{align*}
    \eta_0(t)&=(X(t),-Y(t))\\
    \eta_1(t)&=(X(t),2-Y(t))
\end{align*}
and let $\tilde{\Omega}_j$ be the region bounded by $x=-\pi/k$, $x=\pi/k$, $\eta$ and $\eta_j$.
We extend the domain of $f_0$ to $\tilde{\Omega}_0$ by noting that since $v_0=0$ on $y=0$, if $y<0$ we can define $f_0(z)=\overline{f_0}(\bar z )$.
Similarly, we extend the domain of $f_1$ to $\tilde{\Omega}_1$ by noting that since $v_1=0$ on $y=1$, if $y>1  $ we can define $f_1(z)=\overline{f_1}(2i+\bar z )$.

Therefore, by definition, we have $f_0(\bar{z})=\overline{f_0(z)}$. We also have from the evenness of $\Psi$ with respect to $x$ that $f_0(\bar z)=\overline{f_0(- z)}$. Therefore, $f_0(-z)=f_0(z)$.
We now integrate
\[\frac{f_0(z)}{(e^{ikz}-e^{ikw})(e^{ikz}-e^{-ik w})}\] around $\partial \tilde{\Omega}_0$, and apply the residue theorem for poles on the boundary. We see that for $w$ in the interface,
\begin{align*}
    -\frac{\pi}{k}f_0(w)&=\int_{\partial \tilde{\Omega}_0} \frac{f_0(z)}{(e^{ikz}-e^{ikw})(e^{ikz}-e^{-ik w})} \ dz\\
    &=\int_{\eta_0} \frac{f_0(z)}{(e^{ikz}-e^{ikw})(e^{ikz}-e^{-ik w})} \ dz - \int_{\eta} \frac{f_0(z)}{(e^{ikz}-e^{ikw})(e^{ikz}-e^{-ik w})} \ dz,
\end{align*}
where we integrate anticlockwise around $\partial \tilde{\Omega}_0$ in the first line, and along both $\eta$ and $\eta_0$ from left to right in the second line. The contributions from the left and right boundaries do not appear as they cancel by periodicity of the integrand. Now using evenness of $f_0$, we see
\begin{align*}
    -\frac{\pi}{k}f_0(w)&=\int_{\eta} \frac{f_0(z)}{(e^{-ikz}-e^{ikw})(e^{-ikz}-e^{-ik w})} \ dz - \int_{\eta} \frac{f_0(z)}{(e^{ikz}-e^{ikw})(e^{ikz}-e^{-ik w})} \ dz\\
    &=\int_{\eta} f_0(z) \left(\frac{1}{(e^{-ikz}-e^{ikw})(e^{-ikz}-e^{-ik w})} - \frac{1}{(e^{ikz}-e^{ikw})(e^{ikz}-e^{-ik w})} \right) \ dz\\
    &=\int_{\eta} f_0(z) g(z,w) \ dz.
\end{align*}
 Some elementary manipulation gives us
\begin{align*}
    -\frac{\pi }{k} f_0(w)&=\int_{-\pi} ^{\pi} f_0(X(t)+iY(t)) g(X(t)+iY(t),w)(X'(t)+iY'(t)) \ ds\\
    &=\int_{0} ^{\pi} f_0(X(t)+iY(t)) g(X(t)+iY(t),w)(X'(t)+iY'(t)) \\
    &\qquad + \overline{f_0(X(t)+iY(t))} g(-X(t)+iY(t),w)\overline{(X'(t)+iY'(t))} \ ds\\
\end{align*}
which straightforwardly yields \eqref{eqn:lowerHolomorphic}.

Similarly integrating
\[\frac{f_1(z)}{(e^{ikz}-e^{ik(i+w)})(e^{ikz}-e^{ik(i-w)})}\]
around $\partial \tilde{\Omega}_1$ gives us
\begin{equation}\label{eqn:f1}
     \begin{aligned} f_1(w) &= -\frac{ k }{\pi}\int_0^{\pi} f_1(X(t)+iY(t)) g(i-X(t)-iY(t),w-i)(X'(t)+iY'(t))\\
    &\qquad \qquad + \overline{f_1(X(t)+iY(t))}g(i+X(t)-iY(t),w-i ) \overline{(X'(t)+iY'(t))} \ ds
\end{aligned}  
\end{equation}
By continuity of $\Psi_x$ and $\Psi_y$ at the interface,
\begin{equation}\label{eqn:f0f1continuity}
    f_1(X(t),Y(t))=f_0(X(t),Y(t))+(Y(t)-H).
\end{equation}
Therefore inserting \eqref{eqn:f0f1continuity} into \eqref{eqn:f1} to eliminate $f_1$ yields
\eqref{eqn:upperHolomorphic}.

\bibliographystyle{plain}
\bibliography{refs}

\end{document}